\documentclass[10pt,oneside]{article}
\usepackage[T1]{fontenc}
\usepackage[english]{babel}

\usepackage{graphicx,accents}

\usepackage[cal=cm]{mathalfa}

\usepackage[T1]{fontenc}
\usepackage{titlesec}

\usepackage{hhline}

\titleformat
{\chapter} 
[display] 
{\bfseries\Large\itshape} 
{Story No. \ \thechapter} 
{0.5ex} 
{
    \rule{\textwidth}{1pt}
    \vspace{1ex}
    \centering
} 
[
\vspace{-0.5ex}%
\rule{\textwidth}{0.3pt}
] 

\titleformat{\section}[wrap]
{\normalfont\bfseries}
{\thesection.}{0.5em}{}

\titlespacing{\section}{12pc}{1.5ex plus .1ex minus .2ex}{1pc}

\usepackage{tocloft} 

\usepackage{titlesec}
\titleformat{\subsection}[runin]
       {\normalfont\bfseries}
       {\thesubsection}
       {0.5em}
       {}
       [.]

\makeatletter
\@ifundefined{dddot}{}{}
\usepackage{mathdots}  
\makeatother

\makeatletter
\@ifundefined{ddddot}{}{}
\usepackage{yhmath}
\makeatother

\usepackage{geometry} 
\usepackage{xcolor} 
\usepackage{tikz} 
\usepackage[draft=false]{hyperref} 
\usepackage{mathtools}
\usepackage{amsmath,amsthm,amssymb}
\usepackage[all,cmtip]{xy}

\usepackage{fullpage}

\usepackage[capitalise]{cleveref}  
\newcommand{\clevertheorem}[3]{%
	\newtheorem{#1}[thm]{#2}
	\crefname{#1}{#2}{#3}
}
\numberwithin{equation}{section} 
\numberwithin{figure}{section} 

\theoremstyle{plain} 
\newtheorem{thm}{Theorem}[subsection]
\crefname{thm}{Theorem}{Theorems}
\newtheorem*{thm*}{Theorem}
\clevertheorem{proposition}{Proposition}{Propositions}
\clevertheorem{prop}{Proposition}{Propositions}
\newtheorem*{prop*}{Proposition}
\clevertheorem{theorem}{Theorem}{Theorems}
\clevertheorem{lemma}{Lemma}{Lemmas}
\clevertheorem{corollary}{Corollary}{Corollaries}
\clevertheorem{cor}{Corollary}{Corollaries}
\clevertheorem{conj}{Conjecture}{Conjectures}
\clevertheorem{questn}{Question}{Questions}

\theoremstyle{definition} 
\clevertheorem{definition}{Definition}{Definitions}
\clevertheorem{assumption}{Assumption}{Assumptions}
\clevertheorem{defn}{Definition}{Definitions}
\clevertheorem{notn}{Notation}{Notations}
\clevertheorem{conv}{Convention}{Conventions}

\clevertheorem{remark}{Remark}{Remarks}
\clevertheorem{rmk}{Remark}{Remarks}
\clevertheorem{warn}{Warning}{Warnings}
\clevertheorem{example}{Example}{Examples}
\clevertheorem{summ}{Summary}{Summaries}

\usepackage{amsmath,amsthm,amssymb,amsfonts,extarrows}
\usepackage{enumerate,amscd,amsxtra,MnSymbol}
\usepackage{mathrsfs}
\usepackage{color}
\usepackage[cmtip,all]{xy}
\usepackage{eucal}
\usepackage{bbold}
\usepackage{upgreek}
\usepackage{paralist}
\usepackage{tikz-cd}
\usepackage{dsfont}
\usepackage{bbm}
\usepackage[bbgreekl]{mathbbol}

\DeclareMathSymbol\bbDelta \mathord{bbold}{"01}
\DeclareMathSymbol\bDelta \mathord{bbold}{"01}

\newtheorem{remark*}{Remark}

\newtheorem{construction}[thm]{Construction}

\newtheorem{notation}[thm]{Notation}

\newtheorem{conjecture}[thm]{Conjecture}

\newcommand{\bD}{{\mathbb D}}

\newcommand{\bN}{{\mathbb N}}

\newcommand{\mA}{{\mathcal A}}
\newcommand{\mB}{{\mathcal B}}
\newcommand{\mC}{{\mathcal C}}
\newcommand{\mD}{{\mathcal D}}
\newcommand{\mE}{{\mathcal E}}
\newcommand{\mF}{{\mathcal F}}

\newcommand{\mO}{{\mathcal O}}
\newcommand{\mP}{{\mathcal P}}

\newcommand{\mS}{{\mathcal S}}

\newcommand{\mV}{{\mathcal V}}
\newcommand{\mW}{{\mathcal W}}

\newcommand{\A}{A}
\newcommand{\B}{B}
\newcommand{\C}{C}

\newcommand{\F}{{F}}
\newcommand{\G}{{G}}

\renewcommand{\L}{{\mathrm L}}

\newcommand{\N}{{\mathrm N}}

\newcommand{\R}{{\mathrm R}}

\newcommand{\T}{{T}}

\newcommand{\X}{X}
\newcommand{\Y}{Y}
\newcommand{\Z}{Z}

\newcommand{\bj}{{j}}
\newcommand{\bi}{{i}}
\newcommand{\m}{{m}}
\newcommand{\bk}{{k}}

\newcommand{\p}{{p}}
\newcommand{\q}{{q}}

\newcommand{\n}{{n}}
\newcommand{\br}{{r}}

\newcommand{\op}{\mathrm{op}}

\newcommand{\Grp}{\mathrm{Grp}}

\newcommand{\colim}{\mathrm{colim}}

\newcommand{\rev}{{\mathrm{rev}}}

\newcommand{\ot}{\otimes}

\newcommand{\co}{\mathrm{co}}

\newcommand{\univ}{\mathrm{univ}}
\newcommand{\strict}{\mathrm{strict}}

\newcommand{\Gpd}{\mathrm{Gpd}}

\newcommand{\id}{\mathrm{id}}
\newcommand{\Cat}{\mathrm{Cat}}

\newcommand{\Set}{\mathrm{Set}}

\newcommand{\Alg}{\mathrm{Alg}}

\newcommand{\Fun}{\mathrm{Fun}}

\newcommand{\lax}{{\mathrm{lax}}}
\newcommand{\oplax}{{\mathrm{oplax}}}

\newcommand{\tu}{{\mathbb 1}}

\newcommand{\ev}{{\mathrm{ev}}}

\newcommand{\Map}{{\mathrm{Map}}}

\newcommand{\bZ}{{\mathbb{Z}}}

\newcommand{\Mor}{{\mathrm{Mor}}}

\newcommand{\PrL}{\mathrm{Pr^L}}
\newcommand{\PrR}{\mathrm{Pr^R}}

\newcommand{\gaunt}{\mathrm{gaunt}}
\newcommand{\Steiner}{\mathrm{Steiner}}

\newcommand{\cube}{{\,\vline\negmedspace\square}}

\newcommand{\scat}{\mathcal{C}\mathit{at}}

\newcommand{\vertrule}[1][1ex]{\rule{.4pt}{#1}}
\newcommand{\bd}{{\,\,\vertrule{}\!\!\Diamond}}
\newcommand{\abd}{{\,\,\vertrule{}\!\!\bar{\Diamond}}}

\begin{document}

\title{\textsc{An Oriented Street--Roberts Conjecture}}

\author{David Gepner and Hadrian Heine}

\maketitle

\begin{abstract}
We formulate a notion of oriented polytope, including Street's oriented simplices and Gray's oriented cubes, and use this to prove an oriented version of the Street--Roberts conjecture, presenting $(\infty,\infty)$-categories as sheaves on suitable families of oriented polytopes, generalizing work of Campion.
This allows us to understand $(\infty, \infty)$-categories from a geometric perspective, as directed analogues of homotopy types.
These familes of oriented polytopes induce basic operations in higher category theory: for instance, the join, Gray tensor, and bicone arise from the geometry of the orientals, cubes, and orthoplexes, respectively.
We study the interaction of these operations and derive some geometric formulae, generalizing work of Ara--Maltsiniotis, Verity, and others.
\end{abstract}

\tableofcontents

\section{Introduction}

\subsection{Higher dimensional categories}
The purpose of this paper is to provide a geometric perspective on the theory of higher dimensional categories.
Our main result is a presentation of $\infty$-categories,\footnote{We henceforth refer to $(\infty,n)$-categories simply as $n$-categories, for all $0\leq n\leq\infty$.} as sheaves on categories of oriented polytopes, reflecting the inherent geometry of higher dimensional category theory.

The notion of dimension plays as crucial a role in higher dimensional category theory as it does in geometry, and these two notions of dimension tend to agree whenever it makes sense to compare them.
Sets are $0$-dimensional geometric objects, as are $\infty$-groupoids, since colimits of $0$-dimensional objects remain $0$-dimensional.
Nevertheless, higher dimensional categories can be built out of lower dimensional categories, provided one uses a version of the colimit which is compatible with categorical dimension.\footnote{In a subsequent paper we develop a theory of {\em oriented} colimits, generalizing the notion of lax colimits and the Gray tensor.}

In general, an $\infty$-category is $n$-dimensional if it is an $n$-category but not an $(n-1)$-category.\footnote{There has been a great deal on $n$-category theory, and also for $n=\infty$, by a number of authors, including  \cite{ara2014higher}, \cite{Ara2020ComparisonOT}, \cite{AraLafontMetayerOrient}, \cite{ayala2018flagged}, \cite{Barwicknrel}, \cite{barwickunicity}, \cite{batanin1998monoidal}, \cite{BERGER2002118}, \cite{campion2022cubesdenseinftyinftycategories}, \cite{campion2025cubical}, \cite{cheng2004higher}, \cite{cheng2019weak}, \cite{dean2024computads}, \cite{goldthorpe2023homotopy}, \cite{Harpaz2020}, \cite{kapranov1991combinatorial}, \cite{loubaton2024complicialmodelinftyomegacategories}, \cite{maehara2023orientals}, \cite{moser2022model}, \cite{MOSER2024107620},  \cite{ozornova2023quillen},   \cite{simpson2011homotopy}, \cite{stefanich2020presentable}, \cite{STREET1987283}, \cite{VERITY2}.}
The most basic example of a $0$-dimensional category, or an $\infty$-groupoid, is the $0$-disk $\bD^0$, the terminal $\infty$-category, playing the role of the point.
Higher dimensional disks are obtained by suspending lower dimensional disks:
for instance, the $1$-disk $\bD^1=\{0\to 1\}$ is the basic $1$-dimensional category, consisting of single nonidentity morphism between a source and target object.
If $\mD$ is an $n$-dimensional category, its {\em suspension} is the $(n+1)$-dimensional category $S(\mD)$ freely generated by a $\mD$-indexed family of $1$-cells with fixed source and target.
More precisely, $S(\mD)$ contains a pair $\{0,1\}$ of objects with morphism object $\Mor_{S(\mD)}(0,1)=\mD$, and no other nondegenerate morphisms (in particular, $\Mor_{S(\mD)}(1,0)=\emptyset$).

The $n$-dimensional disk $\bD^n$ is the $n$-fold suspension of $\bD^0$, and is therefore an $n$-category with a unique ordered pair of $m$-cells in all dimensions $m<n$ and a single nondegenerate $n$-cell.
The $n$-disk $\bD^n$ is the most basic $n$-dimensional category, as it is free on a single $n$-cell, and it turns out that any $n$-category can be built out of $\bD^n$ via iterated colimits.
An $n$-dimensional cell of an $\infty$-category is specified by a map from the $n$-disk, a restatement of the fact that the $n$-disk is the free $n$-cell.
In practice, it is essential to be able to perform the suspension operation. Hence we must work on the level of $\infty$-categories,
defined as the limit
\[
\infty\Cat=\lim\{\cdots\to (n+1)\Cat\to n\Cat\to (n-1)\Cat\to\cdots\to 0\Cat\}
\]
taken along the core functors, which forget the highest dimensional morphisms
\footnote{One could also take the limit along the truncation functors, which invert the top-dimensional morphisms. This results in the full subcategory of $\infty\Cat$ consisting of the Postnikov complete $\infty$-categories \cite{gepner2026homotopy}.
This full subcategory is not appropriate for some purposes since the infinite cobordism $\infty$-category and its classifying space have equivalent Postnikov completions.}.

A main subtlety of higher category theory is that dimension shifting operations, such as the suspension, fail to be functors of $\infty$-categories, which are functors enriched in the cartesian product of $\infty$-categories, but only are functors enriched in the Gray tensor product of $\infty$-categories \cite{gepner2025oriented}, a refined tensor product of $\infty$-categories that laxifies the cartesian product, which appears as a localization of the Gray tensor product.
This failure comes from the fact that the cartesian product is not compatible with categorical dimension: the product of an $n$-category and an $m$-category is an $\max\{n,m\}$-category.
On the other hand, the Gray tensor product behaves additive with respect to categorical dimension and so supports higher-categorical constructions: the Gray tensor product of an $n$-category and an $m$-category is an $n+m$-category.

A consequence of categorical additivity is that the Gray tensor product fails to be symmetric and happens to be antisymmetric in the following sense: 
for any two $\infty$-categories $X,Y$ there is a canonical equivalence 
\[
X\bar{\boxtimes} Y:=Y\boxtimes X\simeq (X^\co\boxtimes Y^\co)^\co
\]
where the ``co'' involution reverses all even-dimensional cells.

The Gray tensor product is significant for higher category theory since it encodes notions of lax and oplax natural transformations: 
like the cartesian product, the Gray tensor product preserves colimits in each variable and therefore admits right adjoints 
given by the $\infty$-categories of functors and (op)lax natural transformations.
More precisely, if $X$, $Y, Z$ are $\infty$-categories, there are natural equivalences of $\infty$-groupoids
\[
\Map_{\infty\Cat}(Y,\Fun^{\lax}(X,Z))\simeq\Map_{\infty\Cat}(X\boxtimes Y,Z)\simeq\Map_{\infty\Cat}(X,\Fun^{\oplax}(Y,Z)).
\]

\subsection{The geometry of $\infty$-categories}
In this paper we do not generally restrict to univalent $\infty$-categories, the full subcategory consisting of those $\infty$-categories in which every equivalence of any dimension is equivalent to an identity.
These are also called complete $\infty$-categories, in the sense of Rezk \cite{rezk2001model}, and arise from $\infty$-categories by forcing the fully faithful and essentially surjective functors to be equivalences.
At the other extreme, we have the strict $\infty$-categories, namely those $\infty$-categories whose underlying $n$-categories are strict in all dimensions $n$.
By design, the mapping categories in a strict $\infty$-category are again strict, and in particular the mapping spaces in a strict category are sets.

The intersection of the strict and the univalent $\infty$-categories the class of {\em gaunt} $\infty$-categories, the pullback
\[
\xymatrix{
& \infty\Cat^\gaunt\ar[ld]\ar[rd] &\\
\infty\Cat^\strict\ar[rd] & & \infty\Cat^\univ\ar[ld]\\
& \infty\Cat & .}
\]
By definition, an $\infty$-category is gaunt if it has no nontrivial automorphisms in any dimension.
It has been shown \cite{barwickunicity} that gaunt $\infty$-categories generate all $\infty$-categories under colimits; however, this is still a much too large and complicated class.
This motivates one to look for easier classes of $\infty$-categories, which are freely generated by a higher dimensional graph.

The $n$-disk generates $n\Cat$
under small colimits, in the sense that $n\Cat\subset\infty\Cat$ is the smallest full subcategory containing $\bD^n$ which is closed under colimits.
However, this is unsatisfactory, as there is no {\em canonical} way to express a general $n$-category as a colimit of disks.
Instead, we would like to find small full subcategories $\mD\subset\infty\Cat$ such that the restricted Yoneda embedding
\[
\N_\mD:\infty\Cat\to\Fun(\mD^{\op},\mS),
\]
given by $\N_\mD(X)=\Map_{\infty\Cat}(-,X)$ is fully faithful.
Equivalently, since the $\mD$-nerve $\N_\mD$ always admits a left adjoint, the $\mD$-realization functor $||-||_\mD$, we can ask that the counit transformation
\[
||\N_\mD(X)||_\mD\simeq\colim_{D\in\mD_{/X}} D\to X
\]
is an equivalence.
If so, $X$ is canonically the colimit in $\infty\Cat$ of the diagram of the objects of $\mD$ equipped with a map to $X$.

There are several previously known presentations of the category of $\infty$-categories.
The work of Roberts, Street, Verity, and others led to a presentation of $\infty\Cat$ as a localization of marked, or stratified, simplicial sets, or spaces.
While this is a very efficient model for many purposes, it arguably arose from a failure of a more direct presentation directly as a localization of a presheaf category.
Since marked simplicial spaces happens to be a presentable category, this does provide a presentation of $\infty\Cat$, but it is more indirect.

Direct presentations of $\infty\Cat$ were provided by Rezk, as a full subcategory of presheaves on Joyal's category $\Theta$, the smallest full subcategory of $\infty\Cat$ closed under wedges and suspensions and containing the point, and Barwick--Schommer-Pries, as a full subcategory of presheaves on the gaunt $\infty$-categories (one can restrict to the compact gaunt $\infty$-categories, which are necessarily $n$-categories for some $n\geq 0$).
Barwick and Schommer-Pries \cite{barwickunicity} also show that any automorphism of $\infty\Cat$ is a composite of the involutions which reverse the orientations of morphisms in some fixed dimensions $0<n<\infty$.
For instance, reversing all the even-dimensional cells (the ``co''-involution) produces anti-versions of various $\infty$-categories which arise geometrically.
More recently, work of Campion on the oriented cubes \cite{campion2022cubesdenseinftyinftycategories} and the Gray tensor product \cite{campion2023graytensorproductinftyncategories} shows that $\infty$-categories are a localization of presheaves of spaces on the cubes.

Street's {\em oriented simplices} \cite{STREET1987283} come in positive and negative oriented versions, which in dimension two are the $2$-categories
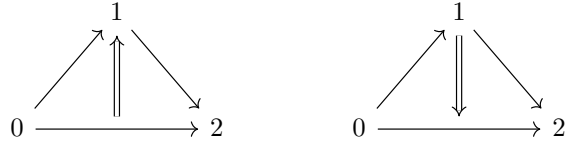
\[\begin{tikzcd}[row sep=3em]
& 1  \arrow{dr} & \\ 
0 \arrow{ur}  \arrow[""{name=foo}]{rr} && 
\arrow[Rightarrow, from=foo, to=1-2,shorten >=0.5ex] 2
\end{tikzcd}\qquad\qquad
\begin{tikzcd}[row sep=3em]
&1  \arrow{dr} \arrow[Leftarrow, from=foo, to=1-2,shorten >=0.5ex] & \\ 
0 \arrow{ur} \arrow[""{name=foo}]{rr} && 
 2
\end{tikzcd}
\]
which are related by reversing the direction of the $2$-cell.
We can systematically work with either the oriented or antioriented version by fixing an orientation of the Gray tensor product.
The tensor of the $1$-disk
$
\bD^1=\{0\to 1\}
$
with itself is the refinement of the cartesian product $\bD^1\times\bD^1$ in which the two maps from the first object $(0,0)$ to the last object $(1,1)$ are not equal, but only comparable, via the $2$-cell.
The standard orientation $\bD^1\boxtimes\bD^1$ and the reverse orientation $\bD^1\bar{\boxtimes}\bD^1$ are the $2$-categories
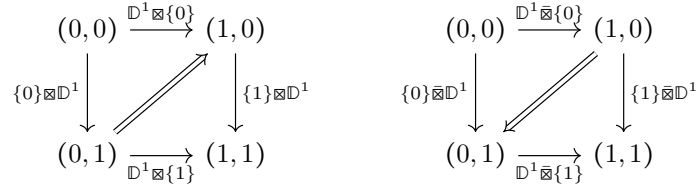
\[
\begin{tikzcd}[row sep=3em]
(0,0) \arrow[d, "\{0\}\boxtimes\bD^1"'] \arrow[r, "\bD^1\boxtimes\{0\}"] & (1,0) \arrow[d, "\{1\}\boxtimes\bD^1"] & \\ 
(0,1) \arrow[Rightarrow]{ur}\arrow[r, "\bD^1\boxtimes\{1\}"'] & (1,1)
\end{tikzcd}\qquad
\begin{tikzcd}[row sep=3em]
(0,0) \arrow[d, "\{0\}\bar{\boxtimes}\bD^1"'] \arrow[r, "\bD^1\bar{\boxtimes}\{0\}"] & (1,0)\arrow[Rightarrow]{ld} \arrow[d, "\{1\}\bar{\boxtimes}\bD^1"] & \\ 
(0,1) \arrow[r, "\bD^1\bar{\boxtimes}\{1\}"'] & (1,1)
\end{tikzcd}
\]
and one is obtained from the other by applying the $(-)^\co$ involution, reversing the even-dimensional cells.

These squares can be viewed as expressing transformations between the vertical arrows.
In the square on the left, which depicts the standard orientation of the Gray tensor product, this is an oplax transformation, and in the square on the right, which depicts the reverse orientation of the Gray tensor product, this is a lax transformation.
The $n$-fold Gray tensor product of $\bD^1$, using the standard and reverse orientation, respectively, define the oriented and antioriented cubes $\cube^n$ and $\bar{\cube}^n$, respectively.
Collectively, these cubes encode the combinatorics of oplax and lax transformations.

\subsection{Oriented polytopes}
Instead of searching for general dense subcategories of $\infty\Cat$, or pursuing the notion of higher test category in full generality, we will restrict to dense subcategories of a specific combinatorial geometric nature.
Specifically, are interested in families of polytopes equipped with orientations, compatible with the orientations of their lower dimensional faces.
For simplicity, we will suppose that any such family consists of a single $n$-dimensional polytope for each natural number $n$.

Polytopes are in particular cell complexes in the classical sense, and therefore have associated cellular chain complexes.
Similarly, oriented polytopes have underlying cell complexes, in which the zero-cells are the vertices, the one-cells are the edges, etc., and moreover these complexes carry a natural positivity structure\footnote{More precisely, they are acyclic directed complexes, and in fact Steiner complexes; see \cite{Steiner2004OmegacategoriesAC} or the appendix for the basics of this remarkable theory.} in which the signs appearing in the differentials are computed by the orientations of the cells.
Hence we may regard oriented polytopes as certain kinds of Steiner $\infty$-categories.
Forgetting these orientations loses essentially all interesting information, since the underlying polytopes, viewed as ordinary topological spaces, are contractible, or acyclic as non-directed cellular chain complexes.
In particular, an oriented space has an underlying unoriented space, which is an $\infty$-groupoid, or anima, or homotopy type, in the usual sense.

For our purposes, a family of oriented polytopes $\mO$ consists of Steiner $\infty$-categories $\mO^n$, $n\in\bN$, satisfying:
\begin{enumerate}[\normalfont(1)]\setlength{\itemsep}{-2pt}
\item For every $n \geq 0$ the $n$-category $\mO^n $ has a unique non-invertible $n$-morphism.
\item For every $n \geq 0$ the canonical functor
$$\underset{\mO^m \rightarrowtail\mO^n \mid m < n}{\colim} \mO^m \to \partial \mO^n:= \iota_{n-1}(\mO^n) $$
is an equivalence.
Here, the colimit is taken along the atomic inclusions.

\item For every $n \geq 0$ the topological cell complex $|\mO^n|$ is the cell structure on an unoriented polytope in which the $m$-cells are the $m$-dimensional faces and the attaching maps are piecewise linear injections. 
In particular, $|\mO^n|$ is a contractible topological space, and $\tau_{\leq 0}\mO^n$ is a contractible $\infty$-groupoid.
\end{enumerate}

Coxeter showed that there are three infinite families of regular polytopes; namely, the simplices, the cubes, and the orthoplexes.\footnote{The orthoplex was originally called the crossed polytope, until it was renamed by Conway.}
These infinite families lift to the setting of higher categories.
For instance, the oriented simplex $\bDelta^n$ is an oriented polytope, and in fact the category of oriented simplices $\bDelta$ is a dense family of oriented polytopes.
The {\em antioriented simplices} are another dense family.
The oriented cubes have already been constructed as a full subcategory of $\infty\Cat$, but the  oriented orthoplexes appear to be new.
Whereas the simplices are self-dual as polytopes, the cubes and the orthoplexes are dual to one another.

The oriented orthoplexes are more difficult to construct, possibly in part because we do not know of an associative binary operation encoded by the oriented orthoplexes.
However, orthoplex geometry does suggest a unary operation called the oriented bicone, which glues an antioriented cone onto an oriented cone along their common intersection.
Just like the oriented join and tensor operations, the oriented bicone corepresents a useful operation which we call the bislice.

\subsection{The Street--Roberts conjecture}

The fact that the simplicial nerve functor
$
\N_\Delta:\Cat\to\Fun(\Delta^{\op},\mS)
$
fully faithfully embeds categories into simplicial objects\footnote{Segal actually treated the case of topological categories, which are a model for nonunivalent categories.} goes back at least to Grothendieck \cite{grothendieckpursuing}, \cite{Grothendieck1958-1960} and Segal \cite{MR232393}.
The category of $1$-morphisms in $\mC$, together with all their composition relations, can be encoded as a {\em category object}; i.e., a simplicial object which satifies the {\em Segal condition}, which is to say that $n$-simplices correspond to composible $n$-tuples of $1$-simplices.

To extend this to higher dimensions, it is useful to regard $\Delta\subset\Cat$ as the smallest full subcategory containing the point which is closed under the 1-categorical join operation.
Thus, by analogy, we should instead consider the smallest full subcategory of $\infty\Cat$ containing the point which is closed under the $\infty$-categorical join operation.
This recovers Street's category of {\em oriented simplices} $\bDelta\subset\infty\Cat$ \cite{STREET1987283}.

The Street nerve is the simplicial object
\[
\N_\mathrm{Street}:\infty\Cat\to\Fun(\Delta^{\op},\infty\Gpd)
\]
corepresented by the oriented simplices, but restricted to the {\em atomic} maps between them.
The usual cosimplicial category $\Delta$ sits inside $\bDelta$ as the subcategory of the atomic morphisms, namely those $n$-cells which do not decompose as composites of nondegenerate $n$-cells.

Unfortunately, the Street nerve is not fully faithful.
To correct this defect, Roberts, Street, Verity and others \cite{VERITY2} \cite{VERITY1} introduced and studied complicial sets, which are simplicial sets equipped with a suitable collection of marked simplices, called {\em thin} simplicies, which correspond to those oriented simplices whose top cell is sent to an equivalence.
The fact that complicial sets or spaces form a model for the category of $\infty$-categories is a nontrivial result, originally due to Verity \cite{VERITY2} in the strict setting and Loubaton \cite{loubaton2024complicialmodelinftyomegacategories}, building on work of Ozornova--Rovelli \cite{ozornova2023quillen}, in the weak setting.

While this is a powerful model which is useful for many purposes, it has the usual drawback that homotopy invariant constructions (that is, constructions which are intrinsic to the underlying category of $\infty$-categories) are only obtained after fibrant replacement, a difficult and somewhat inexplitic procedure in general. This is unfortunately the case even for the basic generators, the simplices themselves, which are not actually complicial sets.
Instead, their fibrantly replacements correspond to the Street nerve of the orientals \cite{maehara2023orientals}.
Nevertheless, there has been a great deal of interesting work in $\infty$-category theory carried out from the perspective of complicial sets.

\subsection{Main results}
Grothendieck's philosophy of test categories posits that homotopical categorical structures can be modeled as presheaves on a suitable category of rigid and combinatorial objects, called a test category, and that there is generally significant flexibility in the choice of test category.
We introduce the concept of a dense family of oriented polytopes, and prove that $\infty$-categories can be modeled as presheaves on any dense family of oriented polytopes satisfying certain sheaf conditions reflecting the geometry of the oriented polytopes.
For any dense family of oriented polytopes $\mO=\{\mO^n\}_{n\in\bN}\subset\infty\Cat$, indexed by the natural numbers,
we consider a nerve functor 
$$\N_\mO:\infty\Cat\to\Fun(\mO^{\op},\infty\Gpd)$$
by restricting the Yoneda-embedding.
We prove the following nerve theorem:

\begin{theorem}[\cref{polytopnerv}]\label{thm:orientalsegal}
Let $\mO=\{\mO^n\}_{n\in\bN}\subset\infty\Cat$ be a dense family of oriented polytopes.
The oriented $\mO$-nerve functor $$\N_\mO:\infty\Cat\to\Fun(\mO^{\op},\infty\Gpd)$$
is a fully faithful right adjoint with essential image those presheaves $X:\mO^{\op}\to\infty\Gpd$ which are local for the following families of maps:
\begin{enumerate}[\normalfont(1)]\setlength{\itemsep}{-2pt}
\item The boundary decomposition
$\colim_{\mO^m \rightarrowtail\mO^n \mid m < n} \N_\mO(\mO^m) \to \N_\mO(\partial \mO^\n).$
\item The top cell decomposition
$\N_\mO(\bD^\n) \underset{{\partial\N_\mO(\bD^\n)}}{\coprod} \N_\mO(\partial \mO^\n) \to \N_\mO(\mO^\n).$
\item The globular decomposition
$\N_\mO(\bD^{i_0}) \!\!\!\underset{\N_\mO(\bD^{j_1})}{\coprod} \N_\mO(\bD^{i_1}) \!\!\!\underset{\N_\mO(\bD^{j_2})}{\coprod} \cdots \underset{\N_\mO(\bD^{j_n})}{\coprod} \!\!\!\N_\mO(\bD^{i_n}) \to \N_\mO(\bD^{i_0} \underset{\bD^{j_1}}{\coprod} \bD^{i_1} \underset{\bD^{j_2}}{\coprod} \cdots \underset{\bD^{j_n}}{\coprod} \bD^{i_n}).$
\end{enumerate}
Here, $n$, $i_0,\ldots, i_n$, and $j_1,\ldots, j_n$ are natural numbers, and 
$ \bD^{j_\ell} \rightarrowtail \bD^{i_\ell},\bD^{j_\ell} \rightarrowtail \bD^{i_{\ell-1}}$ are monomorphisms.
\end{theorem}

\cref{thm:orientalsegal} is useful so far we have a good supply of dense families of oriented polytopes.
While there are more numerous families of regular (convex) polytopes in low dimensions, Coxeter showed that there are only three infinite families of regular polytopes: the simplices, which are self-dual, as well as the cubes and the orthoplexes (also known as the crossed complexes), which are dual to one another.
These fundamental geometric objects lift to the world of $\infty$-categories.
Indeed, the existence of the oriented simplices and the oriented cubes was already well-known, and are easily constructed using the machinery of Steiner theory.
The orthoplexes, on the other hand, are significantly more subtle.

\begin{theorem}(\cref{orientaldense}, \cref{cubicaldense}).
The three infinite families of regular polytopes, the simplices, cubes, and orthoplexes, admit canonical lifts to the category of $\infty$-categories.
More precisely, the oriented simplices, the oriented cubes, and the oriented orthoplexes are Steiner $\infty$-categories whose underlying cell structures coincide with standard triangulations of the simplices, the cubes, and the orthoplexes.
Moreover, the oriented simplices and the oriented cubes are dense families of oriented polytopes.
\end{theorem}

A main ingredient in the proof of this theorem is the following fundamental fact relating the oriented simplices and the oriented cubes:

\begin{theorem}\label{theorem:families}(\cref{Thh})
For each natural number $n$, the oriented simplex $\bDelta^n$ is a retract of the oriented cube $\cube^n$.
\end{theorem}

As in ordinary geometry, the $n$-simplex is a deformation retraction of the $n$-cube.
Our proof of \cref{theorem:families} uses a suitable lax version of such a homotopy.

We also show that any object of the category $\Theta$ is a retract of an oriented simplex, and consequently also a retract of an oriented cube.
In  \cite{campion2022cubesdenseinftyinftycategories}, Campion proves that the oriented cubes form a dense subcategory of the category of $\infty$-categories.
We use a similar strategy to show that the oriented simplices in fact form a dense family of oriented polytopes.
Combined with the aforementioned retractions, this implies that the oriented cubes also form a dense family of oriented polytopes.
It seems harder to show that the family of oriented orthoplexes is dense, though this seems plausible.

The density of the oriented simplices is closely related to the Street--Roberts conjecture and the complicial approach to $\infty$-category theory.
The Street nerve is the restriction of the oriented simplicial nerve to the subcategory of oriented simplices and atomic maps thereof, which is equivalent to the usual simplex category $\Delta\subset\bDelta$.
However, the induced functor $\Delta\to\infty\Cat$ (which is not the usual embedding) is not dense, and consequently additional data are required in order the render the Street nerve fully faithful.
The complicial philosophy of Roberts and Street was to equip the sets or spaces with occur as the values of the Street nerve with markings, keeping track of the ``thin'' simplices.
This leads to a category of marked, or (pre)stratified, simplicial spaces, which can be described as presheaves on a category $\Delta^+$
containing $\Delta.$
They conjectured that this was enough to render the Street nerve fully faithful, which was later proved by Verity \cite[Theorem 266]{VERITY2} in the strict setting, and Loubaton \cite{loubaton2024complicialmodelinftyomegacategories} in the weak setting.

From another perspective, however, one might argue that the Street nerve ought to have been defined using the full subcategory of $\infty\Cat$ consisting of the orientals, instead of restricting only to the atomic morphisms.
The density of the oriented simplices might therefore be viewed as an ``unmarked'' and arguably more conceptually satisfying resolution of the Street--Roberts conjecture.
Specializing \cref{thm:orientalsegal} to the dense family of oriented simplices takes the following form, where we also simplify the boundary condition:

\begin{corollary}\label{cor:orientalsegal}
The oriented simplicial nerve functor
$$
\N_{\bDelta}:\infty\Cat\to \Fun(\bDelta^{\op},\infty\Gpd)
$$
is a fully faithful right adjoint embedding with image those presheaves $X:\bDelta^{\op}\to\infty\Gpd$ which are local for the following families of maps:
\begin{enumerate}[\normalfont(1)]\setlength{\itemsep}{-2pt}
\item[\em{(1)}] The boundary decomposition
$\mathrm{coeq}(\underset{0 \leq i <j \leq n}{\coprod}\N_\bDelta(\bbDelta^{n-2})\rightrightarrows \underset{0 \leq i \leq n}{\coprod} \N_\bDelta(\bbDelta^{n-1})) \to \N_\bDelta(\partial\bbDelta^\n).$
\item[\em{(2)}] The top cell decomposition $\N_\bDelta(\bD^\n) \underset{\partial\N_\bDelta(\bD^\n)}{\coprod} \N_\bDelta(\partial\bbDelta^\n) \to \N_\bDelta(\bbDelta^\n).$
\item[\em{(3)}] The globular decomposition
$\N_\bDelta(\bD^{i_0}) \!\!\underset{\N_\bDelta(\bD^{j_1})}{\coprod} \!\!\N_\bDelta(\bD^{i_1}) \underset{\N_\bDelta(\bD^{j_2})}{\coprod} \cdots \underset{\N_\bDelta(\bD^{j_n})}{\coprod} \!\!\N_\bDelta(\bD^{i_n}) \to \N_\bDelta(\bD^{i_0} \underset{\bD^{j_1}}{\coprod} \bD^{i_1} \underset{\bD^{j_2}}{\coprod} \cdots \underset{\bD^{j_n}}{\coprod} \bD^{i_n}).$
\end{enumerate}
Here, $n$, $i_0,\ldots, i_n$, and $j_1,\ldots, j_n$ are natural numbers, and 
$ \bD^{j_\ell} \rightarrowtail \bD^{i_\ell},\bD^{j_\ell} \rightarrowtail \bD^{i_{\ell-1}}$ are monomorphisms.
\end{corollary}

Campion uses the density of the oriented cubes to obtain a model-independent construction of the oplax Gray tensor product of $\infty$-categories \cite{campion2023graytensorproductinftyncategories}.
Specializing \cref{thm:orientalsegal} to the dense family of oriented cubes results in the following analogue:

\begin{corollary}\label{cor:cubicalsegal}
The oriented cubical nerve functor
$$
\N_{\cube}:\infty\Cat\to\Fun(\cube^{\op},\infty\Gpd)
$$
is a fully faithful right adjoint embedding with image those presheaves $X:\cube^{\op}\to\infty\Gpd$ which are local for the following families of maps:
\begin{enumerate}[\normalfont(1)]\setlength{\itemsep}{-2pt}
\item[\em{(1)}] The boundary decomposition
$\mathrm{coeq}(\underset{0 \leq i <j \leq n}{\coprod} 4 \times \N_\cube(\cube^{n-2})\rightrightarrows \underset{0 \leq i \leq n}{\coprod} 2 \times \N_\cube(\cube^{n-1})) \to \N_\cube(\partial\cube^\n).$
\item[\em{(2)}] The top cell decomposition $\N_\cube(\bD^\n) \underset{\partial\N_\cube(\bD^\n)}{\coprod} \N_\cube(\partial\cube^\n) \to \N_\cube(\cube^\n).$
\item[\em{(3)}] The globular decomposition
$\N_\cube(\bD^{i_0}) \!\!\underset{\N_\cube(\bD^{j_1})}{\coprod} \!\!\N_\cube(\bD^{i_1}) \underset{\N_\cube(\bD^{j_2})}{\coprod} \cdots \underset{\N_\cube(\bD^{j_n})}{\coprod} \!\!\N_\cube(\bD^{i_n}) \to \N_\cube(\bD^{i_0} \underset{\bD^{j_1}}{\coprod} \bD^{i_1} \underset{\bD^{j_2}}{\coprod} \cdots \underset{\bD^{j_n}}{\coprod} \bD^{i_n}).$
\end{enumerate}
Here, $n$, $i_0,\ldots, i_n$, and $j_1,\ldots, j_n$ are natural numbers, and 
$ \bD^{j_\ell} \rightarrowtail \bD^{i_\ell},\bD^{j_\ell} \rightarrowtail \bD^{i_{\ell-1}}$ are monomorphisms.
\end{corollary}

We expect that the family of oriented orthoplexes is a also dense family of oriented polytopes, but we reserve a more in-depth treatment of oriented polytopes for future work.

The connection between $\infty$-categories as presheaves on a dense family of oriented polytopes, versus $\infty$-categories as complicial spaces, remains somewhat indirect and mysterious.
For this reason we propose the following conjecture, which if true, it would provide a comparison between these two types of models.

\begin{conjecture}(\cref{localeq2})
The map 
$$(\Delta^n)^t \to \N(\tau_{n-1}(\bDelta^n)) $$ of
stratified simplicial spaces is local with respect to complicial spaces.
\end{conjecture}

The geometry of oriented polytopes give rise to a number of fundamental operations in the category of $\infty$-categories: the wedge-suspension operation reflects the geometry of the thetas, the join the geometry of the orientals, and the Gray tensor the geometry of the cubes.

Using our techniques, we show that the join formula holds in the setting of weak $\infty$-categories:

\begin{theorem}(\cref{locmon3}, \cref{cor:weakvsstrictjoin})
\begin{enumerate}[\normalfont(1)]\setlength{\itemsep}{-2pt}
\item
The convolution monoidal structure on $\mP(\bDelta^\triangleleft)$ of descends along the 
localization $$\mP(\bDelta) \rightleftarrows \infty\Cat.$$
Similarly, the convolution monoidal structure on $\mP(\bDelta^\triangleleft)$
descends along the localization $$ \mP(\bDelta)\otimes\Set \rightleftarrows \infty\Cat^{\mathrm{strict}}. $$
\item
For every pair of $\infty$-categories $X$ and $Y$ there is a pushout square in $\infty\Cat:$
\[
\xymatrix{
X \boxtimes Y \coprod X \boxtimes Y \ar[r]\ar[d] & X\boxtimes \bD^1 \boxtimes Y \ar[d]\\
X \coprod Y \ar[r] & X\star Y.
}
\]
\item
Let $\mC$ be a small $\infty$-category.
The oplax slice, or oplax over category, functor is the right adjoint $$ \infty\Cat_{\mC/ } \to \infty\Cat, \qquad F \mapsto \mC_{//^\oplax F}$$ of the functor $ (-) \star \mC: \infty\Cat \to \infty\Cat_{\mC/ }.$
The lax coslice, or lax under category, functor is the right adjoint $$ \infty\Cat_{\mC/ } \to \infty\Cat,\qquad F \mapsto \mC_{F//^\lax }$$ of the functor $ \mC \star (-): \infty\Cat \to \infty\Cat_{\mC/ }.$
\end{enumerate}
\end{theorem}

\subsection{Relation to other work}
Higher category has been developed extensively in the past few decades.
The first major development was the synthesis of homotopy theory and category theory. See for instance \cite{bergner2007three}, \cite{MR420609}, \cite{joyal2002quasi}, \cite{lurie.HTT}, \cite{rezk2001model}, \cite{simpson1997closed}, \cite{toen2005vers} for some foundational works in this area.
\footnote{Since the category of $1$-categories forms a $2$-category, in order to systematically study $1$-category theory, it is necessary to develop $2$-category and double category theory.
See for instance \cite{abellan2020relative}, \cite{abellan2024infty}, \cite{ayala2024stratified} \cite{gagna2022equivalence}, \cite{gaitsgory2019study}, \cite{gray1965sheaves}, \cite{Haugseng2020OnLT}, \cite{heine2026local}, \cite{loubaton2025squares}, \cite{lurie2009infinity}, \cite{riehl2020infinity}, \cite{moser2022model}, \cite{moser2023model}, \cite{stern20212}.}
The basic theory of $\infty$-categories, the Gray tensor product, fibrations, and limits and colimits have been worked out recently in various models.
Notably, in the strict setting, Verity developed the model of complicial sets \cite{VERITY1}, and Ara-Guetta \cite{ara2025lax} study the lax slice construction.
and its functoriality.
In the homotopical context, Loubaton adapted the approach of Verity to weak $\infty$-categories \cite{loubaton2024complicialmodelinftyomegacategories} and has written extensively about foundations from this perspective \cite{loubaton2024categorical}.

Our approach is influenced by the work of Campion \cite{campion2022cubesdenseinftyinftycategories}, who has used cubical models of $\infty$-categories to implement the Gray tensor product in the homotopical setting.
We expect the theory of oriented categories to have applications both internally to higher category theory, as well as externally, to mathematical physics, representation theory and non-commutative algebra.
For instance Johnson-Freyd-Reutter\cite{johnson2025build} apply the Gray tensor product to
construct Hopf algebras in braided monoidal categories from a suitable retract in a 3-category, which specializes to the classical Tannakian reconstruction of a Hopf algebra from a monoidal category with duals and a fiber functor.

\subsection{Notation and terminology}

We fix a hierarchy of set-theoretic universes whose objects we call small, large, very large, etc.
We call a space (equivalently, $\infty$-groupoid) $X$ small, large, etc. if for any choice of basepoint and natural number $n$ its homotopy sets $\pi_n X$ are small, large, etc.
We call an $\infty$-category small, large, etc. if its maximal subspace and all its mapping spaces are small, large, etc.

We refer to (not necessarily univalent) weak $(\infty,\n)$-categories for $0 \leq \n \leq \infty$ simply as $\n$-categories, and we refer to (not necessarily univalent) weak $(\n,\n)$-categories as $(\n,\n)$-categories.
In particular, we refer to (not necessarily univalent) $(\infty,1)$-categories as 1-categories, or simply categories.
We will sometimes want to work strictly, which can be viewed as a basechange along the colimit-preserving symmetric monoidal functor $\mS\to\Set$.
In this case, we will refer to strict $(\n,\n)$-categories simply as strict $\n$-categories.\footnote{Note that an $(n,n)$-category need not be a strict $(n,n)$-category if $n>2$.
For instance, the fundamental $\infty$-groupoid of the $2$-sphere $S^2\simeq \mathrm{B}^2\Omega^2 S^2$ is an $(\infty,0)$-category which is not strict, nor are its $n$-truncations for any $n>2$.}

\begin{notation}
We will make use of the following notation and terminology when discussing categories, in the sense of categories enriched in the monoidal category of $\infty$-groupoids under the cartesian product.
\begin{enumerate}[\normalfont(1)]\setlength{\itemsep}{-2pt}
\item We write $\mS$ for the category of spaces, by which we mean small $\infty$-groupoids, homotopy types, or anima, and $\Set$ for the category of small sets.
\item We write $\infty\Cat$ for the large category of small $\infty$-categories.

\item We write $\Delta$ for (a skeleton of) the category of finite, non-empty, partially ordered sets and order preserving maps, whose objects we denote by $[\n] = \{0 < ... < \n\}$ for $\n \geq 0$.\footnote{This should not be confused with the category $\bDelta$ of oriented simplices.}
\item We write $\Map_{\mC}(A,B)$ for the space of maps (equivalently, $1$-morphisms) from $A$ to $B$ in $\mC$, for any category $\mC$ containing an ordered pair of objects $(A,B)\in\mC$.

\item We write $\ast$ for the final object and
$\emptyset$ for the initial object in any category.

\item We call a map of spaces $X \to Y$ an embedding if its fibers are empty or contractible, or likewise if the induced map $X \to X \times_Y X$ is an equivalence, or likewise if it is of the form $X \to X \coprod X'$ for spaces $X,X'.$ 

\item We call a morphism $X \to Y$ in any category $\mC$ a monomorphism if for every $Z \in \mC$ the induced map of spaces $\Map_\mC(Z,X) \to \Map_\mC(Z,Y) $ is an embedding. So embeddings of spaces are precisely monomorphisms in $\mS.$ 

\item We call a fully faithful functor $\mC \to \mD$ an embedding generalizing the notation for spaces.

\item We call a functor an inclusion if it is a monomorphism in $\Cat$. A functor $\mC \to \mD$ is an inclusion if and only if it induces an embedding on maximal subspaces and on all mapping spaces.

\item For a diagram $X\to Z\leftarrow Y$ in a category $\mC$ we write $X\underset{Z}{\prod} Y$ or $X\underset{Z}{\times} Y$ for the pullback, and given a diagram $X\leftarrow W\to Y$ in a category $\mC$, we write $X\underset{W}{\coprod} Y$ or $X\underset{Z}{+} Y$ for the pushout.
\item If $\mC$ and $\mD$ are categories and $\mC\to\mD$ is a left adjoint functor with right adjoint $\mD\to\mC$, we often write $\mC\rightleftarrows\mD$ for this adjunction, where the left adjoint is understood to be the functor going from left to right.

\item 
We write $\mC_*$ or $\mC_{\ast/}$ for the category of pointed objects in a category $\mC$, i.e. the full subcategory of $\Fun([1],\mC)$ of arrows in $\mC$ whose source is a final object.

\item We write $\PrL$ and $\PrR$ for the subcategories of $\widehat{\Cat}$ spanned by the presentable categories and the left and right adjoint functors, respectively.
There is a canonical equivalence $\PrL \simeq (\PrR)^\op$
sending left to right adjoints.
\item We write $\otimes$ for the symmetric monoidal structure on $\PrL$ and $\PrR$.
More precisely, by \cite{lurie.higheralgebra}, $\PrL$ carries a closed symmetric monoidal structure such that the subcategory inclusion $\PrL \subset \widehat{\Cat}$ is a lax symmetric monoidal with respect to the the cartesian structure on $\widehat{\Cat}$.
\item We write $\infty\scat$ for the large $\infty$-category of small $\infty$-categories.\footnote{The morphism $\infty$-categories are formed via enrichment in the cartesian monoidal structure.}
\item We write $\Fun(\mD,\mC)$ for the $\infty$-category of functors from an $\infty$-category $\mD$ to an $\infty$-category $\mC$, the value at $\mC$ of the right adjoint to the functor $(-)\times\mC:\infty\Cat\to\infty\Cat$ for the cartesian product.

\item We write $\bD^1$ for the walking arrow, the category with two objects and a unique non-identity arrow.
\item We write $\partial\bD^1$ and $S^0$ for the maximal subspace in $\bD^1$, the set with two elements.

\item We write $\iota_{n}\mC$ for the $n$-category arising from an $\infty$-category $\mC$ by discarding all noninvertible morphisms above dimension $n.$
\item We write $\tau_{n}\mC$ for the $n$-category arising from an $\infty$-category $\mC$ by inverting all morphisms above dimension $n.$
\end{enumerate}
\end{notation}

\subsection*{Acknowledgements}
We thank Ben Antieau, Tim Campion, Rune Haugseng, Fan Huang, Felix Loubaton, Naruki Masuda, Thomas Nikolaus, Emily Riehl, Markus Spitzweck, Germ\'an Stefanich, and Dominic Verity for interesting conversations related to the subject of this paper.
We thank the MPIM for their hospitality while much of this work was carried out.

\vspace{.25cm}
\section{\mbox{Higher dimensional categories}}
\vspace{.25cm}

\subsection{Enriched categories}

We first recall the notion of homotopy coherent enrichment, as defined and studied in, for instance, \cite{MR3345192}, \cite{heine2024higher}, \cite{heine2024bienriched} \cite{heine2025equivalence}, \cite{HINICH2020107129}.
For every presentably monoidal category $\mV$ there is a presentable 2-category $${\mV}\mathrm{-}\Cat$$ of (not necessarily univalent) $\mV$-enriched categories and $\mV$-enriched functors and a forgetful functor $$ \iota: {\mV}\mathrm{-}\Cat \to \Cat$$ to the presentable 2-category $\Cat$ of (not necessarily univalent) categories,
which is an equivalence for $\mV= \mS$ the category of homotopy types \cite[Corollary 3.23.]{heine2024bienriched}. 
Let $${\mV}\mathrm{-}\Cat^\univ \subset {\mV}\mathrm{-}\Cat$$ be the reflexive full subcategory of univalent $\mV$-enriched categories.

\begin{notation}Let $\mV$ be a presentably monoidal category.
A $\mV$-enriched category $\mC$ has an underlying category $\iota(\mC)$, for every objects $X,Y \in \iota(\mC)$ a morphism object 
$$\Mor_\mC(X, Y) \in \mV$$
and for every objects $X,Y,Z \in \iota(\mC)$ a composition morphism in $\mV:$
$$\Mor_\mC(Y, Z) \ot \Mor_\mC(X, Y) \to \Mor_\mC(X, Z).$$
We write $\X \in \mC$ for $\X \in \iota(\mC)$ and usually notationally identify $\mC$ with $\iota(\mC).$
\end{notation}

The following is \cite[Example 2.134.]{heine2024bienriched}:

\begin{example}
Let $\mV$ be a presentably monoidal category. Every presentably left $\mV$-tensored category is a $\mV$-enriched category.
Every $\mV$-linear functor between presentably left $\mV$-tensored categories is a $\mV$-enriched functor.
In particular, $\mV$ is a $\mV$-enriched category.
\end{example}

For every enriched category there is an opposite one:

\begin{notation}
There is an involution $$(-)^\circ: {\mV}\mathrm{-}\Cat \simeq {\mV^\rev}\mathrm{-}\Cat$$ forming the opposite enriched category.
For every $\mC \in {\mV}\mathrm{-}\Cat$ and $X,Y \in \mC$
there are canonical equivalences
$ \iota(\mC^\circ) \simeq \iota(\mC)^\op$
and $$ \Mor_{\mC^\circ}(X,Y) \simeq \Mor_{\mC}(Y,X).$$
\end{notation}

The following is \cite[Proposition 3.72.]{heine2024bienriched}:

\begin{proposition}

Let $\mV,\mW$ be presentably monoidal categories and 
$\phi:\mV \to \mW$ a lax monoidal functor.
There is an induced functor
$\phi_!: \mV \mathrm{-}\Cat \to \mW \mathrm{-}\Cat$
that transfers the enrichment.
This functor descends to a functor
$\phi_!: \mV \mathrm{-}\Cat^\univ \to \mW \mathrm{-}\Cat^\univ.$
For every $\mV$-enriched category $\mC$ and $X,Y \in \iota(\mC)$
there is an equivalence
$$ \Mor_{\phi_!(\mC)}(X,Y) \simeq \phi(\Mor_\mC(X,Y)).$$
    
\end{proposition}

\begin{proposition}\label{monochar}
Let $\mV$ be a presentably monoidal category.
A $\mV$-enriched functor $\phi: \mC \to \mD$ is a monomorphism in $\mV\mathrm{-}\Cat$ if and only it it induces an embedding $\iota(\mC) \to \iota(\mD)$ on underlying spaces and for every $A,B \in \mC$ the induced morphism 
$\Mor_\mC(A, B) \to \Mor_\mD(\phi(A), \phi(B))$ in $\mV$ is a monomorphism.
    
\end{proposition}

\begin{proof}
The commutative square
$$\begin{xy}
\xymatrix{
\mC \ar[d]^{=} \ar[r]^=
& \mC \ar[d]^\phi 
\\ 
\mC \ar[r]^\phi & \mD}
\end{xy}$$
in $\mV\mathrm{-}\Cat$ is a pullback square if and only if
it induces a pullback square on underlying spaces and on morphism objects. This holds because forming underlying spaces and forming morphism objects preserves pullbacks, and a $\mV$-enriched functor is an equivalence if and only if it induces equivalences on underlying spaces and on morphism objects.   
\end{proof}

The following is \cite[Example 2.38.]{heine2024higher}:

\begin{proposition}\label{enrsub}

Let $\mV$ be a presentably monoidal category and $\mC$ a $\mV$-enriched category. For every $X,Y \in \mC$ let
$T_{X,Y} \to \Mor_\mC(X,Y)$ be a monomorphism in $\mV$.
The following are equivalent:
\begin{enumerate}
\item For every $X \in \mC$ the unit $\tu_\mV \to \Mor_\mC(X,X)$ in $\mV$ factors through $T_{X,X} \to \Mor_\mC(X,X)$
and for every $X,Y,Z \in \mC$ the composition morphism $\Mor_\mC(Y,Z) \ot \Mor_\mC(X,Y) \to \Mor_\mC(X,Z) $ in $\mV$ 
restricts to a morphism $T_{Y,Z} \ot T_{X,Y} \to T_{X,Z} $.


\item There is a monomorphism of $\mV$-enriched categories $\theta: \mB \to \mC$
such that for every $X,Y \in \mC$ the induced monomorphism 
$\Mor_\mB(X,Y) \to \Mor_\mC(\theta(X),\theta(Y))$ in $\mV$ identifies with 
$T_{\theta(X),\theta(Y)} \to \Mor_\mC(\theta(X),\theta(Y))$.

\end{enumerate}

\end{proposition}

\subsection{Suspension and wedge}Let $\mV$ be a presentably monoidal category.
In this subsection we construct the suspension functor $S:\mV\to {_\mV\Cat}$, which sends an object $A$ of $\mV$ to the $\mV$-enriched category $S(A)$ with two objects $\{0,1\}$ and $$\Mor_{S(A)}(0,0)=\Mor_{S(A)}(1,1)=\tu, \ \Mor_{S(A)}(0,1)=A, \Mor_{S(A)}(0,0)=\emptyset.$$

\begin{proposition}\label{susp}
Let $\mV$ be a presentably monoidal category, $n \geq 1$ and $\A_1, ..., \A_\n \in \mV$.
There is a $\mV$-enriched category $S(\A_1,...,\A_\n) $ satisfying the following properties:
\begin{enumerate}[\normalfont(1)]\setlength{\itemsep}{-2pt}
\item The space of objects of $S(\A_1,...,\A_\n) $ is the set $\{0,...,\n\}.$

\item For every $0 \leq \ell \leq \n$ the unit $\tu \to \Mor_{S(\A_1,...,\A_\n)}(\ell,\ell)$ is an equivalence.

\item For every $0 \leq \bk < \ell \leq \n$ the morphism object $\Mor_{S(\A_1,...,\A_\n)}(\ell,\bk)$ is initial.

\item For every $0 \leq \ell < \n$ there is an equivalence $ \Mor_{S(\A_1,...,\A_\n)}(\ell,\ell+1) \simeq \A_{\ell+1}$. 

\item For every $0 \leq \bk < \m \leq \n$ the following induced morphism is an equivalence $$\bigotimes_{\bk \leq \ell < \m}  \A_{\ell+1} \simeq \bigotimes_{\bk \leq \ell < \m} \Mor_{S(\A_1,...,\A_\n)}(\ell,\ell+1) \to  \Mor_{S(\A_1,...,\A_\n)}(\bk,\m).$$

\item For every $\mV$-enriched category $\mC $ and objects $\X_0, ..., \X_\n$ of $\mC$
the induced map is an equivalence $$ \Map_{\mV\mathrm{-}\Cat_{\coprod_{0 \leq \ell \leq \n} B(\tu)/}}(S(\A_1,...,\A_\n), (\mC, \X_0, ..., \X_\n)) \to \prod_{0 \leq \ell < \n} \Map_\mV(\A_{\ell+1}, \Mor_\mC(\X_\ell,\X_{\ell+1})).$$

\end{enumerate}

\end{proposition}

\begin{proof}

By \cite[Remark 4.22, Corollary 4.28]{HEINE2023108941} the category $\Fun(\{0,...,\n\} \times \{0,...,\n\} ,\mV)$ carries a presentably monoidal structure whose tensor unit $\bar{\tu} $ satisfies $$\bar{\tu}(\ell,\ell) \simeq \tu \qquad\mathrm{and}\qquad 
\bar{\tu}(\bk,\ell) \simeq \emptyset$$ 
for $0 \leq \bk \neq \ell \leq \n,$ and whose tensor product $\F_1, ...,  \ot \F_\m$ for
$\F_1, ...,\F_\m \in \Fun(\{0,...,\n\},\mV)$ and $\m \geq 2 $ satisfies $$ (\F_1 \ot ... \ot \F_\m)(\bk,\ell) \simeq \coprod_{0 \leq \bj_1, ..., \bj_{\m-1} \leq \n}\F_1(\bk,\bj_1)\ot \F_2(\bj_1,\bj_2) \ot \F_3(\bj_2,\bj_3) \ot ... \ot \F_\m(\bj_{\m-1},\ell).$$

By \cite[Theorem 6.7, Remark 4.22]{HEINE2023108941} an associative algebra for this monoidal structure is a left $\mV$-enriched category whose underlying space is $\{0,...,\n\}$. More precisely, there is a canonical equivalence $$ \{\{0,...,\n\}\}\times_\mS {{\mV\mathrm{-}\Cat}} \simeq \Alg(\Fun(\{0,...,\n\} \times \{0,...,\n\},\mV)) $$
sending $\mC$ to $\{\Mor_\mC(\bk, \ell)\}_{0\leq \bk, \ell \leq \n}.$

The functor $\theta: \{1,...,\n\} \to \{0,...,\n\} \times \{0,...,\n\}, \ell \mapsto (\ell-1, \ell)$ 
induces a functor $$\theta^* : \Fun(\{0,...,\n\} \times \{0,...,\n\},\mV) \to \Fun( \{1,...,\n\} ,\mV) \simeq \mV^{\times\n}, $$ which admits a fully faithful left adjoint $\rho$ such that for every
$\X_1, ...,\X_\n \in \mV$ and $ 0 \leq \bk, \ell \leq \n$ there are canonical equivalences $\rho(\X_1, ...,\X_\n)(\bk, \ell) \simeq \X_{\ell}$ if $\ell = \bk+1$ and $ \rho(\X_1, ...,\X_\n)(\bk, \ell) \simeq \emptyset$ otherwise. 

Let $$S(\A_1,...,\A_\n) \in \{\{0,...,\n\}\}\times_\mS {{\mV\mathrm{-}\Cat}} $$ be the free associative algebra 
on $\rho(\A_1,...,\A_\n) \in \Fun(\{0,...,\n\} \times \{0,...,\n\},\mV).$ 
For every $0 \leq \bk \leq \m \leq \n$ and $\br \geq 0$ there are equivalences $\rho(\A_1,...,\A_\n)^{\ot\br}(\bk,\m) \simeq \emptyset$ if $\br \neq \m-\bk$
and $\rho(\A_1,...,\A_\n)^{\ot\br}(\bk,\m) \simeq \bigotimes_{\bk \leq \ell < \m} \A_{\ell+1}$ if $\br = \m-\bk > 0 $ and $\rho(\A_1,...,\A_\n)^{\ot\br}(\bk,\m) \simeq \tu$ if $\br = \m-\bk =0.$ 
Hence the formula for the free associative algebra \cite[Theorem 1.1]{FreeAlgebras} 
provides an equivalence
$$ \Mor_{S(\A_1,...,\A_\n)}(\bk,\m) \simeq \bigotimes_{\bk \leq \ell < \m}  \A_{\ell+1}$$
for every $0 \leq \bk < \m \leq \n$ and $ \Mor_{S(\A_1,...,\A_\n)}(\bk,\m) \simeq \tu $
for every $0 \leq \bk = \m \leq \n$. 
This proves (1)-(5). It remains to prove (6). 
Let $\X$ be the space of objects of $\mC$ and $\varphi:  \{0,...,\n\} \to \X$ the map
corresponding to the family $(\A_1,...,\A_\n)$ of objects of $\mC.$
The map in (6) factors as 
$$ \Map_{\mV\mathrm{-}\Cat_{\coprod_{0 \leq \ell \leq \n} B(\tu)/}}(S(\A_1,...,\A_\n), \mC) \simeq \Map_{\{\{0,...,\n\}\}\times_\mS {{\mV\mathrm{-}\Cat}}}(S(\A_1,...,\A_\n), \varphi^*(\mC)) \simeq $$$$
\Map_{\Fun(\{0,...,\n\} \times \{0,...,\n\},\mV)}(\rho(\A_1,...,\A_\n),\varphi^*(\mC)) \simeq 
\prod_{0 \leq \ell < \n} \Map_\mV(\A_{\ell+1}, \Mor_\mC(\X_\ell,\X_{\ell+1})).$$	
\end{proof}

Specializing \cref{susp} to $n=1$ gives the following:

\begin{corollary}
Let $\mV$ be a presentably monoidal category and $\A \in \mV$.
There is a $\mV$-enriched category $S(\A) $, which we call the suspension of $A$, satisfying the following properties:
\begin{enumerate}[\normalfont(1)]\setlength{\itemsep}{-2pt}
\item The space of objects of $S(\A) $ is the set $\{0,1\}.$

\item For every $0 \leq \ell \leq 1$ the unit $\tu \to \Mor_{S(\A)}(\ell,\ell)$ is an equivalence.

\item The morphism object $\Mor_{S(\A)}(1,0)$ is initial.

\item There is an equivalence $ \Mor_{S(\A)}(0,1) \simeq \A$. 

\item For every $\mV$-enriched category $\mC $ and objects $\X,\Y$ of $\mC$
the induced map is an equivalence $$ \Map_{\mV\mathrm{-}\Cat_{B(\tu) \coprod B(\tu)/}}(S(\A), (\mC; \X,\Y)) \to \Map_\mV(\A, \Mor_\mC(\X,\Y)).$$

\end{enumerate}
    
\end{corollary}

\begin{definition}

Let $\mV$ be a presentably monoidal category, $\mC, \mD$ be $\mV$-enriched categories and $X,Y \in \mC, Y,Z \in \mD$
objects. The bipointed wedge $$(\mC; X,Y) \vee (\mD; Y,Z) $$
is the pushout $ (\mC \coprod_{\{Y\}} \mD; X,Z).$
    
\end{definition}

\begin{corollary}
Let $\mV$ be a presentably monoidal category, $n \geq 1$ and $\A_1, ..., \A_\n \in \mV$.
\begin{enumerate}[\normalfont(1)]\setlength{\itemsep}{-2pt}
\item The space of objects of the $\mV$-enriched category $S(A_1) \vee ...\vee S(A_n)$ is the set $\{0,...,\n\}.$
\item For every $0 \leq \ell \leq \n$ the unit $\tu \to \Mor_{S(A_1) \vee ...\vee S(A_n)}(\ell,\ell)$ is an equivalence.
\item For every $0 \leq \bk < \ell \leq \n$ the morphism object $\Mor_{S(A_1) \vee ...\vee S(A_n)}(\ell,\bk)$ is initial.
\item For every $0 \leq \ell < \n$ there is an equivalence $ \Mor_{S(A_1) \vee ...\vee S(A_n)}(\ell,\ell+1) \simeq \A_{\ell+1}$. 
\item For every $0 \leq \bk < \m \leq \n$ the following induced morphism is an equivalence $$\bigotimes_{\bk \leq \ell < \m}  \A_{\ell+1} \simeq \bigotimes_{\bk \leq \ell < \m} \Mor_{S(A_1) \vee ...\vee S(A_n)}(\ell,\ell+1) \to  \Mor_{S(A_1) \vee ...\vee S(A_n)}(\bk,\m).$$
\end{enumerate}

\end{corollary}

\begin{proof}

By \cref{susp} (6) there is a canonical equivalence of left $\mV$-enriched categories
$$ S(\A_1,...,\A_\n) \simeq S(A_1) \vee ...\vee S(A_n). $$
Thus \cref{susp} (5) gives the result.
\end{proof}

\subsection{Cartesian enrichment}\label{int}
The universal property of the cartesian product allows one to simplify the theory of enrichment by viewing enriched categories 
as internal categories in which the object of objects is a space.

\begin{definition}
Let $\mV$ be a category that admits finite limits.
A {\em category object} in $\mV$ is a	functor $X: \Delta^\op \to \mV$ such that for every $n \geq 0$
the induced morphism $X_n \to X_1 \times_{X_0} ... \times_{X_0} X_1$ is an equivalence.
	
\end{definition}

\begin{notation}

Let $\mV$ be a category that admits finite limits and $\mW \subset \mV$ a full subcategory.
Let $\Cat(\mV) \subset \Fun(\Delta^\op,\mV)$ be the full subcategory of category objects.
Let $\Cat(\mV;\mW) \subset \Cat(\mV)$ be the full subcategory of category objects $X$ such that $X_0 \in \mV$ belongs to $\mW.$

\end{notation}

\begin{remark}
	
The colocalization $\mV \rightleftarrows \Fun(\Delta^\op,\mV): \ev_0$ whose left adjoint is the diagonal embedding and whose right adjoint evaluates at $[0] \in \Delta$, restricts to colocalizations
$\mV \rightleftarrows \Cat(\mV), \mW \rightleftarrows \Cat(\mV; \mW).$
	
\end{remark}

\begin{remark}\label{present}
Let $\mV$ be a presentable category.
The full subcategory $\Cat(\mV) \subset \Fun(\Delta^\op,\mV)$ is an accessible localization
with respect to the set of morphisms $$ \Map(-,[1]) \times X \coprod_{\Map(-,[0]) \times X}\cdots\coprod_{\Map(-,[0]) \times X}  \Map(-,[1]) \times X  \to  \Map(-,[n]) \times X  $$ for $n \geq 0$ and $X \in \mV.$
Since $\mV$ is presentable, this is of small generation, so $\Cat(\mV)$ is presentable, and $\Cat(\mV)$  is cartesian closed if $\mV$ is cartesian closed.
Since any (generating) local equivalence is inverted by the functor $ \Cat(\mV) \to \mV$ evaluating at zero, the localization $L: \Fun(\Delta^\op,\mV) \rightleftarrows \Cat(\mV)$ is relative to $\mV.$ 

Let $\mV \leftrightarrows \mW $ be an accessible localization.
Then
\[
\Cat(\mV) \rightleftarrows \Cat(\mV; \mW)
\]
is an accessible localization whose right adjoint is the canonical embedding, since the functor $\Cat(\mV) \to \mV$ evaluating at zero admits a left adjoint and a right adjoint. Hence $\Cat(\mV;\mW)$ is presentable, and $\Cat(\mV;\mW)$ is cartesian closed if $\mV$ is cartesian closed.
\end{remark}

Let $\mV$ be a presentable category. 
There is a unique left adjoint functor $\mS \to \mV$ preserving the final objects.
If this functor is fully faithful, we consider $\Cat(\mV;\mS)$ using the embedding $\mS \to \mV.$

In the cartesian monoidal case, $\mV$-enriched categories can sometimes be regarded as category objects relative to the subcategory of spaces.
The next result is \cite[5.6.1 Corollary]{HINICH2020107129}:

\begin{proposition}
Let $\mV$ be a presentable category such that the unique colimit preserving functor $\mS\to\mV$ which sends the point to the final object is a fully faithful right adjoint, and the functor $\mS^{\op}\to\Cat$ given by $X\mapsto \mV_{/X}$ preserves limits. Then there is a canonical equivalence
\[
\mV\mathrm{-}\Cat\simeq\Cat(\mV;\mS).
\]
\end{proposition}

\subsection{$\infty$-categories}

In this subsection we inductively define $n$-categories for any $0 \leq n \leq \infty$ via the theory of homotopy coherent enrichment.

\begin{definition}
For every $\n \geq 0$ we inductively define the presentable cartesian closed category $\n\Cat$ of small (not necessarily univalent) $\n$-categories by setting
$$(\n+1)\Cat:= {\n\Cat}\mathrm{-}\Cat$$
starting with $0\Cat :=\mS$. 
\end{definition}

\begin{notation}
For every $\n \geq 0$ we inductively define colocalizations
$$\n\Cat \rightleftarrows (\n+1)\Cat: \iota_\n,$$
where both adjoints  preserve finite products and filtered colimits.
Let
$$0\Cat= \mS \rightleftarrows 1\Cat = {\mS}\mathrm{-}\Cat : \iota_0 $$ be the canonical colocalization whose right adjoint assigns the space of objects. Let $$(\n+1)\Cat= {\n\Cat}\mathrm{-}\Cat \rightleftarrows (\n+2)\Cat= {(\n+1)\Cat} \mathrm{-}\Cat:\iota_{\n+1}:= (\iota_\n)_! $$
be the induced adjunction.	
	
\end{notation}

\begin{definition}The presentable category $\infty\Cat$ of small (non-univalent) $\infty$-categories is the limit
$$\infty\Cat:= \lim(\cdots\xrightarrow{\iota_{\n}} \n \Cat \xrightarrow{\iota_{\n-1}}\cdots \xrightarrow{\iota_0} 0 \Cat) $$
of presentable categories and right adjoint functors.

\end{definition}

The next proposition follows from the fact that the functor ${_{(-)}\Cat}$
preserves small limits:

\begin{proposition}\label{fix}
	
There is a canonical equivalence
$$ \infty\Cat \simeq {\infty\Cat}\mathrm{-}\Cat. $$
	
\end{proposition}

\begin{notation}
Let $\partial\bD^1$ denote the two element set $\{0,1\}$.
\end{notation}

\begin{remark}
We will often regard $\partial\bD^1$ as a totally ordered set by declaring $0<1$.
However, it should not be confused with $\bD^1$, the category object corresponding to the $1$-simplex and represented by the ordinal $[1]$.
Rather, the ordering on $\partial\bD^1$ comes from the fact that it is a (not full) subcategory of $\bD^1$.
\end{remark}

\begin{notation}
Let $\Mor: \infty\Cat_{\partial\bD^1/} \to \infty\Cat$
be the canonical functor $$\infty\Cat_{\partial\bD^1/} \simeq \mS_{\partial\bD^1/}\times_\mS {\infty\Cat}\mathrm{-}\Cat \to \infty\Cat $$
sending $(\mC,\X,\Y)$ to $\Mor_\mC(\X,\Y).$
\end{notation}

\begin{remark}\label{homfil}
The functor $\Mor: \infty\Cat_{\partial\bD^1/} \to \infty\Cat$
preserves filtered colimits, limits, and univalent and strict $\infty$-categories.
 	
\end{remark}

\begin{definition}

A functor $X \to Y$ of $\infty$-categories is an inclusion -- or subcategory inclusion-- if is a monomorphism in the category $\infty\Cat,$ i.e. for every $\infty$-category $Z$ the induced map $$\Map_{\infty\Cat}(Z,X) \to \Map_{\infty\Cat}(Z,Y)$$ is an embedding.
In this case we also say that $X$ is a subcategory of $Y$.
 
\end{definition}

\cref{monochar} implies the following:

\begin{corollary}

A functor $\phi: X \to Y$ of $\infty$-categories is an inclusion if and only if it induces an embedding $\iota_0(X) \to \iota_0(Y)$ on underlying spaces and for every $A,B \in X$ the induced functor
$$\Mor_X(A, B) \to \Mor_Y(\phi(A), \phi(B))$$ is an inclusion.

\end{corollary}

\begin{definition}Let $ n \geq 1.$ By induction on $n$ we define $n$-univalent $\infty$-categories.
An $\infty$-category is 1-univalent if it is local with respect to the unique functor $\{0 \simeq 1 \} \to *$, where $\{0\simeq 1\}$ is the category obtained from $\bD^1$ by inverting the unique map from $0$ to $1$.
An $\infty$-category is $n+1$-univalent if it is 1-univalent and all morphism $\infty$-categories are $n$-univalent.
An $\infty$-category is univalent if it is $n$-univalent for every $n \geq 1.$
\end{definition}

\begin{definition}Let $ n \geq 0.$ By induction on $n$ we define $n$-strict $\infty$-categories.
An $\infty$-category is 0-strict if it is a set.
An $\infty$-category is $n+1$-strict if it is 0-strict and all morphism $\infty$-categories are $n$-strict.
An $\infty$-category is strict if it is $n$-strict for every $n \geq 0.$
\end{definition}

\begin{definition}
Let $ \infty\Cat^{\univ}, \infty\Cat^{\strict}\subset \infty\Cat$ 
be the respective full subcategories of univalent, strict $n$-categories.
\end{definition}

\begin{remark}
For every $\n \geq 0$ the colocalization
$\n\Cat \rightleftarrows (\n+1)\Cat: \iota_\n$
restricts to the respective full subcategories of univalent $n$-categories and strict $n$-categories.

\end{remark}

\begin{remark}
By definition there are canonical equivalences
\begin{align*}
0\Cat^{\univ}&\simeq\mS,\\
n\Cat^{\mathrm{univ}}&\simeq {(\n-1)\Cat^\univ} \mathrm{-}\Cat^\univ,\\
\infty\Cat^{\mathrm{univ}}&\simeq  \lim(\cdots\xrightarrow{\iota_{\n}} \n \Cat^{\mathrm{univ}}\xrightarrow{\iota_{\n-1}}\cdots \xrightarrow{\iota_0} 0 \Cat^{\mathrm{univ}}).
\end{align*}

\end{remark}

\begin{notation}
For every $0 \leq \n \leq \m$ the left adjoint embeddings $\n\Cat \leftrightarrows \m\Cat$ preserve small limits and thus induce left adjoint embeddings $\n\Cat \leftrightarrows \infty\Cat: \iota_\n$ that preserve small limits
and so admit left adjoints $\tau_\n: \infty\Cat \to \n\Cat$ by presentability.
Similarly, we obtain left adjoint embeddings $$\n\Cat^{\mathrm{univ}}\leftrightarrows \infty\Cat^{\mathrm{univ}}: \iota_\n,$$$$ \n\Cat^{\mathrm{strict}}\leftrightarrows \infty\Cat^{\mathrm{strict}}: \iota_\n$$ that preserve small limits and so admit left adjoints, which we also denote by $\tau_\n.$

\end{notation}

We have the following filtration of any $\infty$-category:

\begin{lemma}\label{decom}
Let $\mC$ be an $\infty$-category.
The sequential diagram $$\iota_0(\mC) \to \cdots \to \iota_\n(\mC) \to \iota_{\n+1}(\mC) \to \cdots \to \mC$$ exhibits $\mC$ as the colimit in $\infty\Cat$ of the diagram $\iota_0(\mC) \to \cdots \to \iota_\n(\mC) \to \iota_{\n+1}(\mC) \to \cdots$
\end{lemma}

\begin{proof}
For every $\infty$-category $\mD$ the induced map
$$ \Map_{\infty\Cat}(\mC,\mD) \to \lim(\cdots \to \Map_{\infty\Cat}(\iota_{\n+1}(\mC),\mD) \to \Map_{\infty\Cat}(\iota_\n(\mC),\mD)  \to \cdots \to \Map_{\infty\Cat}(\iota_0(\mC),\mD)) $$$$\simeq \lim(\cdots \to \Map_{\infty\Cat}(\iota_{\n+1}(\mC),\iota_{\n+1}(\mD)) \to \Map_{\infty\Cat}(\iota_\n(\mC),\iota_\n(\mD)) \to \cdots \to \Map_{\infty\Cat}(\iota_0(\mC),\iota_0(\mD))) $$
is an equivalence by the definition of $\infty\Cat$ as a limit.	
\end{proof}

\begin{lemma}\label{carclo}Let $ 0 \leq \n \leq \infty.$
The presentable category $\n\Cat$ is cartesian closed.

\end{lemma}

\begin{proof}
We prove that the product on $\infty\Cat$ preserves small colimits in each variable.
By definition $\infty\Cat$ is the limit of the diagram
$$\cdots\xrightarrow{\iota_{\n}} \n \Cat \xrightarrow{\iota_{\n-1}} \cdots \xrightarrow{\iota_0} 0 \Cat $$
of presentable categories and right adjoint functors.
Via the canonical equivalence $\Pr^\L\simeq (\Pr^\R)^\op$ the latter is the colimit of the diagram
$$0\Cat \subset \cdots \subset \n\Cat \subset \cdots $$
in $\PrL$, which underlies a diagram in $\mathrm{CAlg}(\PrL)$ of presentable cartesian symmetric monoidal categories, and so inherits a presentable symmetric monoidal structure.
The tensor unit of $\infty\Cat$ is the final object since the embedding $\Cat \subset \infty\Cat$ is symmetric monoidal and preserves finite products.
We need to see that for every $\infty$-categories $\mC$ and $\mD$ the canonical functor
$\mC \ot \mD \to \mC \times \mD$ is an equivalence.
The category $\infty\Cat$ is the limit of a diagram of presentable categories and
small filtered colimit and small limit preserving functors. Thus filtered colimits and small limits in $\infty\Cat$ are formed componentwise.
This guarantees that the product of $\infty\Cat$ preserves filtered colimits because for every $\n \geq 0$ the product of $\n\Cat$ preserves filtered colimits.
Consequently, in view of \cref{decom} when proving that the canonical functor $\mC \ot \mD \to \mC \times \mD$ is an equivalence we can assume that $\mC,\mD$ are $\n$-categories. 
In this case the functor $\mC \ot \mD \to \mC \times \mD$ is an equivalence because the embedding
$\n\Cat \subset \infty\Cat$ is symmetric monoidal and preserves finite products.
\end{proof}

\begin{notation}

For every $\infty$-category $\mC$ let 
$\Fun(\mC,-): \infty\Cat \to \infty\Cat$ be the right adjoint of the functor $(-) \times \mC:\infty\Cat \to \infty\Cat.$
    
\end{notation}

\begin{corollary}
The category $\infty\Cat$ refines to an $\infty$-category $\infty\scat$ such that, for any objects $\mC$ and $\mD$,
\[
\Mor_{\infty\scat}(\mC,\mD)=\Fun(\mC,\mD).
\]
\end{corollary}

\begin{definition}Let $\n \geq 1.$
\begin{enumerate}[\normalfont(1)]\setlength{\itemsep}{-2pt}
\item Every functor of $\infty$-categories is $-1$-full.
	
\item A functor of $\infty$-categories is $0$-full if it is essentially surjective.
	
\item A functor of $\infty$-categories is $\n$-full if it induces $n-1$-full functors on morphism $\infty$-categories.

\end{enumerate}
	
\end{definition}

\begin{definition}Let $ \n \geq 1.$
\begin{enumerate}[\normalfont(1)]\setlength{\itemsep}{-2pt}
\item A functor of $\infty$-categories is $-1$-faithful if it is an equivalence.
	
\item A functor of $\infty$-categories is $0$-faithful if it is fully faithful.
		
\item A functor of $\infty$-categories is $\n$-faithful if it induces $n-1$-faithful functors on morphism $\infty$-categories.

\end{enumerate}
	
\end{definition}

It is important to observe that a $0$-full and $0$-faithful functor of $\infty$-categories need {\em not} be an equivalence.
This is because we are not forcing our $\infty$-categories to be univalent, or complete in the sense of Rezk \cite{rezk2001model}.
For instance, the category with two isomorphic objects is not equivalent (in the sense of categorical isomorphism) to the terminal category, though there are fully faithful and essentially surjective functors from either one to the other.
However, the univalent completion of a $0$-full and $0$-faithful functor is always an equivalence.

\begin{proposition}\label{completion}Let $0 \leq \n \leq \infty$.
The embedding $ \n\Cat^{\mathrm{univ}}\subset \n\Cat$ of the full subcategory of univalent $n$-categories admits a left adjoint, which we call univalent completion and denote by $(-)^\wedge$, that preserves finite products.
For $ 0 \leq \n < \infty$ the local equivalences are precisely the $n-1$-faithful functors, which are $m$-full for every $m < n.$
\end{proposition}

\begin{proof}
We proceed by induction on $\n \geq 0.$
By definition the embedding $0\Cat^{\mathrm{univ}}\subset 0\Cat $ is the identity.
So the statement for $\n=0$ is trivial.
We assume that the embedding $ \n\Cat^{\mathrm{univ}}\subset \n\Cat$ admits a left adjoint $\mathrm{L}_n$ that preserves finite products and whose local equivalences are the 
$n-1$-faithful functors, which are $m$-full for every $m < n.$
Then the embedding $$(n+1)\Cat^{\mathrm{univ}}= {\n\Cat^\univ}\mathrm{-}\Cat^\univ \subset (\n+1)\Cat= {\n\Cat}\mathrm{-}\Cat $$ is right adjoint to the functor $$\L_{\n+1}: (\n+1)\Cat= {\n\Cat}\mathrm{-}\Cat  \xrightarrow{(L_n)_!}  (\n+1)\Cat= {\n\Cat^\univ}\mathrm{-}\Cat \to {\n\Cat^\univ}\mathrm{-}\Cat^\univ $$ that preserves finite products and whose local equivalences are the $n$-faithful functors, which are $m$-full for every $m \leq n.$
 
For every $\mC \in \infty\Cat$ the $\n+1$-category $\L_{\n+1}(\iota_\n(\mC))$ is an $\n$-category.
Hence the canonical functor $\iota_\n(\mC) \to \L_{\n+1}(\iota_\n(\mC))$ induces a functor
$\L_\n(\iota_\n(\mC)) \to \L_{\n+1}(\iota_\n(\mC))$.
We obtain a canonical functor $$\L_\n(\iota_\n(\mC))\to \L_{\n+1}(\iota_\n(\mC)) \to \L_{\n+1}(\iota_{\n+1}(\mC)).$$
By \cref{decom} the canonical functor $$\colim(\iota_0(\mC) \to \cdots \to \iota_\n(\mC)) \to \cdots) \to \mC$$ is an equivalence.
We obtain a functor $$\mC \simeq \colim(\iota_0(\mC) \to \cdots \to \iota_\n(\mC)\to\cdots) \to \L_\infty(\mC):= \colim(\L_0(\iota_0(\mC)) \to \cdots \to \L_\n(\iota_\n(\mC))\to\cdots)$$
that induces for every univalent $\infty$-category $\mD$ a map
$\Map_{\infty\Cat}(\L_\infty(\mC),\mD) \to \Map_{\infty\Cat^\univ}(\mC,\mD).$
The latter identifies with the
equivalence
\begin{align*}
&\lim(\cdots\to\Map_{\n\Cat}(\L_\n(\iota_\n(\mC)),\iota_\n(\mD)) \to ... \to \Map_{0\Cat}(\L_0(\iota_0(\mC)),\iota_0(\mD))) \to\\
&\lim(... \to\Map_{\n\Cat}(\iota_\n(\mC),\iota_\n(\mD)) \to ... \to \Map_{0\Cat}(\iota_0(\mC),\iota_0(\mD))).
\end{align*}
The left adjoint $\L$ preserves finite products since sequential colimits commute with finite products in $\infty\Cat$ by \cref{carclo}.
\end{proof}

\begin{definition}\label{ruik} Let $\n \geq 0.$
We inductively define involutions $$(-)^\op_\n, (-)^\co_\n: \n\Cat \to \n\Cat$$ by setting
$(-)^\co_0, (-)^\op_0:  0\Cat \to 0\Cat$ are the identities, $$(-)^\op_{\n+1}: {(\n+1)}\Cat \xrightarrow{(-)^\circ} {(\n+1)}\Cat \xrightarrow{((-)^\co_\n)_!}{(\n+1)}\Cat, $$ $$(-)^\co_{\n+1}:=((-)^\op_\n)_!: {(\n+1)}\Cat \xrightarrow{ }{(\n+1)}\Cat.$$
There are commutative squares:
$$\begin{xy}
\xymatrix{
{(\n+1)}\Cat \ar[d]^{\iota_\n}  \ar[r]^{(-)_{\n+1}^\op} & {(\n+1)}\Cat  \ar[d]^{\iota_\n}
\\ 
{\n}\Cat \ar[r]^{(-)_{\n}^\op} & {\n}\Cat
}
\end{xy}
\qquad
\begin{xy}
\xymatrix{
{(\n+1)}\Cat \ar[d]^{\iota_\n}  \ar[r]^{(-)_{\n+1}^\co} & {(\n+1)}\Cat  \ar[d]^{\iota_\n}
\\ 
{\n}\Cat \ar[r]^{(-)_{\n}^\co} & {\n}\Cat
}
\end{xy}
$$
and so induced involutions on the limit $$(-)^\op, (-)^\co: \infty\Cat \to \infty\Cat$$
that preserve strict and univalent $\infty$-categories.

\end{definition}

\begin{remark}\label{oppo} By \cref{ruik} there are commutative squares, where $\sigma$ permutes the distinguished objects:

$$\begin{xy}
\xymatrix{
\infty\Cat_{\partial\bD^1/} \ar[d]^{\Mor}  \ar[r]^{(-)^\co} & \infty\Cat_{\partial\bD^1/}  \ar[d]^{\Mor}
\\ 
\infty\Cat \ar[r]^{(-)^\op} & \infty\Cat
}
\end{xy}\qquad
\begin{xy}
\xymatrix{
\infty\Cat_{\partial\bD^1/} \ar[d]^{\Mor}  \ar[r]^{\sigma \circ (-)^\op} & \infty\Cat_{\partial\bD^1/}  \ar[d]^{\Mor}
\\ 
\infty\Cat \ar[r]^{(-)^\co} & \infty\Cat.
}
\end{xy}
$$

\end{remark}

Next we introduce the categorical suspension, which applied to the final and initial $\infty$-categories gives rise to the categorical disks and their boundaries, respectively.
\cref{susp} guarantees that there is an $\infty$-category satisfying the following definition: 

\begin{definition}\label{suspi}
Let $\mC$ be an $\infty$-category. The categorical suspension of $\mC$, denoted by $S(\mC),$
is the unique $\infty$-category whose underlying space is the two-element set $\{0,1\}$ and whose morphism $\infty$-categories are $$\Mor_{S(\mC)}(1,0)\simeq \emptyset, \ \Mor_{S(\mC)}(0,1)\simeq \mC, \ \Mor_{S(\mC)}(0,0)\simeq \Mor_{S(\mC)}(1,1) \simeq *$$ 
and such that for every $\infty$-category $\mD$ containing an ordered pair of objects $(\Y,\Z) \in \mD$ the induced map
\begin{equation*}
\Map_{\infty\Cat_{\partial\bD^1/}}((S(\mC); 0,1),(\mD; \Y,\Z) ) \to \Map_{\infty\Cat}(\mC, \Mor_\mD(\Y,\Z))
\end{equation*}
is an equivalence.

\end{definition}

\subsection{Gaunt $\infty$-categories}
The notion of a gaunt $\infty$-category goes back to Barwick--Schommer-Pries \cite{barwickunicity}.
In practice, gaunt $\infty$-categories are something like the skeleta upon which general $\infty$-categories are built.
Indeed, Barwick--Schommer-Pries present $n$-categories as a localization of presheaves of spaces on gaunt $n$-categories. 
We explore the geometry of gaunt $\infty$-categories in the this section, and establish that certain small families of gaunt $\infty$-categories are dense in $\infty\Cat$.

\begin{definition}Let $\n \geq 0$.
The $\n$-disk is the $n$-fold suspension $\bD^\n:= S^{\n}(*)$ of the terminal $\infty$-category $\ast$.
\end{definition}

The $n$-disk is sometimes referred to as the walking $\n$-cell or $n$-morphism, because $\bD^n$ corepresents the space of $n$-cells functor $\infty\Cat\to\mS$.
Hence it may be regarded as being freely generated by a single $n$-cell, with no relations.

\begin{definition}Let $\n \geq 0$.
The boundary of the $\n$-disk is the $n$-fold suspension $\partial\bD^\n:= S^{\n}(\emptyset)$ of the initial $\infty$-category $\emptyset$.
	
\end{definition}

\begin{remark}
The functor $\emptyset \subset *$ induces inclusions $\partial\bD^\n \subset \bD^\n$ for every $\n \geq 0.$
	
\end{remark}

\begin{example}
Then $\partial\bD^0=\ast, \partial\bD^1=S(\emptyset)=*\coprod*$ is the set with two elements.
Viewed as a subject of $\bD^1$, these elements acquire a natural ordering.
\end{example}

\begin{definition}
The bipointed wedge of an ordered pair $(\mC,\mD)$ of bipointed $\infty$-categories  $\infty\Cat_{/\partial\bD^1}$ and $\mC'\in\infty\Cat_{/\partial\bD^1}$ is the bipointed $\infty$-category obtained by taking the pushout along the common basepoint.
\end{definition}

\begin{definition}\label{Theta}
	
Let $\Theta' \subset \infty\Cat_{\partial\bD^1/}$ be the full subcategory generated by $\bD^0$ under suspensions and bipointed wedges and $\Theta \subset \infty\Cat$ the essential image of $\Theta'$ under the forgetful functor.	
	
\end{definition}

\begin{remark}

An $\infty$-category belongs to $\Theta$ if and only if it is of the form $$ \bD^{i_0} \coprod_{\bD^{j_1}} \bD^{i_1} \coprod_{\bD^{j_2}} \cdots \coprod_{\bD^{j_n}} \bD^{i_n} $$
for a sequence of natural numbers $n, i_0,\ldots, i_n, j_1,\ldots, j_n$ and monomorphisms 
$ \bD^{j_\ell} \rightarrowtail \bD^{i_\ell},\bD^{j_\ell} \rightarrowtail \bD^{i_{\ell-1}}$, $ 1 \leq \ell \leq n$.

\end{remark}

\begin{theorem}\label{theta}

The restricted Yoneda embeddings \begin{equation}\label{jopol}
\infty\Cat \to \Fun(\Theta^\op,\mS)\qquad\mathrm{and}\qquad \infty\Cat^\strict \to \Fun(\Theta^\op,\Set)\end{equation} are fully faithful and admit left adjoints that preserve finite products.
For every $ 0 \leq n \leq \infty$ the full subcategories $n\Cat \subset \infty\Cat, n\Cat^\strict \subset \infty\Cat^\strict$ are generated under small colimits by the disks of dimension smaller $n+1.$

A $\Theta$-space (set) is in the essential image of the restricted Yoneda embedding if and only if it satisfies the Segal condition, i.e. it is local with respect to the map 
$$ \N(\bD^{i_0}) \coprod_{\N(\bD^{j_1})} \N(\bD^{i_1})\coprod_{\N(\bD^{j_2})} \cdots \N(\coprod_{\bD^{j_n}}) \N(\bD^{i_n}) \to \N(\bD^{i_0} \coprod_{\bD^{j_1}} \bD^{i_1} \coprod_{\bD^{j_2}} \cdots \coprod_{\bD^{j_n}} \bD^{i_n}).$$

\end{theorem}

\begin{proof}
For every $n \geq 0$ let $\Theta_n \subset \Theta$ be the full subcategory of $(n,n)$-categories.
So $\Theta $ is the sequential colimit of the diagram $ \Theta_0 \subset ... \subset\Theta_n \subset .... $	
As a consequence of \cite[Corollary 4.2]{Haugseng_2017}, \cite[5.6.1 Corollary]{HINICH2020107129} the restricted Yoneda embeddings \begin{equation}\label{Yoned1}
n\Cat \to \Fun(\Theta_n^\op,\mS), \ n\Cat^\strict \to \Fun(\Theta_n^\op,\Set)\end{equation} are fully faithful and the essential images are precisely the Segal $\Theta_n$-spaces (sets).
The restricted Yoneda embeddings \ref{jopol} factor as
\begin{align*}
\infty\Cat =\lim(\cdots \to n\Cat\to\cdots) \to &\lim(\cdots \to \Fun(\Theta_n^\op,\mS)\to\cdots) \simeq \Fun(\Theta^\op,\mS)\\
\infty\Cat^\strict =\lim(\cdots\to n\Cat^\strict\to\cdots) \to &\lim(\cdots \to \Fun(\Theta_n^\op,\Set)\to\cdots) \simeq \Fun(\Theta^\op,\Set)
\end{align*}
respectively, and so are fully faithful as well.
By \cite[Proposition 7.21]{Cartesian}, for every $n \geq 0$, the restricted Yoneda embeddings \ref{Yoned1} admit left adjoints $L_n$ that preserve finite products.
Hence the restricted Yoneda embeddings
\begin{align*}
\infty\Cat =\lim(\cdots \to n\Cat\to\cdots) \to &\lim(\cdots \to \Fun(\Theta_n^\op,\mS)\to\cdots) \simeq \Fun(\Theta^\op,\mS)\\
\infty\Cat^\strict =\lim(\cdots\to n\Cat^\strict\to\cdots) \to &\lim(\cdots \to \Fun(\Theta_n^\op,\Set)\to\cdots) \simeq \Fun(\Theta^\op,\Set)
\end{align*}
admit left adjoints $L$ that send $X$ to the sequential colimit $\colim(L_0(X_0) \to L_1(X_1) \to ...)$
where $X_n$ is the image of $X$ in $\Fun(\Theta_n^\op,\mS), \Fun(\Theta_n^\op,\Set)$, respectively.
Hence also the left adjoints $L$ preserve finite products since $\infty\Cat, \infty\Cat^\strict$ are cartesian closed by \cref{carclo}.

The second part of the statement follows from the fact that for every $0 \leq n \leq \infty$ 
any full subcategory of $n\Cat$ that contains the disks of dimension smaller $n+1$ and is closed under small colimits also contains $\Theta_n$ and therefore agrees with $n\Cat$ by the first part of the proof. The same holds for $n\Cat^\strict$.
\end{proof}

\begin{proposition}\label{subchar} Let $X$ be an $\infty$-category and $\mE$ a collection of cells of $X.$
The following are equivalent:

\begin{enumerate}

\item There is a subcategory inclusion $Y \to X$ such that for every
$n \geq 0$ an $n$-morphism of $X$ belongs to $Y$ if and only if 
it belongs to $\mE.$

\item 
For every $\theta \in \Theta$ and functors $\phi: \theta \to X$
and $\bD^m \to \theta$ for $n \geq 0$, the composition $\bD^n \to \theta \xrightarrow{\phi} X$ belongs to $\mE$
if for every standard cell $\bD^n \to \theta$ the functor $\bD^n \to \theta \xrightarrow{\phi} X$ belongs to $\mE$.

\item The collection $\mE$ contains all identity 1-morphism between objects of $X$ in $\mE$, for every objects $A, B $ of $ X$ in $\mE$ there is a subcategory inclusion $Y_{A,B} \to \Mor_X(A,B)$
such that an $n$-morphism of $\Mor_X(A,B)$ belongs to $Y_{A,B}$ if and only if the corresponding $n+1$-morphism of $X$ belongs to $\mE,$
and for every objects $A,B,C$ of $X$ in $\mE$ the composition
functor $\Mor_X(B,C) \times \Mor_X(A,B) \to \Mor_X(A,C)$ sends $Y_{B,C} \times Y_{A,B} $ to $Y_{A,C} $.

\item The collection $\mE$ contains all identity 1-morphism between objects of $X$ in $\mE$, for every objects $A, B $ of $ X$ in $\mE$ there is a subcategory inclusion $Y_{A,B} \to \Mor_X(A,B)$
such that an $n$-morphism of $\Mor_X(A,B)$ belongs to $Y_{A,B}$ if and only if the corresponding $n+1$-morphism of $X$ belongs to $\mE,$
and for every morphisms $A' \to A, B \to B' $ of $X$ in $\mE$ the induced functor $\Mor_X(A,B) \to \Mor_X(A',B')$ sends $Y_{A,B} $ to $Y_{A',B'} $.

\end{enumerate}

    
\end{proposition}

\begin{proof}
1. clearly implies 2. We prove first that 2. implies 1.
Let $\N: \infty\Cat \to \Fun(\Theta^\op, \mS)$ be the nerve, which is an equivalence to the full subcategory of Segal $\Theta$-spaces by \cref{theta}.
For every $\theta \in \Theta$ let $\mE(\theta) \subset \N(X)(\theta)$ be the full subspace of objects $F$ such that for every standard cell $\bD^n \to \theta $ of $\theta$ for $n \geq 0$ the induced map $ \N(X)(\theta) \to \N(X)(\bD^n)$ sends $Z$ to an element of $\mE_n.$
Then 2. is equivalent to say that for every map $\theta \to \theta' $ in $\Theta$ the induced map 
$\N(X)(\theta') \to \N(X)(\theta) $ restrict to a map $\mE(\theta') \to \mE(\theta) $. Hence 2. is equivalent to the condition that the Segal $\Theta$-space $\N(X)$ restricts to a $\Theta$-space $\N(X)_\mE$ that sends $\theta$ to 
$\mE(\theta).$ By construction the $\Theta$-space $\N(X)_\mE$ is a Segal $\Theta$-space. The resulting monomorphism $\N(X)_\mE \to \N(X)$ of Segal
$\Theta$-spaces represents a monomorphism $Y \to X$ in $\infty\Cat.$
By construction, for every $n \geq 0$ an $n$-morphism of $X$ belongs to $Y$ if and only if it belongs to $\mE_n.$

By \cref{enrsub} conditions 1. and 3. are equivalent.
3. clearly implies 4. We prove that 4. implies 3.

We assume that 4. holds. Then the collection of composable 1-morphisms of $\mE$ is closed under composition.
Hence, for every $A,B,C \in X $ in $\mE$ the composition
functor $\Mor_X(B,C) \times \Mor_X(A,B) \to \Mor_X(A,C)$ sends $Y_{B,C} \times Y_{A,B} $ to $Y_{A,C} $ if and only if for every 1-morphisms $f,g: B \to C, f', g': A \to B $ in $\mE$ the induced functor 
$$ \xi: \Mor_{\Mor_X(B,C)}(f,g) \times \Mor_{\Mor_X(A,B)}(f',g') \to \Mor_{\Mor_X(A,C)}(f f', g g') $$ preserves cells of $\mE.$

The functor $\xi$ factors as
$$ \Mor_{\Mor_X(B,C)}(f,g) \times \Mor_{\Mor_X(A,B)}(f',g') \to \Mor_{\Mor_X(A,C)}(f g',g g') \times \Mor_{\Mor_X(A,C)}(f f',f g') \to $$$$ \Mor_{\Mor_X(A,C)}(f f', g g').$$
Hence 4. implies 3.
    
\end{proof}

In the following we study $\infty$-categories with no nontrivial equivalences. These behave especially simply and are very useful as many building blocks of $\infty$-categories are of this type.

\begin{definition}Let $\n \geq 1$.
An $\infty$-category is gaunt if it is strict and univalent.
\end{definition}

\begin{notation}
Let $\infty\Cat^\gaunt \subset \infty\Cat$ be the full subcategory of gaunt $\infty$-categories.		
\end{notation}

\begin{remark}\label{dimen}
Let $\n \geq 0$.
An $\infty$-category is gaunt if and only if for every $\n \geq 0$ the underlying $n$-category of $\mC$ is gaunt.
An $\n$-category is gaunt if and only if it satisfies the following inductive condition, which we call gaunt$'$.
A 1-category is gaunt$'$ if and only if every equivalence is an identity.
An $\n$-category is gaunt$'$ if and only if the underlying category and all morphism $\n$-1-categories are gaunt$'$.
\end{remark}

\begin{example}
For every gaunt $\infty$-category $\mC$ the categorical suspension $S(\mC)$ is gaunt.
To see this it suffices to observe that every non-identity morphism of $S(\mC)$ is not invertible since it goes from 0 to 1
and there is no morphism from 1 to 0 in $S(\mC).$

\end{example}

\begin{proposition}\label{gaunto}
Let $\mC$ be an $\infty$-category. The following are equivalent:
\begin{enumerate}[\normalfont(1)]\setlength{\itemsep}{-2pt}
\item The $\infty$-category $\mC$ is gaunt.

\item For every $\n \geq 0$ the underlying $\n$-category $\iota_\n(\mC)$ is an $(\n,\n)$-category. 

\item For every $\infty$-category $\mB$ the mapping space $\Map_{\infty\Cat}(\mB,\mC)$ is a set.

\item For every gaunt $\infty$-category $\mB$ the mapping space $\Map_{\infty\Cat}(\mB,\mC)$ is a set.

\item For every $\bk \geq 0$ the mapping space $\Map_{\infty\Cat}(\bD^\bk,\mC)$ is a set.

\item The $\infty$-category $\mC$ is local with respect to all projections $  (* \coprod_{*\coprod *} *) \times \mB 
\to \mB $ for $\mB \in \infty\Cat.$

\item The $\infty$-category $\mC$ is local with respect to all projections $ (* \coprod_{*\coprod *} *) \times \bD^\bk
\to \bD^\bk $ for $\bk \geq 0.$

\end{enumerate}	

\end{proposition}

\begin{proof}
We first prove that (1) and (2) are equivalent.
By \cref{dimen} the $\infty$-category $\mC$ is gaunt if and only if for every $\n \geq 0$ the underlying $\n$-category $\iota_\n(\mC)$ is gaunt.
An $\infty$-category $\mC$ satisifies (2) if and only if for every $\n \geq 0$ the underlying $\n$-category $\iota_\n(\mC)$ satisfies (2).
So we can assume that $\mC$ is an $\n$-category.
Since for every $\n \geq 1$ and $\X, \Y \in \mC$ there is an equivalence $\Mor_{\iota_\n(\mC)}(\X,\Y)\simeq \iota_{\n-1}(\Mor_{\mC}(\X,\Y))$, an $\n$-category $\mC$ satisfies (2) if and only if the underlying space of $\mC$ is a set and every morphism $\n$-1 category of $\mC$ satisfies (2). 
Thus by induction on $\n \geq 0$ it suffices to see that every gaunt space is a set.
A space is gaunt if and only if for every $\n \geq 0$ the $\n$-truncation of the space is gaunt.
Hence via the Postnikov tower it suffices to see that a gaunt $(\n,0)$-category is a set. In this case the statement follows immediately by induction on $\n \geq 0 $ from the fact that a groupoid is a set when every isomorphism is an identity.
This proves the equivalence between (1) and (2).

Condition (3) implies (4) and (4) implies (5). Condition (5) implies (3) since $\infty\Cat$ is generated under small colimits by the disks (\cref{theta}) and sets are closed under limits in the category of spaces.

Since $\infty\Cat^\gaunt \subset \infty\Cat^\strict$, condition (1) implies (4). We prove that (3) implies (1).
Condition (3) for $\mB=*$ gives that the underlying space of $\mC$ is a set. So every equivalence of
$\mC$ is an identity and $\mC$ is 1-gaunt.
We prove by induction on $\n \geq 1$ that (3) implies that the underlying $n$-category of $\mC$ is gaunt.
We have to see that for every $\X,\Y \in \mC$, the underlying $(n-1)$-category of the morphism $\infty$-category
$\Mor_\mC(\X,\Y)$ is gaunt. So by induction it suffices to show that for every $\infty$-category $\mB$ the space $\Map(\mB, \Mor_\mC(\X,\Y))$ is a set.
This space identifies with the space
\[
\Map_{\infty\Cat_{\partial \bD^1/}}(S(\mB), (\mC,\X,\Y)) \simeq \{(\X,\Y) \}\times _{\iota_0(\mC) \times \iota_0(\mC)}\Map_{\infty\Cat}(S(\mB), \mC),
\]
which is a pullback of sets by (3) and so a set.

We complete the proof by showing that (6) is equivalent to (3) and that (7) is equivalent to (5).
This follows immediately from the fact that a space $\X$ is a set if and only if it is local with respect to the map $* \coprod_{*\coprod *} *  \to *.$
Locality with respect to this map means that the canonical map $\X \to \X \times_{\X \times \X} \X$
induced by the diagonal map is an equivalence, which says that all loop spaces of $\X$ are contractible, in which case $\X$ is a set. 
\end{proof}

\begin{corollary}
	
The full subcategory $\infty\Cat^\gaunt \subset \infty\Cat$ is an accessible localization and so presentable.	
	
\end{corollary}

\begin{corollary}
Every right adjoint functor $\infty\Cat \to \infty\Cat$ preserves gaunt $\infty$-categories.	
	
\end{corollary}

\subsection{Steiner $\infty$-categories}
Of special importance is a class of gaunt $\infty$-categories
known as strong Steiner $\infty$-categories,
which are gaunt $\infty$-categories generated by a set of well-behaved loop-free generators.
We refer to the latter simply as Steiner $\infty$-categories since we will not work with the non-strong variant.

\begin{definition}\label{polygraph}
A strict $\infty$-category $\mO$ is a {\em polygraph} if for every $\n \geq 0$ there is set $\mE_n$ of $n$-morphisms of $\mO$, called the {\em atomic} $n$-morphisms, such that the following induced commutative square in $\infty\Cat^\strict$ is a pushout square:
\[
\xymatrix{
\coprod_{\mE_\n} \partial\bD^n \ar[d]\ar[r] & \iota_{n-1}(\mO) \ar[d]\\
\coprod_{\mE_\n} \bD^n \ar[r] & \iota_{n}(\mO).
}
\]
We call $\coprod_{n \geq 0}\mE_n $ a graded set of generators for $\mO.$
\end{definition}

\begin{remark}
    All of the objects of a strict $\infty$-category $\mO$ are technically atomic $0$-cells. Hence atomicity is not a useful notion in dimension zero.
\end{remark}

By \cite[Proposition 2.4]{ARA2023107313} every polygraph has a unique set $\mE$ of generators, which is precisely the set of non-trivial morphisms that are indecomposable, i.e. can only be factored by using trivial morphisms.
Moreover by \cite[Proposition 2.3]{ARA2023107313} the set $\mE$ composition-generates the polygraph in the evident sense.
By \cite[Definition 2.22]{ARA2023107313} there is the following definition of a strongly loop-free set of generators:

\begin{definition}
Let $\mC$ be a polygraph with set of atomic generators $\mE.$
The preorder relation $\leq_\bN$ on $\mE$ is the preorder generated by the condition that
$$ X \leq_\bN Y  $$
for $X, Y \in \mE $ if $X$ appears as a factor in the source of $Y$ or $Y$ appears as a factor in the target of $X.$ 
    
\end{definition}

\begin{definition}

An atomic set of generators of a polygraph is strongly loop-free if the preorder relation $\leq_\bN$ is a partial order, i.e. is also antisymmetric.

\end{definition}

The following definition is \cite[Theorem 2.30]{ARA2023107313} and it is shown \cite[Theorem 2.30]{ARA2023107313} that it agrees with Steiner's original definition \cite[Definition 2.2]{Steiner2004OmegacategoriesAC}:

\begin{definition}
A Steiner $\infty$-category is a polygraph whose set of generators is strongly loop-free.

\end{definition}

\begin{remark}

Steiner $\infty$-categories are gaunt.    
\end{remark}

\begin{remark}
What we call Steiner $\infty$-category is usually called strong Steiner $\infty$-category, and the name Steiner $\infty$-category is reserved for a weaker notion. Since we don't use this weaker notion, we will simplify the terminology and refer to these objects as Steiner $\infty$-categories.
\end{remark}

\begin{notation}
Let $\infty\Cat^\Steiner \subset \infty\Cat^\gaunt$ be the full subcategory of Steiner $\infty$-categories.	
\end{notation}

We have the following useful result of Campion \cite{campion2023inftyncategoricalpastingtheorem}:

\begin{proposition}\label{pasting1} Every pushout square 
$$
\xymatrix{
\mA \ar[r] \ar[d] & \mB \ar[d] \\ \mC \ar[r] & \mD 
}
$$
in $\infty\Cat^\strict$, where the top horizontal functor is an inclusion and the left vertical functor belongs to the saturated class generated by the inclusions $\partial \bD^n \subset \bD^n$ for $n \geq 0$, is a pushout square in $\infty\Cat.$

\end{proposition}

\begin{proof}
This is \cite[Theorem C]{campion2023inftyncategoricalpastingtheorem}.
\end{proof}

\begin{lemma}\label{cello}\label{orientdec}
Let $n \geq 0$ and $\mO$ a $n$-dimensional Steiner $\infty$-category
that has a unique non-invertible $n$-morphism.
The following induced commutative square in $\infty\Cat$ is a pushout square:
\[
\xymatrix{
\partial\bD^n\ar[d]\ar[r] & \iota_{n-1}(\mO) \ar[d]\\
\bD^n \ar[r] & \mO.
}
\]

\end{lemma}

\begin{proof}

The top horizontal functor is an inclusion by loopfreeness of Steiner $\infty$-categories. 
By \cref{pasting1} the commutative square is a pushout square in $\infty\Cat$ if it is a pushout square in $\infty\Cat^\strict$.
By \cref{polygraph} every Steiner $\infty$-category is a polygraph, i.e. that for every $i \leq n$ the following commutative square is a pushout square in $\infty\Cat^\strict$:
\[
\xymatrix{
\coprod_{A_i} \partial\bD^i \ar[d]\ar[r] & \iota_{i-1}(\mO) \ar[d]\\
\coprod_{A_i} \bD^i \ar[r] & \iota_{i}(\mO),
}
\]
where $A_i$ is the set of atomic generators of dimension $i.$
So the result follows from $i=n$ and the assumption that $A_i$ has precisely one object.
\end{proof}

We have the following useful result of Ara--Lucas \cite{ara.folkmodel}:

\begin{proposition}\label{pasting2}

Let $K$ be a category.
Every colimit diagram $\phi: K^\triangleright \to \infty\Cat^\Steiner$,
where for every morphism $f: X \to Y$  in $ K^\triangleright $ the induced functor $\phi(f): \phi(X) \to \phi(Y) $ preserves atomic generators, is already a colimit diagram in $\infty\Cat^\strict.$
    
\end{proposition}

\begin{proof}
This is \cite[Theorem 5.6]{ara.folkmodel}. 
\end{proof}

Steiner $\infty$-categories are useful since they can be uniquely built from chain complexes:

\begin{definition}

An augmented directed complex is a quadruple $(\A,\partial, \epsilon, \B)$ consisting of an augmented chain complex of abelian groups $(\A, \partial, \epsilon)$ concentrated in non-negative degrees 
and a $\bN$-graded submonoid $\B \subset \A$ such that $\epsilon(\B_0) \subset \bN.$
A map of augmented directed complexes is a map of augmented chain complexes that restricts to the graded submonoids.
\end{definition}
\begin{notation}

We write $\mathrm{ADC}$ for the category of augmented directed complexes.
\end{notation}

\begin{example}\label{hhht}
The final augmented directed complex is the final augmented chain complex $\bZ$
equipped with the submonoid $\bN \subset \bZ.$

\end{example}

By \cite[Definition 3.4, Definition 3.5, Definition 3.6]{Steiner2004OmegacategoriesAC} there is the notion of a strongly loop-free unital basis of an augmented directed complex.
\begin{definition}
	
An augmented directed complex is a Steiner complex if it admits a strongly loop-free unital basis.

\end{definition}

\begin{remark}
What we call Steiner complex is usually called strong Steiner complex, and the name Steiner complex is reserved for a weaker notion. Since we don't use this weaker notion, we decided for that terminology.
    
\end{remark}

\begin{notation}
Let $\mathrm{ADC}^\Steiner \subset \mathrm{ADC}$ be the full subcategory of Steiner complexes.		
	
\end{notation}

\begin{example}
By \cref{example_theta} every object of $\Theta$ is a Steiner $\infty$-category.
    
\end{example}

For every strict $\infty$-category $\mA$ and $\n \geq 0$ let $\mA_\n$ be the set of $\n$-morphisms of $\mA$ and $\circ_\n$ the $\n$-th composition operator
and $\mathrm{d}_\n^-,\mathrm{d}_\n^+$ the $\n$-th source and target operators.
The next proposition is \cite[Definition 2.4]{Steiner2004OmegacategoriesAC}:
\begin{proposition}
Let $\mA$ be a strict $\infty$-category.
There is an augmented directed complex $\lambda(\mA)$ such that $$\lambda(\mA)_\q= \bZ[\mA_\q]/ \{\X\circ_\p\Y -\X-\Y \mid \X,\Y \in \mA_\q, \p < \q \ \mathrm{such} \ \mathrm{that} \ \X\circ_\p\Y \ \mathrm{exists} \}$$ for every $\q \geq 0$.
The differential $\partial_{\q+1}: \lambda(\mA)_{\q+1} \to \lambda(\mA)_\q$
is induced by the map $\mA_{\q+1} \to \lambda(\mA)_\q, \X \mapsto \mathrm{d}^+_\q(\X)- \mathrm{d}^-_\q(\X).$
The augmentation $\epsilon: \lambda(\mA)_0 \to \bZ$ sends all objects of $\mA$ to 1.
The submonoid  $\B_\q \subset \lambda(\mA)_\q$ for $\q \geq 0$ is the submonoid generated by the images in $\lambda(\mA)_\q $ of the $\q$-morphisms of $\mA$.

\end{proposition}

We obtain a functor $\lambda:\infty\Cat^{\mathrm{strict}} \to \mathrm{ADC}$ from strict $\infty$-categories to augmented directed complexes.

The next theorem is \cite[Proposition 4.6, Theorem 5.6, Theorem 5.11]{Steiner2004OmegacategoriesAC} and \cite[Theorem 1.28] {ARA2023107313}:
\begin{theorem}\label{hh}
The functor $\lambda:\infty\Cat^{\mathrm{strict}} \to \mathrm{ADC}$
admits a right adjoint
$\nu: \mathrm{ADC} \to \infty\Cat^{\mathrm{strict}}$.
The resulting adjunction $\lambda:\infty\Cat^{\mathrm{strict}} \rightleftarrows \mathrm{ADC}:\nu$
restricts to an equivalence
$\infty\Cat^{\Steiner} \simeq \mathrm{ADC}^\Steiner$
on the full subcategories of Steiner $\infty$-categories and Steiner complexes.
\end{theorem}

\begin{definition}
We say that a Steiner $\infty$-category is finitely generated if its associated chain complex is finitely generated, in the sense that it is a bounded complex of finitely generated abelian groups.
We write $\infty\Cat^{\Steiner}_\mathrm{fg}\subset\infty\Cat^{\Steiner}$ for the full subcategory of finitely generated Steiner $\infty$-categories.
\end{definition}

\begin{notation}
The tensor product of augmented directed complexes $(\A, \partial, \epsilon, \B), (\A', \partial', \epsilon', \B')$ is the augmented directed complex $$(\A, \partial, \epsilon, \B) \ot (\A', \partial', \epsilon', \B') = (\A \ot \A', \partial^{\A \ot \A'}, \epsilon \ot \epsilon', \B \boxtimes \B'),$$
where $(\A \ot \A', \partial^{\A\ot\A'})$ is the tensor product of chain complexes of abelian groups
and $\B \boxtimes \B'$ is the image of the map $\B \ot \B' \to \A \ot \A'$ from the tensor product of graded commutative monoids. 

\end{notation}

\begin{remark}\label{Graytensor}
    
By \cite[Example 3.10]{Steiner2004OmegacategoriesAC} the tensor product of augmented directed complexes defines a monoidal structure on $\mathrm{ADC}$ whose tensor unit is the final object.
This monoidal structure on $\mathrm{ADC}$ is not symmetric and by \cite[Example 3.10]{Steiner2004OmegacategoriesAC} restricts to the full subcategory $\mathrm{ADC}^\Steiner$.
\end{remark}

\begin{definition}

The Gray monoidal structure on $\infty\Cat^\Steiner$
is the monoidal structure corresponding via the equivalence $\lambda: \infty\Cat^\Steiner \simeq \mathrm{ADC}^\Steiner$ of \cref{hh} to the monoidal structure of Steiner complexes
of \cref{Graytensor}.
We write $\boxtimes$ for the Gray tensor product of Steiner $\infty$-categories.

\end{definition}

\begin{remark}

The tensor unit of the Gray monoidal structure is the final $\infty$-category.
    
\end{remark}

\begin{definition}
Let $n \geq 0$. The oriented $n$-cube $\cube^n$ is the
$\n$-fold Gray tensor product $ (\cube^1)^{\boxtimes \n}$.

\end{definition}

\begin{notation}\label{sqGray}

Let $\cube \subset \infty\Cat^\Steiner$ be the full subcategory of oriented cubes.
\end{notation}

\begin{definition}Let $\A, \B$ be augmented directed complexes.
The join complex of $\A, \B$ is $$\A \diamond \B := \A \ot \lambda(\bD^1) \ot \B \coprod_{\A\ot \lambda(\partial\bD^1)\ot\B} \A \oplus \B.$$
\end{definition}

\begin{remark}\label{suspmod2}
There is an identity $$ \A \diamond \B = \A \oplus \A \ot \B[1] \oplus \B $$
as graded abelian groups respecting the graded submonoids.

\end{remark}

\begin{remark}\label{jointensor}
    
By \cite[7.1]{ara.folkmodel} the join defines a (non-symmetric) monoidal structure on the category $\mathrm{ADC}$ of augmented directed complexes, which by \cite[Corollaire 6.21]{Dimitri_Ara_2020} 
restricts to the full subcategory $\mathrm{ADC}^\Steiner$ of Steiner complexes. 
\end{remark}

\begin{definition}

The join monoidal structure on $\infty\Cat^\Steiner$
is the monoidal structure corresponding via the equivalence $\lambda: \infty\Cat^\Steiner \simeq \mathrm{ADC}^\Steiner$ of \cref{hh} to the join monoidal structure of Steiner complexes
of \cref{Graytensor}.
We write $\star$ for the join of Steiner $\infty$-categories.
\end{definition}

\begin{remark}

The tensor unit of the join monoidal structure is the initial $\infty$-category.
    
\end{remark}

\begin{definition}Let $\n \geq 0.$ The $\n$-th oriental $\bbDelta^\n$ is 
the $\n+1$-fold join $ \bD^0 \star \cdots \star \bD^0.$
\end{definition}

\begin{notation}	
Let $\bbDelta \subset \infty\Cat^\Steiner$ be the full subcategory of orientals.
Let $\bbDelta_{\geq -1} \subset \infty\Cat$ be the full subcategory of orientals and the empty $\infty$-category.
\end{notation}

\begin{remark}
Let $\n \geq 0.$ The $n$-th oriental corresponds to the $\n$-dimensional Steiner complex
$$ 	(... \to 0 \to ... \to 0 \to \bigoplus_{[\bk] \hookrightarrow [\n]} \bZ \to ... \to \bZ^{\n+1}),$$
which arises by applying the join iterately. The group of $\bk$-chains is free on the set of order preserving injections $[\bk] \to [\n]$. The differentials are the alternating sums of face maps. The augmentation is the sum.
Thus 
$\lambda(\bbDelta^\n)$ is the complex of normalized integral chains on the $n$-simplex.
\end{remark}

\subsection{Thetas are retracts of orientals}

We will prove the following theorem:

\begin{theorem}\label{retract}
Every object of $\Theta'$ is a retract in ${\infty\Cat_{\partial\bD^1/}}$ of a bipointed oriented simplex.
\end{theorem}

\begin{notation}Let $\n, \m \geq 1$.
The canonical functor $$\alpha: \bbDelta^\n \to \bbDelta^{\n}*\bbDelta^{\m-1}= \bbDelta^{\n+\m}$$ sends $0$ to $0$ and $\n$ to $\n$ and the canonical functor $$\beta: \bbDelta^\m \to \bbDelta^{\n-1}*\bbDelta^{\m}= \bbDelta^{\n+\m}$$ 
sends $0$ to $\n$ and $\m$ to $\n+\m$ and so induce a bipointed functor $$\alpha+\beta: \bbDelta^\n \vee \bbDelta^\m \to \bbDelta^{\n+\m}.$$

\end{notation}

\begin{notation}
Let $\mA \in {\infty\Cat^\Steiner_{\partial\bD^1/}}$.
\begin{enumerate}[\normalfont(1)]\setlength{\itemsep}{-2pt}
\item Let $$p_\mA: \mA \boxtimes \bD^1 \to \mA* \bD^0 \simeq (\mA \boxtimes \bD^1) \coprod_{\mA \boxtimes \{1\}}\bD^0$$
be the quotient functor, which is a morphism of $ {\infty\Cat_{\partial\bD^1/}}$.	

\item Let
$$\ell_A: \mA* \bD^0 \simeq (\mA \boxtimes \bD^1) \coprod_{\mA \boxtimes \{1\}}\bD^0 \to S(\mA)\simeq  \bD^0 \coprod_{\mA\boxtimes\{0\}} (\mA \boxtimes \bD^1)\coprod_{\mA \boxtimes\{1\}} \bD^0$$ be the quotient functor,
which is a morphism of ${\infty\Cat_{\partial\bD^1/}}$.	

\item Let
$$q_A:= \ell_\mA \circ p_\mA: \mA \boxtimes \bD^1 \to \mA* \bD^0\to S(\mA)$$ be the composition in ${\infty\Cat_{\partial\bD^1/}}$.	
\end{enumerate}

\end{notation}

To prove \cref{retract} we prove the following results.

\begin{proposition}\label{10}
Let $\n,\m \geq 0$. The functor
$$\zeta_{\n,\m}:=\alpha +\beta: \bbDelta^\n \vee \bbDelta^\m \to \bbDelta^{\n+\m}$$
in ${\infty\Cat_{\partial\bD^1/}}$ admits a left inverse $\theta$ in ${\infty\Cat_{\partial\bD^1/}}.$

\end{proposition}

\begin{proof}We proceed by induction on $\m.$
If $\m=0$, the functor $\zeta_{\n,\m}$
identifies with the identity of $\bbDelta^\n.$ So there is nothing to show.	
Next we prove the case $\m=1$ by induction on $\n.$
Assume the statement holds for $\m=1$ and a given $\n \geq 0.$
The functor
$\zeta_{\n+1,\m}$
factors as the following morphisms in $\Cat_{\infty*,*}:$
$$\bbDelta^{\n+1} \coprod_{\bD^0} \bbDelta^{\m} \xrightarrow{\gamma:=\id_{\bbDelta^{\n+1}} \coprod_{\{1\}} (\{1\}\boxtimes \id_{\bbDelta^\m})} \bbDelta^{\n+1} \coprod_{\bD^1} \bbDelta^{\m+1} \cong \bD^0 \star (\bbDelta^{\n} \vee \bbDelta^{\m}) \xrightarrow{\id_{\bD^0}\star \zeta_{\n,\m}} \bbDelta^{\n+\m+1}.$$
So it is enough to see that the first functor $\gamma$ in the composition
admits a left inverse $\tau$ in ${\infty\Cat_{\partial\bD^1/}}$ if $\m=1.$

Let $\kappa: \bbDelta^2 \to \tau_{\leq1} \bbDelta^2=\bD^1 \vee \bD^1 $ be the
truncation functor.
Then $\kappa$ is a left inverse of the canonical functor 
$\bD^1 \vee \bD^1 \to \bbDelta^{2}$ in ${\infty\Cat_{\partial\bD^1/}}.$
Hence $ \kappa \circ (\id_{\bD^1}\boxtimes\{-\}):\bD^1 \to \bD^1 \vee \bD^1$ is the inclusion to the first summand and
$ \kappa \circ (\{+\}\boxtimes \id_{\bD^1}):\bD^1 \to \bD^1 \vee \bD^1$ is the inclusion to the second summand.

Let $\tau: \bbDelta^{\n+1} \coprod_{\bD^1} \bbDelta^{2} \to \bbDelta^{\n+1} \vee \bbDelta^{1}$ be the functor
that is the identity on $\bbDelta^{\n+1}$ and the functor
$\bbDelta^{2} \xrightarrow{\kappa} \bD^1 \vee \bD^1 \xrightarrow{(\id_{\bbDelta^{1}} \star \{\n\})\vee \id_{\bD^1}} \bbDelta^{\n+1} \vee \bbDelta^{1}$ on $\bbDelta^{2}.$
This determines $\tau$ since the composition $ \bD^0 \star \{0\} \xrightarrow{\bD^0 \star\{0\}} \bbDelta^2 \xrightarrow{\kappa} \bD^1 \vee \bD^1 $
is the inclusion to the first summand.
The functor $\tau$ is a left inverse of $\gamma$
since $ \kappa \circ (\{1\}\star \id_{\bD^0}):\bD^1 \to \bD^1 \vee \bD^1$ is the inclusion to the second summand.
Consequently, by induction on $\n$ the statement holds for $\m=1.$

Now we prove the general case by induction on $\m.$ Assume the statement holds for a given $\m \geq 0$ and every $\n \geq 0.$
The functor
$\zeta_{\m+1,\n}$ factors as the following morphisms in ${\infty\Cat_{\partial\bD^1/}}:$
$$ \bbDelta^\n \coprod_{\bD^0} \bbDelta^{\m+1} \xrightarrow{\delta:=(\id_{\bbDelta^\n}\star\{0\}) \coprod_{\{0\}}\id_{\bbDelta^{\m+1}}} \bbDelta^{\n+1} \coprod_{\bD^1} \bbDelta^{\m+1} \cong (\bbDelta^\n \vee \bbDelta^{\m})\star \bD^0 \xrightarrow{\zeta_{\n,\m}\star \id_{\bD^0}} \bbDelta^{\n+\m+1}.$$
Consequently, it is enough to see that the first functor $\delta$ in the composition
admits a left inverse $\rho$ in ${\infty\Cat_{\partial\bD^1/}}$.
By the case $\m=1$, which we have proven above, the functor
$\zeta_{n,1} $ admits a left inverse $\kappa.$ 

We define $\rho: \bbDelta^{n+1} \coprod_{\bD^1} \bbDelta^{\m+1} \to \bbDelta^n \vee \bbDelta^{\m+1}$ to be the functor
that is the identity on $\bbDelta^{\m+1}$ and the functor
$$\bbDelta^{n+1} \xrightarrow{\kappa} \bbDelta^n \vee \bD^1 \xrightarrow{\id_{\bbDelta^n} \vee (\{0\} \star \id_{\bD^0})} \bbDelta^n \vee \bbDelta^{\m+1}$$ on $\bbDelta^{n+1}.$
Then $\rho$ is a left inverse of $\delta$ in ${\infty\Cat_{\partial\bD^1/}}$ because the functor
$\kappa\circ (\id_{\bbDelta^n}\star\{0\}): \bbDelta^n \to \bbDelta^n \vee \bD^1$ is the inclusion to the first summand.
\end{proof}

\begin{lemma}\label{hio} For every Steiner $\infty$-category $\mC$ the functor $\varrho_{\mC}: \bD^1 \boxtimes S(\mC) \to S(\bD^1 \boxtimes \mC)$ admits a section in ${\infty\Cat_{\partial\bD^1/}}$.
\end{lemma}

\begin{proof}
	
The functor $\varrho_{\mC}$ admits a section if and only if the induced map of Steiner complexes admits a section.
Thus we may work instead with the Steiner complex $A$ associated to $\mC.$

The map on Steiner complexes is the quotient map $$\rho: \lambda(\bD^1) \ot S(A)\cong S(A)[1] \oplus S(A) \oplus S(A)\cong A[2] \oplus \bZ[1] \oplus \bZ[1] \oplus A[1] \oplus \bZ \oplus \bZ \oplus A[1] \oplus \bZ \oplus \bZ $$$$ \to S(\lambda(\bD^1) \ot A) \cong \lambda(\bD^1) \ot A[1] \oplus \bZ \oplus \bZ \cong A[2] \oplus A[1] \oplus A[1] \oplus \bZ \oplus \bZ$$
induced by the map $ \epsilon + \epsilon: \bZ^4 \to \bZ^2.$
The first differential of $\lambda(\bD^1)\ \ot S(A)$ is the map
$$A_0 \oplus A_0 \oplus \bZ \oplus \bZ 
\to \bZ\oplus \bZ \oplus \bZ \oplus \bZ $$
sending $(\X,\Y,n,m)$ to $(\epsilon(X)+n, -\epsilon(X)+\m, \epsilon(Y)-n, -\epsilon(Y)-\m) $
and the second differential is the map
$$ A_1 \oplus A_1 \oplus A_0 \to A_0 \oplus A_0 \oplus \bZ \oplus \bZ  $$
sending $(\X,\Y,\Z)$ to $(\partial_1(\X)+\Z,\partial_1(\Y)-\Z,-\epsilon(\Z),\epsilon(\Z)).$
On the other hand, the first differential of $ S(\lambda(\bD)^{1} \ot A)$ is the map
$$(\epsilon + \epsilon, -\epsilon-\epsilon): A_0 \oplus A_0 \to \bZ \oplus \bZ $$
and the second differential of $S(\lambda(\bD)^{1} \ot A) $ is the first differential of $\lambda(\bD)^{1} \ot A $, which is the map 
$$ A_1 \oplus A_1 \oplus A_0 \to A_0 \oplus A_0 $$
sending $(\X,\Y,\Z)$ to $(\partial_1(\X)+\Z,\partial_1(\Y)-\Z).$

We define a $\bN$-graded map $$\phi: S(\lambda(\bD^1)\ \ot A)\to \lambda(\bD^1)\ot S(A) $$
that is in degree $\bk > 1$ the identity of 
$ A_{\bk-1} \oplus A_{\bk-1} \oplus A_{\bk-2}$,
in degree $1$ the map $$ A_{0} \oplus A_{0} \to A_{0} \oplus A_{0} \oplus \bZ \oplus \bZ$$
sending $(\X,\Y)$ to $(\X,\Y,\epsilon(\Y), \epsilon(\X))$
and in degree 0 the map
$$\bZ\oplus \bZ \to \bZ \oplus \bZ\oplus \bZ \oplus \bZ: (\X,\Y) \mapsto (\X,0,0,\Y).$$
The $\bN$-graded map preserves differentials since for every $(\X,\Y,\Z) \in  
A_1 \oplus A_1 \oplus A_0 \cong S(\lambda(\bD)^{1}\ot A)_2 $
we have 
\begin{align*}
\partial_2(\phi_2(\X,\Y,\Z))&=(\partial_1(\X)+\Z,\partial_1(\Y)-\Z,-\epsilon(\Z),\epsilon(\Z))\\
&=(\partial_1(\X)+\Z,\partial_1(\Y)-\Z,\epsilon(\partial_1(\Y)-\Z),\epsilon(\partial_1(\X)+\Z))\\
&=\phi_1(\partial_1(\X)+\Z,\partial_1(\Y)-\Z) = \phi_1(\partial_2(\X,\Y,\Z))
\end{align*}
and for every $(\X,\Y) \in  A_{0} \oplus A_{0}\cong S(\lambda(\bD)^{1}\oplus A)_1$
we have
\begin{align*}
\partial_1(\phi_1(\X,\Y))&=\partial_1(\X,\Y,\epsilon(\Y), \epsilon(\X)) 
=(\epsilon(X)+\epsilon(Y), -\epsilon(X)+\epsilon(X), \epsilon(Y)-\epsilon(Y), -\epsilon(Y)-\epsilon(X))\\
&=(\epsilon(\X)+\epsilon(\Y), 0 ,0, -\epsilon(\X)-\epsilon(\Y))
= \phi_0(\epsilon(\X)+\epsilon(\Y), -\epsilon(\X)-\epsilon(\Y)) = \phi_0(\partial_1(\X,\Y)).
\end{align*}
This provides an explicit section on the level Steiner complexes.
\end{proof}

\begin{proposition}\label{6}
Let $n \geq 0$. The functor 
$ q_{\cube^n}: \cube^{n+1} \to S(\cube^{n}) $
admits a section in ${\infty\Cat_{\partial\bD^1/}}$.	
\end{proposition}

\begin{proof}
We proceed by induction on $n \geq 0.$
For $n=0$ there is nothing to show since the functor $ q_{\cube^n}$ identifies with the identity.
Assume the statement holds for $n.$
The functor 
$ q_{\cube^{n+1}}: \cube^{n+2} \to S(\cube^{n+1}) $
factors as the following morphisms in ${\infty\Cat_{\partial\bD^1/}}:$ $$ \cube^{n+2} \xrightarrow{ \bD^1 \boxtimes q_{\cube^n}} \bD^1 \boxtimes S(\cube^{n}) \xrightarrow{\varrho_{\cube^n}} S(\cube^{n+1}).$$
So by induction hypothesis we are reduced to show that the functor $\varrho_{\cube^n}$ admits a section in ${\infty\Cat_{\partial\bD^1/}}$ for every $n \geq0.$
This follows from \cref{hio}.
\end{proof}

\begin{corollary}\label{11}
Let $n \geq 0$. The functor 
$ \ell_{\bbDelta^n}: \bbDelta^{n+1} \to S(\bbDelta^{n}) $
admits a section in ${\infty\Cat_{\partial\bD^1/}}$.	
\end{corollary}

\begin{proof}
There is a commutative square
\[
\xymatrix{
\cube^{n+1} \ar[r]^{q_{\cube^n}} \ar[d] & S(\cube^{n}) \ar[d]\\
\bbDelta^{n+1}  \ar[r]^{\ell_{\bDelta^n}} &S(\bbDelta^{n}). 
}
\]	

By \cref{6} the functor $q_{\cube^n}$ has a section. By \cref{Thh} the right vertical functor of the square has a section. 
Hence also $ \ell_{\bDelta^n}$ has a section given by the composition of sections
$ S(\bbDelta^{n}) \to S(\cube^{n}) \to \cube^{n+1}$ followed by the left vertical functor
$ \cube^{n+1} \to \bbDelta^{n+1}.$	
\end{proof}

\begin{proof}[Proof of \cref{retract}]
Let $\bDelta' \subset \infty\Cat_{\partial\bD^1/ }$
be the full subcategory of orientaled simplices bipointed at the leftmost and rightmost object.
The closure $\overline{\bDelta'} \subset \infty\Cat_{\partial\bD^1/ }$ of $\bDelta' \subset \infty\Cat_{\partial\bD^1/ }$ under retracts contains the final $\infty$-category and is closed under bipointed wedges (\cref{10}) and suspensions  (\cref{11}) and so by definition of $\Theta' $ (\cref{Theta}) contains
$\Theta'.$
\end{proof}

We obtain the following:

\begin{theorem}\label{orientaldense}
The full subcategory $\bDelta$ is dense in $\infty\Cat$.
\end{theorem}

\begin{proof}
By \cref{retract} the idempotent completion of $\bDelta$
contains $\Theta.$
By \cref{theta} the full subcategory $\Theta$ is dense in $\infty\Cat$ and therefore also every full subcategory of $\infty\Cat$ whose idempotent completion contains $\Theta,$ is dense in $\infty\Cat.$ 
\end{proof}

\subsection{Orientals are retracts of oriented cubes}

\begin{notation}
Let $s: \bDelta^2 \to \cube^2$ be the functor
corresponding to the maps $(0,0) \to (0,1), (0,1) \to (1,1)$ and $(0,0) \to (1,0) \to (1,1)$
and the unique 2-morphism from the composite map $(0,0) \to (0,1) \to (1,1)$ to the composite map $(0,0) \to (1,0) \to (1,1)$.

\end{notation}

\begin{remark}

The functor $s: \bDelta^2 \to \cube^2$ is a section of the functor
$q: \cube^2= \bD^1 \boxtimes \bD^1 \to \bD^1 \star \bD^0 = \bDelta^2.$

\end{remark}

\begin{remark}\label{remar}
Recall that any functor of Steiner $\infty$-categories uniquely determines, and is determined by, a map on the level of Steiner complexes.
The functor $q:\cube^2 \to \bDelta^2$ corresponds to the chain map
$$ \lambda(\bD^1) \ot \lambda(\bD^1) \cong \bZ[2] \oplus \bZ^4[1] \oplus \bZ^4 \to \lambda(\bDelta^2) \cong \bZ[2] \oplus \bZ^3[1] \oplus \bZ^3  $$
that is the identity in degree 2, the projection $\bZ^4 = \bZ^3 \oplus \bZ \to \bZ^3$ in degree 1 and  the map $\bZ^4 = \bZ^2 \oplus \bZ^2 \to \bZ^2 \oplus \bZ= \bZ^3$ in degree 0 induced by the codiagonal map $  \bZ^2 \to \bZ.$ 
The functor $s: \bDelta^2 \to \cube^2$ corresponds to the chain map
$$\xi: \lambda(\bDelta^2) \cong \bZ[2] \oplus \bZ^3[1] \oplus \bZ^3 \to \lambda(\bD^1) \ot \lambda(\bD^1) \cong \bZ[2] \oplus \bZ^4[1] \oplus \bZ^4$$
that is the identity in degree 2, the canonical inclusion in degree 0 and in degree 1 the map
$ \bZ^3 = \bZ^2 \oplus \bZ \to \bZ^4 = \bZ^2 \oplus \bZ^2  $ induced by the diagonal map
$ \bZ \to \bZ^2.$

\end{remark}

\begin{notation}

Let $h	:= s \circ q: \bD^1\boxtimes\bD^1\to\bD^1\boxtimes\bD^1$ be the composition.

\end{notation}

\begin{remark}
Consider the functors $f,g:\bD^1\to\bD^1\boxtimes\bD^1$ selecting the identity $\id_{(0,0)}$ and $(0,1)\to (1,1)$, respectively.
Then $h$ determines a lax natural transformation from $f$ to $g$.
\end{remark}

\begin{notation}

For $X$ an $\infty$-category let $h_X=h\boxtimes\id_X$ be the induced endomorphism of $\bD^1\boxtimes\bD^1\boxtimes X$.
\end{notation}

\begin{remark}\label{facto}
The quotient functor $q_X: \bD^1\boxtimes\bD^1\boxtimes X \to  \bD^0\star(\bD^0\star X)$
factors as $$\bD^1\boxtimes\bD^1\boxtimes X \xrightarrow{q \boxtimes X}\bDelta^2 \boxtimes X \simeq (\bD^0 \coprod_{\{0\} \boxtimes \bD^1} \bD^1 \boxtimes \bD^1) \boxtimes X \simeq X \coprod_{\{0\} \boxtimes \bD^1 \boxtimes X} \bD^1 \boxtimes \bD^1 \boxtimes X \to$$$$ \bD^0 \star (\bD^1 \boxtimes X) \simeq \bD^0 \coprod_{\{0\}\boxtimes \bD^1 \boxtimes X} \bD^1 \boxtimes \bD^1 \boxtimes X \to \bD^0\star(\bD^0\star X) .$$
By definition, the functor $h_X: \bD^1\boxtimes\bD^1\boxtimes X \to \bD^1\boxtimes\bD^1\boxtimes X $
fits into a commutative triangle:
$$
\xymatrix{
\bD^1\boxtimes\bD^1\boxtimes X\ar[rr]^{h_X}\ar[rd]_{q \boxtimes X} && \bD^1\boxtimes\bD^1\boxtimes X\ar[ld]^{q \boxtimes X}\\ & \bDelta^2 \boxtimes X
}
$$
and so also into a commutative triangle:
$$
\xymatrix{
\bD^1\boxtimes\bD^1\boxtimes X\ar[rr]^{h_X}\ar[rd]_{q_X} && \bD^1\boxtimes\bD^1\boxtimes X\ar[ld]^{q_X}\\ & \bD^0\star(\bD^0\star X).
}
$$

\end{remark}

\begin{proposition}Let $X$ be an $\infty$-category.
The square
$$
\xymatrix{
\bD^1\boxtimes\bD^0\boxtimes X\ar[rr]^{\bD^1 \boxtimes d_1\boxtimes X}\ar[d]_{\bD^1 \boxtimes p_{\bD^0}} && \bD^1\boxtimes\bD^1\boxtimes X\ar[d]^{(\bD^1 \boxtimes q_X)\circ h_X}\\
\bD^1\boxtimes\bD^0\ar[rr]^{\bD^1 \boxtimes i_{\bD^0}} && \bD^1\boxtimes(\bD^0\star X)
}
$$
commutes and so induces a map $e_X\colon\bD^1\boxtimes(\bD^0\star X)\to\bD^1\boxtimes(\bD^0\star X)$ that fits into a commutative triangle:
\begin{equation}\label{tri}
\xymatrix{
\bD^1\boxtimes(\bD^0\star X)\ar[rr]^{e_X}\ar[rd]_{q_{\bD^0\star X}} && \bD^1\boxtimes(\bD^0\star X)\ar[ld]^{q_{\bD^0\star X}} \\ & \bD^0\star(\bD^0\star X).}
\end{equation}
The resulting square
$$
\xymatrix{
\bD^0\boxtimes(\bD^0\star X)\ar[rr]^{d_1\boxtimes(\bD^0 \star X)}\ar[d]^{p_{\bD^0}} && \bD^1\boxtimes(\bD^0\star X)\ar[d]^{e_X}\\
\bD^0\boxtimes\bD^0\ar[rr]_{d_1\boxtimes i_{\bD^0}} && \bD^1\boxtimes(\bD^0\star X)
}
$$
commutes and so induces a map $s_X\colon\bD^0\star(\bD^0\star X)\to\bD^1\boxtimes(\bD^0\star X)$
that fits into a commutative triangle:
\begin{equation}\label{tri2}
\xymatrix{
\bD^1\star(\bD^0\star X)\ar[rr]^{s_X}\ar[rd]_{=} && \bD^1\boxtimes(\bD^0\star X)\ar[ld]^{q_{\bD^0\star X}} \\ & \bD^0\star(\bD^0\star X).}
\end{equation}

\end{proposition}
\begin{proof}

To see that the first square commutes, it suffices to check that the corresponding square of augmented directed complexes
\[
\xymatrix{
A[1] \oplus A \oplus A \ar[rrr] \ar[dd]^= &&&  A[2] \oplus A[1] \oplus A[1] \oplus A[1] \oplus A \oplus A \oplus  A[1] \oplus A \oplus A\ar[d]\\
&&& A[2] \oplus A[1] \oplus A[1] \oplus A[1] \oplus A \oplus A \oplus 0 \oplus A   \ar[d]^{A \ot \xi}
\\
A[1] \oplus A \oplus A  \ar[rrr] \ar[d] &&& A[2] \oplus A[1] \oplus A[1] \oplus A[1] \oplus A \oplus A \oplus  A[1] \oplus A \oplus A   \ar[d]  \\
\bZ[1] \oplus \bZ \oplus \bZ \ar[rrr] &&& A[2] \oplus A[1] \oplus \bZ[1] \oplus \A[1] \oplus A \oplus \bZ \oplus  A[1] \oplus A \oplus \bZ
}
\]
commutes.
The upper and middle horizontal map is induced by the map $A =  0 \oplus 0 \oplus A \to  A[1] \oplus A \oplus A$ and the lower horizontal map is induced by the map $\bZ \to A[1] \oplus A \oplus \bZ.$ 
The lower left vertical map and right lower vertical map are induced by the map $A \to \bZ.$
Consequently, the bottom square commutes.
The upper right vertical map is induced by the codiagonal map $ A\oplus A \to A$ and the map $A[1] \to 0$. 
The right middle vertical maps uses the map $\xi$ of \cref{remar}.
We prove that the top square commutes.
Hence the upper horizontal map in the diagram sends $(X,Y,Z) \in A[1] \oplus A \oplus A$
to $(0,0,X,0,0,Y,0,0,Z)$ that is sent by the upper right vertical map to $(0,0,X,0,0,Y,0+Z)$ that is sent by the middle right vertical map to $(0,0,X,0,0,Y, 0,0, Z)$.

The induced map $e_X$ induces on Steiner complexes the map 
$$\kappa: A[2] \oplus A[1] \oplus \bZ[1] \oplus \A[1] \oplus A \oplus \bZ \oplus  A[1] \oplus A \oplus \bZ \to A[2] \oplus A[1] \oplus \bZ[1] \oplus \A[1] \oplus A \oplus \bZ \oplus  A[1] \oplus A \oplus \bZ$$
sending $(T, X, n, Y,Z, m, O,V,k)$ to $ (T, X, n, Y,Z, m,Y, 0, \epsilon(V+\eta(k))). $
\cref{facto} implies that there is a commutative triangle
$$
\xymatrix{
\bD^1\boxtimes\bD^1\boxtimes X\ar[rr]^{(\bD^1 \boxtimes q_X) \circ h_X}\ar[rd]_{q_X} && \bD^1\boxtimes(\bD^1\star X) \ar[ld]^{q_{\bD^1 \star X}}\\ & \bD^0\star(\bD^0\star X).
}
$$
This implies by the universal propert of the pushout that there is a commutative triangle (\ref{tri}).

Similarly, for the second square, it suffices to check that the induced square of augmented directed complexes
\[
\xymatrix{A[1] \oplus A \oplus \bZ \ar[r]\ar[d] & A[2] \oplus A[1] \oplus \bZ[1] \oplus A[1] \oplus A \oplus \bZ \oplus A[1] \oplus A \oplus \bZ \ar[d] \\
\bZ \ar[r] & A[2] \oplus A[1] \oplus \bZ[1] \oplus A[1] \oplus A \oplus \bZ \oplus A[1] \oplus A \oplus \bZ
}
\]
commutes, where the top horizontal map is the inclusion into the last three summands, the  
bottom horizontal map is the inclusion into the last summand and the left vertical map factors as the projection $ A[1] \oplus A \oplus \bZ \to A \oplus \bZ$ followed by the map $ \epsilon+\id_\bZ: A \oplus \bZ \to \bZ$
induced by the augmentation $\epsilon:  A  \to \bZ$.
Hence the square commutes as $\kappa$ sends $(0, 0, 0, 0,0, 0, 0,V,k)$ to $ (0, 0, 0, 0,0, 0,0, 0, \epsilon(V+\eta(k))=\epsilon(V)+k).$
The universal property of the quotient $q_{\bD^0 \star X}$ and the commutativity of triangle (\ref{tri}) imply the triangle (\ref{tri2}) commutes.
\end{proof}

\begin{corollary}\label{sec}For every $n \geq 0$ the quotient functor
$ \bD^1\boxtimes \bDelta^n \to \bDelta^{n+1}$
admits a section.

\end{corollary}

\begin{notation}
For every $\n \geq 0$ we inductively define a map $\xi_\n: \cube^\n \to \bbDelta^\n$ in ${\infty\Cat_{\partial\bD^1/}}$ starting with $\xi_0=\id$.	
Assume we have constructed $\xi_\n: \cube^\n \to \bbDelta^\n$ in ${\infty\Cat_{\partial\bD^1/}}.$
Then we define $\xi_{\n+1}: \cube^{\n+1} \to \bbDelta^{\n+1}$ in ${\infty\Cat_{\partial\bD^1/}}$
as $$ \cube^{\n+1} =  \cube^{\n} \boxtimes \bD^1  \xrightarrow{\xi_\n\boxtimes \bD^1}  \bbDelta^\n \boxtimes \bD^1 \xrightarrow{p_{\bbDelta^\n}} \bbDelta^\n \star \bD^0 = \bbDelta^{\n+1}.$$
\end{notation}

We obtain the following theorem as a corollary:
\begin{theorem}\label{Thh}

Let $\n \geq 0$.
The functor $\xi_\n: \cube^\n \to \bbDelta^\n$ admits a section. 
\end{theorem}

\begin{proof}
We proceed by induction on $\n \geq 0.$
For $\n=0$ there is nothing to show.
We assume the statement holds for $n.$
To see the statement for $n+1$ it suffices to show that the map $$ p_{\bbDelta^\n}: \bbDelta^\n \boxtimes \bD^1 \xrightarrow{} \bbDelta^\n * \bD^0 = \bbDelta^{\n+1}$$ admits a section.
This holds by \cref{sec}.
\end{proof}

The following theorem is due to Campion and proven in \cite{campion2022cubesdenseinftyinftycategories}.
It is a corollary of \cref{Thh} and \cref{theta}:

\begin{theorem}\label{cubicaldense}

The full subcategory $\cube$ is dense in $\infty\Cat$.

\end{theorem}

\begin{proof}
By \cref{Thh} the idempotent completion of $\cube$
contains $\bDelta.$
By \cref{orientaldense} the full subcategory $\bDelta$ is dense in $\infty\Cat$ and therefore also every full subcategory of $\infty\Cat$ whose idempotent completion contains $\bDelta,$ is dense in $\infty\Cat.$ 
\end{proof}

\section{\mbox{$\infty$-Categories as oriented spaces}}

\subsection{Oriented polytopes}

\begin{notation}
Let $\mO$ be a polygraph.
Then $\mO$ admits a filtration $\mO=\colim_n\iota_n(\mO)$ such that
\[
\xymatrix{
\coprod_{A_n} \partial\bD^n \ar[d]\ar[r] & \iota_{n-1}(\mO) \ar[d]\\
\coprod_{A_n} \bD^n \ar[r] & \iota_{n}(\mO),
}
\]
$\iota_n(\mO)$ is obtained from from $\iota_{n-1}(\mO)$ by attaching cells.
We write $|\bD^n|$ for the topological $n$-disk and $|\partial\bD^n |$ for its boundary, a topological $n-1$-sphere.
Extending by colimits, we define the {\em strict realization} $|\mO|$ of $\mO$ to be the resulting (compactly generated weak Hausdorf) topological space, together with its cell structure induced from polygraphic structure of $\mO$.
\end{notation}

\begin{definition}
A family of oriented polytopes is a full subcategory $\mO :=\{\mO^n\}_{n\in\bN}\subset\infty\Cat^\Steiner$ satisfying the following axioms:
\begin{enumerate}[\normalfont(1)]\setlength{\itemsep}{-2pt}
\item For every $n \geq 0$ the $n$-category $\mO^n $ has a unique non-invertible $n$-morphism.
\item For every $n \geq 0$ the canonical functor
$$\underset{\mO^m \rightarrowtail\mO^n \mid m < n}{\colim} \mO^m \to \partial \mO^n:= \iota_{n-1}(\mO^n) $$
is an equivalence. Here the inclusions are taken over all atomic inclusions.
\item For every $n \geq 0$ the topological cell complex $|\mO^n|$ is the cell structure on an unoriented polytope in which the $m$-cells are the $m$-dimensional faces and the attaching maps are piecewise linear injections. 
In particular, $|\mO^n|$ is a contractible topological space, and $\tau_{0}\mO^n$ is a contractible $\infty$-groupoid.
\end{enumerate}  
\end{definition}

\begin{notation}
We will say that a family of oriented polytopes is dense if it is dense as a full subcategory of $\infty\Cat$.
\end{notation}

\cref{cello} implies the following: 
\begin{corollary}\label{deco}
Let $\{\mO^n\}_{n\in\bN}$ be a dense family of oriented polytopes and $n \geq 0.$
The canonical commutative square is a pushout square:
\[
\xymatrix{
\partial\bD^n\ar[d]\ar[r] & \partial \mO^n \ar[d]\\
\bD^n \ar[r] & \mO^n.
}
\]

\end{corollary}

\begin{notation}
    
Let $\{\mO^n\}_{n\in\bN}$ be a dense family of oriented polytopes and $n \geq 0.$
For every $n \in \bN$ we call $\bar{\mO}^n:= (\mO^n)^\co$ the antioriented polytope associated to $\mO^n$.
We call $\{\bar{\mO}^n\}_{n\in\bN}$ the associated anti-family.

\end{notation}

\begin{remark}

Let $\{\mO^n\}_{n\in\bN}$ be a dense family of oriented polytopes and $n \geq 0.$
The associated anti-family $\{\bar{\mO}^n\}_{n\in\bN}$ is a family of oriented polytopes.
    
\end{remark}

\begin{proposition}\label{orientdecom}\label{orientdecom2}

Let $\n \geq 0$. 
\begin{enumerate}[\normalfont(1)]\setlength{\itemsep}{-2pt}
\item The canonical functor $\underset{\underset{{k<n}}{[\bk] \hookrightarrow [\n]}}{\colim} \bbDelta^\bk \to \partial\bbDelta^\n$
is an equivalence.
\item The canonical functor $\underset{\underset{{n-2\leq k\leq n-1}}{[k]\hookrightarrow [\n]}}{\colim}\bbDelta^\bk \to \partial\bbDelta^\n $	
is an equivalence.

\item The canonical functor 
$\colim(\underset{0 \leq i <j \leq n}{\coprod}\bbDelta^{n-2}\rightrightarrows \underset{0 \leq i \leq n}{\coprod} \bbDelta^{n-1}) \to \partial\bbDelta^\n $
is an equivalence.

\end{enumerate}

\end{proposition}

\begin{proof}

Statement (3) is equivalent to statement (2) since there is a canonical equivalence
$$\underset{\underset{{n-2\leq k\leq n-1}}{[k]\hookrightarrow [\n]}}{\colim}\bbDelta^\bk \simeq \colim(\underset{0 \leq i <j \leq n}{\coprod}\bbDelta^{n-2}\rightrightarrows \underset{0 \leq i \leq n}{\coprod} \bbDelta^{n-1}).$$

We prove (1) and (2).
We first prove that the result holds when considering the colimit in $\infty\Cat^\strict.$
The canonical directed augmented chain map $\colim_{[\bk] \hookrightarrow [\n], \bk < \n} \lambda(\bbDelta^\bk) \to \lambda(\bbDelta^\n) $
induces a map \begin{equation}\label{map1}
\colim_{[\bk] \hookrightarrow [\n], \bk < \n} \lambda(\bbDelta^\bk) \to \lambda(\partial\bbDelta^\n),\end{equation}
which in degree $m < n $ is the map on free abelian groups
induced by the map
$$ \colim_{[\bk] \hookrightarrow [\n], \bk < \n} \{ [m]\hookrightarrow [k] \} \to \{ [m]\hookrightarrow [n] \}.$$
The latter map is evidently surjective.
The latter map is also injective since for every injective order preserving maps
$[m] \hookrightarrow [k], [k] \hookrightarrow [n] $ the image of
$[m] \hookrightarrow [k]$ in the colimit $\colim_{[\bk] \hookrightarrow [\n], \bk < \n} \{ [m]\hookrightarrow [k] \}$ under the map associated to $[\bk] \hookrightarrow [\n]$ coincides with the image of the identity of $[m]$ in this colimit under the map associated to $[m]\hookrightarrow [\bk] \hookrightarrow [\n]$,
and so only depends on the composition $[m]\hookrightarrow [\bk] \hookrightarrow [\n]$.
Therefore the directed augmented chain map (\ref{map1}) is an isomorphism.
Since for every $k \leq \ell \leq n$ any order preserving injection
$ [k] \hookrightarrow [\ell] $ the augmented directed chain map $\lambda(\bbDelta^\bk) \to \lambda(\bbDelta^\ell) $
preserves generators, by \cref{pasting1} the colimit $\colim_{[\bk] \hookrightarrow [\n], \bk < \n} \lambda(\bbDelta^\bk) $ 
is preserved by the functor $\nu: \mathrm{ADC} \to \infty\Cat^\strict$.
This proves (1) for the colimit in $\infty\Cat^\strict.$
The inclusion of posets
$\{ [\bk] \hookrightarrow [\n], \bk =n-1,n-2 \} \subset \{ [\bk] \hookrightarrow [n], \bk < n \}$ is cofinal. This implies (2) for the colimit in $\infty\Cat^\strict.$

We prove next (1) (and so (2)) by induction on $n \geq 0$.
For $n=0,1$ there is nothing to show.
We assume the statement for $n$ and like to see it for $n+1.$
By assumption the canonical functor $ \colim_{[\bk] \hookrightarrow [\n], \bk = \n-1, n-2} \bbDelta^\bk \to \partial\bbDelta^\n $
is an equivalence.
The inclusion of posets
$\{ [\bk] \hookrightarrow [\n], \bk =n-1,n-2 \} \subset \{ [\bk] \hookrightarrow [n], \bk < n \}$ is trivially cofinal and so a weak equivalence. Since the poset 
$\{ [\bk] \hookrightarrow [n], \bk < n \}$
is cofiltered and so weakly contractible,
also the poset $\{ [\bk] \hookrightarrow [\n], \bk =n-1,n-2 \}$
is weakly contractible.
Hence the canonical functor $$\colim_{[\bk] \hookrightarrow [\n], \bk= n-1,n-2} (\bbDelta^\bk \star \bDelta^0) \to (\partial\bbDelta^\n) \star \bDelta^0 $$
is an equivalence.
Using \cref{cell} the commutative square in $\infty\Cat$
\[
\xymatrix{
\colim_{[\bk] \hookrightarrow [n], \bk= n-1,n-2}\bbDelta^{\bk} \ar[d]\ar[r] & \colim_{[\bk] \hookrightarrow [\n], \bk= n-1,n-2} (\bbDelta^\bk \star \bDelta^0) \ar[d]\\
\bDelta^n \ar[r] & \partial\bDelta^{n+1}
}
\]
is a pushout square.
The pushout of this square is the colimit
$\colim_{[\bk] \hookrightarrow [n+1] , \bk= n,n-1} \bbDelta^{\bk} $ 
in $\infty\Cat.$
\end{proof}

\begin{corollary}\label{orientalfam}

The full subcategory of oriented simplices $\bDelta \subset \infty\Cat$ is a dense family of oriented polytopes.
    
\end{corollary}

\begin{proof}
This follows immediately from \cref{orientdecom} and \cref{retract}.  
\end{proof}

\begin{proposition}\label{cubedecom}
Let $\n \geq 0$.
\begin{enumerate}[\normalfont(1)]\setlength{\itemsep}{-2pt}
\item The canonical functor $\colim_{\theta \in (\{0,1\}^{\triangleright})^{\times n}, |\theta|< n} \cube^{|\theta|} \to \partial\cube^\n$
is an equivalence.
\item The canonical functor $\colim_{\theta \in (\{0,1\}^{\triangleright})^{\times n}, |\theta| = n-1, n-2} \cube^{|\theta|} \to \partial\cube^\n$
is an equivalence.

\item The canonical functor 
$\colim(\underset{0 \leq i <j \leq n}{\coprod} 4 \times \cube^{n-2}\rightrightarrows \underset{0 \leq i \leq n}{\coprod} 2 \times \cube^{n-1}) \to \partial\cube^\n $
is an equivalence.

\end{enumerate}

\end{proposition}
\begin{proof}
Statement (3) is equivalent to statement (2) since there is a canonical equivalence
$$ \colim_{\theta \in (\{0,1\}^{\triangleright})^{\times n}, |\theta| = n-1, n-2} \cube^{|\theta|} \simeq \colim(\underset{0 \leq i <j \leq n}{\coprod} 4 \times \cube^{n-2}\rightrightarrows \underset{0 \leq i \leq n}{\coprod} 2 \times \cube^{n-1}).$$

We prove (1) and (2). We first prove that the result holds when considering the colimit in $\infty\Cat^\strict.$
The canonical directed augmented chain map \[
\colim_{\theta \in (\{0,1\}^{\triangleright})^{\times n}, |\theta|< n} \lambda(\cube^{|\theta|}) \to \lambda(\cube^\n)
\]
induces a map \[
\colim_{\theta \in (\{0,1\}^{\triangleright})^{\times n}, |\theta|< n} \lambda(\cube^{|\theta|})\to \lambda(\partial\cube^\n)
\]
which in degree $m < n $ is the map on free abelian groups
induced by the map
\begin{equation}\label{map2}
\colim_{\theta \in (\{0,1\}^{\triangleright})^{\times n}, |\theta|< n} \{\Lambda \in (\{0,1\}^{\triangleright})^{\times |\theta|} \mid |\Lambda| = m\} \to \{\rho \in (\{0,1\}^{\triangleright})^{\times n} \mid |\rho| = m\}.
\end{equation}
We prove that this map is bijective.
Let 
$\theta \in (\{0,1\}^{\triangleright})^{\times n} $ and $|\theta|< n$
and $\psi_\theta$ the canonical map
$$ \{\Lambda \in (\{0,1\}^{\triangleright})^{\times |\theta|} \mid |\Lambda| = m\} \to \{\rho \in (\{0,1\}^{\triangleright})^{\times n} \mid |\rho| = m\}. $$
Let $ \Lambda \in (\{0,1\}^{\triangleright})^{\times |\theta|} $ such that $|\Lambda| = m$.
We describe $\psi_\theta(\Lambda).$ If we write $\theta$ as $(X_1,...,X_n) \in \{0,1\}^{\triangleright})^{\times n} \setminus \{\infty\}$
and $\Lambda$ as $(Y_1,...,Y_{|\theta|}) \in \{0,1\}^{\triangleright})^{\times |\theta|}$,
then $\psi_\theta(\Lambda)$ arises by replacing the $\ell$-th $X_i$ satisfying $X_i =\infty$ by $Y_\ell$ for $1 \leq \ell \leq |\theta|.$
In particular, there is a morphism $ \psi_\theta(\Lambda) \to \theta$ in the poset
$(\{0,1\}^{\triangleright})^{\times n}$.
So for every $\rho \in (\{0,1\}^{\triangleright})^{\times n} $ such that $|\rho| = m$, which implies that $|\rho| < n,$
we find that $\psi_\rho(\infty) = \rho$, where 
$\infty$ is the final object of $(\{0,1\}^{\triangleright})^{\times m}.$
Hence the map \ref{map2} is surjective.

We prove next that the map \ref{map2} is injective.
Let 
$\theta \in (\{0,1\}^{\triangleright})^{\times n} $ and $|\theta|< n$
and
$ \Lambda \in (\{0,1\}^{\triangleright})^{\times |\theta|} $ such that $|\Lambda| = m$.
We prove that $(\theta, \Lambda)$ and $(\psi_\theta(\Lambda), \infty)$
agree in the colimit 
$\colim_{\theta \in (\{0,1\}^{\triangleright})^{\times n}, |\theta|< n} \{\Lambda \in (\{0,1\}^{\triangleright})^{\times |\theta|} \mid |\Lambda| = m\}.$
Since $ |\psi_\theta(\Lambda)|=m,$ the morphism
$ \psi_\theta(\Lambda) \to \theta$ in the poset
$(\{0,1\}^{\triangleright})^{\times n}$
induces a map
$$\{\infty\}=\{\Gamma \in (\{0,1\}^{\triangleright})^{\times m} \mid |\Gamma| = m\} \to \{\Gamma \in (\{0,1\}^{\triangleright})^{\times |\theta|} \mid |\Gamma| = m\},$$
which sends $\infty$ to $\Lambda.$
Therefore the map (\ref{map2}) is a bijection.

Since for every morphism $\theta \to \theta'$
in $(\{0,1\}^{\triangleright})^{\times n}$
the augmented directed chain map $\lambda(\cube^{|\theta|}) \to \lambda(\cube^{|\theta'|}) $
preserves generators, by \cref{pasting1} the colimit $\colim_{\theta \in (\{0,1\}^{\triangleright})^{\times n}, |\theta|< n} \lambda(\cube^{|\theta|}) $ 
is preserved by the functor $\nu: \mathrm{ADC} \to \infty\Cat^\strict$.
This proves (1) for the colimit in $\infty\Cat^\strict.$
The inclusion of posets
$ \{\theta \in (\{0,1\}^{\triangleright})^{\times n}, |\theta| = n-1, n-2 \} \subset \{\theta \in (\{0,1\}^{\triangleright})^{\times n}, |\theta|< n\}$ is trivially cofinal.

We prove next (1) (and so (2)) by induction on $n \geq 0$.
For $n=0,1$ there is nothing to show.
We assume the statement for $n$ and like to see it for $n+1.$
By assumption the functor $ \colim_{ \theta \in (\{0,1\}^{\triangleright})^{\times n}, |\theta| = n-1, n-2} \cube^{|\theta|} \to \partial\cube^\n $
is an equivalence. 
So also the canonical functor $$ \colim_{ \theta \in (\{0,1\}^{\triangleright})^{\times n}, |\theta| = n-1, n-2} (\cube^{|\theta|} \boxtimes\cube^1)
\to (\partial\cube^\n) \boxtimes\cube^1 $$
is an equivalence.
Using \cref{cell} the commutative square in $\infty\Cat$
\[
\xymatrix{
\colim_{\theta \in (\{0,1\}^{\triangleright})^{\times n}, |\theta| = n-1, n-2} \cube^{|\theta|} \coprod \cube^{|\theta|} \ar[d]\ar[r] & \colim_{ \theta \in (\{0,1\}^{\triangleright})^{\times n}, |\theta| = n-1, n-2} (\cube^{|\theta|} \boxtimes\cube^1) \ar[d]\\
\cube^n \coprod \cube^n \ar[r] & \partial\cube^{n+1}
}
\]
is a pushout square.
The pushout of this square is the colimit
$\colim_{ \theta \in (\{0,1\}^{\triangleright})^{\times n+1}, |\theta| = n, n-1} \cube^{|\theta|} $ 
in $\infty\Cat.$
\end{proof}

\begin{corollary}\label{cubefam}

The full subcategory of oriented cubes $\cube \subset \infty\Cat$ is a dense family of oriented polytopes.
\end{corollary}

\begin{proof}
This follows immediately from \cref{cubedecom}, \cref{retract} and \cref{Thh}.  
\end{proof}

\subsection{The oriented Street--Roberts conjecture}

The {\em homotopy hypothesis}, formulated by Grothendieck, stipulates an equivalence between the categories of homotopy types and $\infty$-groupoids.
Motivated by this philosophy, we show that $\infty$-categories are equivalent to {\em oriented homotopy types}, by which we mean presheaves of $\infty$-groupoids on a suitable category of oriented polytopes, satisfying a certain sheaf-like condition.
In the sequel paper \cite{gepner2025oriented}, we develop the theory of oriented categories, which are categories enriched in oriented spaces under the Gray tensor product, the monoidal structure compatible with the geometry of oriented polytopes.
This higher dimensional analogue of the homotopy hypothesis is naturally an equivalence of oriented categories.

\begin{definition}Let $\mO=\{\mO^n\}_{n\in\bN}\subset\infty\Cat$ be a dense family of oriented polytopes and $\mC$ a category.	
An $\mO$-object in $\mC$ is a functor $\mO^\op \to \mC.$
\end{definition}

\begin{example}
An $\mO$-space is a functor $\mO^\op \to \mS.$	
\end{example}

\begin{definition}Let $\mO$ be a generating family of polytopes.
The $\mO$-nerve functor is the restricted Yoneda-embedding
$$ \N_\mO: \infty\Cat \to \Fun(\mO^\op,\infty\Grp).$$

\end{definition}

\begin{remark}

The $\mO$-nerve functor $ \N_\mO: \infty\Cat \to \Fun(\mO^\op,\infty\Grp)$
sends strict $\infty$-categories to $\mO$-sets
and so restricts to a functor 
$ \infty\Cat^\strict \to \Fun(\mO^\op,\Set).$

\end{remark}

\begin{notation}Let $\mO\subset\infty\Cat$ be a dense family of oriented polytopes.
For every $n \geq 0$ we define the $\mO$-space $ \partial\N_{\mO}(\bD^n) $
equipped with a map of $\mO$-spaces $ \partial\N_{\mO}(\bD^n) \to \N_{\mO}(\partial\bD^n)$ by induction on $n \geq 0:$
\begin{enumerate}[\normalfont(1)]\setlength{\itemsep}{-2pt}
 
\item Let $ \partial\N_{\mO}(\bD^0)=\emptyset. $

\item Let $ \partial\N_{\mO}(\bD^0)=\emptyset \to \N_{\mO}(\partial\bD^0) $ be the unique map.

\item Let $ \partial\N_{\mO}(\bD^{n+1})= \N_{\mO}(\bD^n)\coprod_{\partial\N_{\mO}(\bD^{n})} \N_{\mO}(\bD^n). $

\item Let $ \partial\N_{\mO}(\bD^{n+1})= \N_{\mO}(\bD^n)\underset{\partial\N_{\mO}(\bD^{n})}{\coprod} \N_{\mO}(\bD^n) \to \N_{\mO}(\bD^n) \underset{\N_{\mO}(\partial\bD^{n})}{\coprod} \N_{\mO}(\bD^n) \to \N_{\mO}(\bD^n \underset{\partial\bD^{n}}{\coprod} \bD^n)= \N_{\mO}(\partial\bD^{n+1}). $
\end{enumerate}   
\end{notation}

\begin{definition}\label{orientedSegal}
Let $\mO:= \{\mO_n \mid n \geq 0\}\subset\infty\Cat$ be a dense family of oriented polytopes.
An $\mO$-space satisifies the Segal condition if it is local with respect to the following maps:
\begin{enumerate}[\normalfont(1)]\setlength{\itemsep}{-2pt}
\item For every $\n \geq 0$ the canonical map
$$ \colim_{\mO^m \rightarrowtail\mO^n \mid m < n} \N_{\mO}(\mO^m) 
\to \N_{\mO}(\partial \mO^\n),$$
where the colimit is taken over all atomic inclusions.
\item For every $\n \geq 0$ the canonical map
$$ \N_{\mO}(\bD^\n) \coprod_{\partial\N_{\mO}(\bD^\n)} \N_{\mO}(\partial \mO^\n) \to \N_{\mO}(\mO^\n).$$
\item The map $$ \N_{\mO}(\bD^{i_0}) \coprod_{\N_{\mO}(\bD^{j_1})} \N_{\mO}(\bD^{i_1}) \coprod_{\N_{\mO}(\bD^{j_2})} \cdots \coprod_{\N_{\mO}(\bD^{j_n})} \N_{\mO}(\bD^{i_n}) \to \N_{\mO}(\bD^{i_0} \coprod_{\bD^{j_1}} \bD^{i_1} \coprod_{\bD^{j_2}} \cdots \coprod_{\bD^{j_n}} \bD^{i_n}) $$
for any natural numbers $n, i_0,\ldots, i_n, j_1,\ldots, j_n$ and monomorphisms 
$ \bD^{j_\ell} \rightarrowtail \bD^{i_\ell},\bD^{j_\ell} \rightarrowtail \bD^{i_{\ell-1}}$, $ 1 \leq \ell \leq n$.
\end{enumerate}	
\end{definition}

\begin{definition}\label{orientedcomp}
Let $\mO$ be a generating family of polytopes.
Let $J$ be the (non-univalent) $(1,1)$-category with two objects and a unique equivalence between them. 
An $\mO$-space satisfying the Segal condition is univalent if it is local with respect to the map of oriented $\mO$-spaces $\N_\mO(S^{n}(J)) \to * $ for every $n \geq 0.$
\end{definition}

We prove the following oriented version of the Street--Roberts conjecture.

\begin{theorem}\label{polytopnerv}
Let $\mO$ be a generating family of polytopes.
The $\mO$-nerve functor $$\N_\mO:\infty\Cat\to\Fun(\mO^{\op},\infty\Gpd)$$
is fully faithful and the essential image precisely consists of the
$\mO$-spaces satisfying the Segal condition.
Moreover an $\infty$-category is univalent if and only if its $\mO$-nerve is univalent.
\end{theorem}
\begin{proof}

We observe first that by fully faithfulness of the $\mO$-nerve an $\infty$-category is univalent if and only if its $\mO$-nerve is local with respect to the map of oriented $\mO$-spaces $\N_\mO(S^{n}(J)) \to * $ for every $n \geq 0.$

So it suffices to see that an oriented $\mO$-space $X$ satisfies the Segal condition if and only if it is the $\mO$-nerve of an $\infty$-category.

We start with proving that the $\mO$-nerve of an $\infty$-category satisfies the Segal condition, i.e. is local with respect to the maps of \cref{orientedSegal}.
By adjointness this is equivalent to say that the $\mO$-realization functor inverts the maps of \cref{orientedSegal}.
This holds for the maps (3) of \cref{orientedSegal} because the $\mO$-nerve is fully faithful. It holds for the maps (1) of \cref{orientedSegal} by \cref{orientdecom2}.
It holds for the maps (2) of \cref{orientedSegal} by \cref{orientdec} and because the $\mO$-realization of $\partial\N_{\mO}(\bD^n)$ is $\partial\bD^n$, which follows immediately by induction on $n \geq 0.$

We prove the converse. Let $i: \Theta \subset \bar{\mO}$ denote the canonical embedding to the idempotent completion, $i^*: \mP(\mO) \simeq \mP(\bar{\mO}) \to \mP(\Theta)$ the induced functor, and $i_*$ its right adjoint, the right Kan extension along $i.$
Let $X$ be an oriented $\mO$-space.
Condition (3) of \cref{orientedSegal} implies that the restriction
$i^*(X)$ is a $\Theta$-Segal space. By \cref{theta} this implies that
$i^*(X)$ is the $\Theta$-nerve $\N_\Theta(\mC)$ of an $\infty$-category $\mC$.
There is a canonical equivalence $i^* (\N_{\mO}(\mC)) \simeq \N_\Theta(\mC).$
By density of $\Theta$ the canonical map $\N_{\mO}(\mC) \to i_*(i^*(\N_{\mO}(\mC))) \simeq i_*(\N_\Theta(\mC)) $ is an equivalence. 
So there is an equivalence $\N_{\mO}(\mC) \simeq i_*(\N_\Theta(\mC)) \simeq i_*(i^*(X)).$
Consequently, it suffices to see that the unit
$X \to i_*(i^*(X)) $ is an equivalence. 

By the triangle identities the unit
$X \to i_*(i^*(X)) $ is inverted by $i^*$ and so in particular inverted by the functor corepresented by $\N_{\mO}(\bD^n)$ for every $n \geq 0.$
By the first part of the proof $\N_{\mO}(\mC)$ satisfies the Segal condition.
So it suffices to see that a map $\phi$ between oriented $\mO$-spaces satisfying the Segal condition is an equivalence if it is inverted by the functor
corepresented by $\N_{\mO}(\bD^n)$ for every $n \geq 0.$

We prove by induction on $n \geq 0$ that $\phi$ 
induces an equivalence at any oriented polytope of dimension smaller or equal $n$ if $\phi$ is inverted by the functor corepresented by $\N_{\mO}(\bD^m)$ for every $m \leq n.$
This trivially hold for $n=0.$
Assume it holds for $n \geq 0$ and that $\phi$ is inverted by the functor corepresented by $\N_{\mO}(\bD^m)$ for every $m \leq n+1.$
Then by induction hypothesis the map $\phi$ induces an equivalence at any oriented polytope of dimension smaller or equal $n.$ So it remains to see that $\phi$ induces an equivalence at the oriented polytope of dimension $n+1.$
By (2) of \cref{orientedSegal} the map $\phi$ induces an equivalence at the oriented polytope of dimension $n+1$
if $\phi$ is inverted by the functors corepresented by $\N_{\mO}(\bD^{n+1})$ and $\partial\N_{\mO}(\bD^{n+1})$ and the functor
corepresented by $\N_{\mO}(\partial\mO^{n+1})$.

By assumption the functor $\phi$ is inverted by the functor corepresented by $\N_{\mO}(\bD^{n+1})$.
By (1) of \cref{orientedSegal} the map $\phi$ is inverted by the functor corepresented by $\N_{\mO}(\partial\mO^{n+1})$ since it induces an equivalence at any oriented polytope of dimension smaller or equal $n.$

So it suffices to see that $\phi$ is inverted by the functor corepresented by $\partial\N_{\mO}(\bD^{n+1}).$
This is because, for every $n \geq 0$, 
a map between oriented $\mO$-spaces is inverted by the functor 
corepresented by $\partial\N_{\mO}(\bD^{n+1})$ if it is inverted by the functor corepresented by $\N_{\mO}(\bD^m)$ for every $m \leq n.$
The latter follows immediately by induction on $n \geq 0$
from the inductive definition of $\partial\N_{\mO}(\bD^{n+1})= \N_{\mO}(\bD^n)\coprod_{\partial\N_{\mO}(\bD^{n})} \N_{\mO}(\bD^n).$
\end{proof}

\begin{corollary}\label{compactgeneration}
Every oriented polytope is compact in $\infty\Cat$ and $\infty\Cat^\strict$.    
\end{corollary}

\begin{proof}
By the Yoneda lemma any representable presheaf corepresents a small colimit-preserving functor.
The result follows from the fact that
$\infty\Cat$ is a localization of $\mP(\mO)$
and the local objects are closed under small filtered colimits by \cref{polytopnerv}.
The strict version is similar, using set-valued presheaves.
\end{proof}

\subsection{Complicial spaces}

In the following we recall the theory of complicial sets and complicial spaces developed by \cite{VERITY1}, \cite{VERITY2}, \cite{ozornova2020model}, \cite{loubaton2024categorical}, \cite{loubaton2024complicialmodelinftyomegacategories}.

\begin{notation}
Let $X $ be a simplicial space and $n \geq 1.$
Let $$ \deg^n(X):= \coprod_{[n] \twoheadrightarrow [n-1]} X_{n-1}$$
the space of degenerate $n$-simplices.
There is a canonical map $\deg^n(X) \to X_n.$

\end{notation}

\begin{definition}

A prestratified simplicial space is a pair $(X,\mE,t)$
consisting of a simplicial space $X$ and for every $n \geq 1$ a 
factorization $\deg^n(X) \to \mE_n \to X_n$ of the canonical map.

A prestratified simplicial set is a prestratified simplicial space
$(X,\mE)$ such that $X$ is a simplicial set and for every $n \geq 1$
the space $\mE_n$ is a set.
    
\end{definition}

\begin{notation}

Let $$ \mP(\Delta)^{\#} := \prod_{n\geq 1}\mP([2])\times_{\prod_{n\geq 1}\mP([1])} \mP(\Delta)$$
be the full subcategory of prestratified simplicial spaces,
where the pullback
is taken along the functor $\mP([2]) \to \mP([1]) $ induced by the functor $[1] =\{0<2\} \subset [2]$ and the functor
$ \mP(\Delta) \to \prod_{n\geq 1}\mP([1])$ sending $ X$ to $ \{\deg^n(X) \to X_n\}_{n \geq 1}.$ 
\end{notation}

\begin{notation}

The projection of the pullback restricts to a forgetful functor 
$\mP(\Delta)^{\#} \to \mP(\Delta).$

\end{notation}

\begin{remark}\label{heu}

Let $(A \to B \to C), (A' \to B' \to C') \in \mP([2])$.
There is a canonical equivalence
$$ \Map_{\mP([2])}(A \to B \to C), (A' \to B' \to C') \simeq $$
$$ \Map_{\mP([1])}((A \to B), (A' \to B')) \times_{\Map_{\mS}(B, B')} \Map_{\mP([1])}((B \to C), (B' \to C')) \simeq $$
$$ \Map_{\mS}(A, A') \times_{\Map_{\mS}(A, B')} \Map_{\mS}(B, B') \times_{\Map_{\mS}(B, B')} \Map_{\mS}(B, B') \times_{\Map_{\mS}(B, C')} \Map_{\mS}(C, C') $$
$$ \simeq \Map_{\mS}(A, A') \times_{\Map_{\mS}(A, B')} \Map_{\mS}(B, B') \times_{\Map_{\mS}(B, C')} \Map_{\mS}(C, C').$$

The functor $ \xi: \mP([2]) \to \mP([1]) $ induced by the functor $[1] =\{0<2\} \subset [2]$ induces a map
$$ \Map_{\mP([2])}((A \xrightarrow{\id} B \to C), (A' \to B' \to C')) \to \Map_{\mP([1])}((A \to C), (A' \to C')), $$
which identifies with the canonical map
\begin{equation}\label{maops}
\Map_{\mS}(A, A') \times_{\Map_{\mS}(A, B')} \Map_{\mS}(B, B') \times_{\Map_{\mS}(B, C')} \Map_{\mS}(C, C') \to \Map_{\mS}(A, A') \times_{\Map_{\mS}(A, C')} \Map_{\mS}(C, C').
\end{equation}

The latter map is for instance an equivalence if $A \to B$ or $B' \to C'$
is an equivalence.

Let $U,V, W $ be spaces and $U \coprod V \to W$ a map.
For $(A \to B \to C) = (U \to U \coprod V \to W)$ the map
(\ref{maops}) identifies with the projection
$$
\Map_{\mS}(U, A') \times_{\Map_{\mS}(U,C')} \Map_{\mS}(W, C') \times_{\Map_{\mS}(V,C')} \Map_{\mS}(V, B') \to \Map_{\mS}(U, A') \times_{\Map_{\mS}(U, C')} \Map_{\mS}(W, C').
$$
\end{remark}

\begin{notation}
Let $X$ be a simplicial space.
\begin{enumerate}[\normalfont(1)]\setlength{\itemsep}{-2pt}
\item Let $X^\flat$ be the following prestratified simplicial space:
the simplicial space $X$ equipped with
the factorizations $\{\deg^n(X) \xrightarrow{\id} \deg^n(X) \to X_n\}_{n \geq 1}.$

\item Let $X^\#$ be the following prestratified simplicial space:
the simplicial space $X$ equipped with
the factorizations $\{\deg^n(X) \to X_n \xrightarrow{\id} X_n\}_{n \geq 1}.$

\end{enumerate}
    
\end{notation}

\begin{lemma}\label{stratadj}

The forgetful functor 
$\mP(\Delta)^{\#} \to \mP(\Delta)$ admits a fully faithful left adjoint $(-)^\flat$ that sends $X$ to $X^\flat$ and a fully faithful right adjoint $(-)^{\#}$ that sends $X$ to $X^{\#}.$
\end{lemma}

\begin{proof}

Let $X$ be a simplicial space and
$(Y, \mF) \in \mP(\Delta)^{\#}.$
The induced map $$ \Map_{\mP(\Delta)^{\#}}(X^\flat,(Y, \mF)) \to \Map_{\mP(\Delta)}(X,Y) $$ is an equivalence by \cref{heu}.

Let $X$ be a simplicial space and 
$(Y, \mF) \in \mP(\Delta)^{\#}.$
The induced map $$ \Map_{\mP(\Delta)^{\#}}((Y, \mF), X^{\#}) \to \Map_{\mP(\Delta)}(Y,X) $$ 
is an equivalence by \cref{heu}.

The  left and right adjoints are both fully faithful since the respective unit and counit are equivalences.   
\end{proof}

We embed $\mP(\Delta)$ into $\mP(\Delta)^{\#}$ via the left adjoint of the forgetful functor $\mP(\Delta)^{\#} \to \mP(\Delta)$ , which we henceforth suppress notationally.

\begin{definition}

A stratified simplicial space is a prestratified simplicial space
$(X,\mE)$ such that for every $n \geq 1$
the map $\mE_n \to \X_n$ is an embedding of spaces.

A stratified simplicial set is a stratified simplicial space
$(X,\mE)$ such that $X$ is a simplicial set.
    
\end{definition}

\begin{remark}

Let $(X,\mE)$ be a stratified simplicial space.
For every $n \geq 1$ the map $\mE_n \to X_n$ is an embedding of spaces.
This implies that the factorization $\deg^n(X) \to \mE_n \to X_n$ of the canonical map is uniquely determined and so the structure of the factorization becomes the condition that every element of the subspace $\mE_n$ of $X_n$ is degenerate.
So we can identify a stratified simplicial space with a simplicial space
$X$ and for every $n \geq 1$ a collection $\mE_n$ of distinguished $n$-simplices of $X$, which we call marked or thin $n$-simplices of $X.$
    
\end{remark}

\begin{remark}
The embedding $\mP(\bD^1)_\mathrm{mono}\to\mP(\bD^1)$ 
of the full subcategory of monomorphisms admits a left adjoint which arises from the factorization of any map into an essentially surjective and fully faithful map.
The unit of this adjunction is inverted by evaluation at the target and is sent to an essentially surjective map by evaluation at the source. 
This implies that the full subcategory of $\mP(\Delta)^{\#}$
of stratified simplicial spaces is a reflective full subcategory.
\end{remark}

\begin{notation}

Let $\n \geq 0$.
Let $(\Delta^\n)^t $ be stratified simplicial set $(\Delta^\n, \mE),$ 
where $\mE_m$ is the subset of degenerate $m$-simplices of $\Delta^n$ for every $m \neq n$ and $\mE_n$ is the subspace of
degenerate $n$-simplices together with the identity of $[n].$
\end{notation}

\begin{notation}
Let $$\Delta^+ \subset \mP(\Delta)^{\#}$$ be the full subcategory  
spanned by $\Delta^\n$ and $(\Delta^\n)^t.$
\end{notation}

\begin{lemma}\label{markadj}
Let $n \geq 1$ and $(X,\mE) \in \mP(\Delta)^{\#}.$
There is a natural equivalence
$$ \Map_{\mP(\Delta)^{\#}}((\Delta^n)^t, (X,\mE)) \simeq \mE_n. $$
    
\end{lemma}

\begin{proof}
By \cref{heu} there is a canonical equivalence
$$ \Map_{\mP(\Delta)^{\#}}((\Delta^n)^t, (X,\mE)) \simeq \Map_{\mP(\Delta)}(\Delta^n, X) \times_{X_n} \mE_n \simeq \mE_n.$$
\end{proof}

\begin{proposition}
The restricted Yoneda-embedding
$\gamma: \mP(\Delta)^{\#} \to \mP(\Delta^+)$ is an equivalence.

\end{proposition}

\begin{proof} By \cref{stratadj} for every $(X,\mE) \in \mP(\Delta)^{\#}$
and $n \geq 0$ there is a natural equivalence
$$ \gamma(X,\mE)(\Delta^n) = \Map_{\mP(\Delta)^{\#}}(\Delta^n, (X,\mE)) \simeq \Map_{\mP(\Delta)}(\Delta^n, X) \simeq X_n.$$

By \cref{markadj} for every $(X,\mE) \in \mP(\Delta)^{\#}$ and $n \geq 1$
there is a natural equivalence
$$ \gamma(X,\mE)((\Delta^n)^t) = \Map_{\mP(\Delta)^{\#}}((\Delta^n)^t, (X,\mE)) \simeq \mE_n. $$
Thus the functors $$  \Map_{\mP(\Delta)^{\#}}(\Delta^n,-), \ \Map_{\mP(\Delta)^{\#}}((\Delta^n)^t, -): \mP(\Delta)^{\#} \to \mS $$
preserve small colimits and so the objects
$ \Delta^n, (\Delta^n)^t \in \mP(\Delta)^{\#}$ for $ n \geq 0 $
corepresent a family of jointly conservative small colimits
preserving functors.
This implies that the restricted nerve functor
$\mP(\Delta)^{\#} \to \mP(\Delta^+)$ is conservative and its left adjoint is fully faithful and so an equivalence.
\end{proof}

\begin{notation}
We write $\Delta_{\geq -1}=\Delta^\triangleleft$ and $\bDelta_{\geq -1}=\bDelta^\triangleleft$ for the full subcategories of $\infty\Cat$ consisting of the simplices $\Delta$ and oriented simplices $\bDelta$ together with the initial object, the empty $\infty$-category. 
\end{notation}

\begin{construction}\label{joinalg}
Let $(\Delta_{\geq -1}, \ast)$ be the monoidal category of possibly empty totally ordered sets endowed with the join.
By \cref{locmon3} the restricted Yoneda embedding
$$ (\infty\Cat, \star) \to (\mP(\bDelta_{\geq -1}), \star) $$
is lax monoidal.
The monoidal structure on the right is Day-convolution 
induced by the join monoidal structure on $\bDelta_{\geq -1}$,
the full subcategory of $\infty\Cat$ spanned by the orientals and the empty $\infty$-category.
Therefore an associative algebra structure on an
$\infty$-category $\mC$ with respect to join corresponds to
an associative algebra structure on $\Map_{\infty\Cat}((-)|_{\bDelta_{\geq -1}},\mC) \in \mP(\bDelta_{\geq -1})$
with respect to Day-convolution, which is the same as a lax monoidal structure on the functor $$\Map_{\infty\Cat}((-)_{|\bDelta_{\geq -1}},\mC): \bDelta_{\geq -1}^\op \to \mS.$$

We apply this to $\mC= \bD^0.$
Then an associative algebra structure on $\bD^0$ with respect to join corresponds to a lax monoidal structure on the constant functor $\bDelta_{\geq -1}^\op \to \mS$ with value the final object, the tensor unit of $(\mS, \times).$
The constant functor $\bDelta_{\geq -1}^\op \to \mS$ with value the final object factors as monoidal functors $\bDelta_{\geq -1}^\op \to \bD^0 \to \mS$.
So we obtain a canonical associative algebra structure on $\bD^0$ in $(\infty\Cat, \star).$
    
\end{construction}

\begin{proposition}

There is a unique 
monoidal functor $\bDelta^{(-)}:(\Delta_{\geq -1}, \ast) \to (\infty\Cat, \ast)$ sending $[n]$ to $\bDelta^n.$
    
\end{proposition}

\begin{proof}

Let $\mC$ be a monoidal category.
By \cite[Proposition 2.2.4.9.]{lurie.higheralgebra} for every associative algebra $A$ in $\mC$ there is a unique 
monoidal functor $(\Delta_{\geq -1}, \ast) \to \mC$ sending
$[0]$ to $A.$
So by \cref{joinalg} there is a unique 
monoidal functor $\bDelta^{(-)}:(\Delta_{\geq -1}, \ast) \to (\infty\Cat, \ast)$ sending $[0]$ to $\bD^0$ and so sending
$[n]$ to $\bDelta^n.$
\end{proof}

\begin{definition}\label{nerves}

The functor 
$\bDelta^{(-)}: \Delta \xrightarrow{} \bDelta \subset \infty\Cat$
induces an adjunction
$$ \mP(\Delta) \rightleftarrows \infty\Cat. $$
The right adjoint canonically lifts to a functor $$ \N: \infty\Cat \to \mP(\Delta)^{\#} $$
that sends an $\infty$-category $\mC $ to
the space $\Map_{\infty\Cat}(-,\mC)$ equipped with for every $n \geq 0$
the map $$\Map_{\infty\Cat}(\tau_{n-1}\bDelta^\n,\mC) \to \Map_{\infty\Cat}(\bDelta^\n,\mC).$$
We refer to the functor $\N$ as the Street nerve functor.

\end{definition}

\begin{remark}
The Street nerve functor $ \N: \infty\Cat \to \mP(\Delta)^{\#} $ is conservative and admits a left adjoint $$ | -|: \mP(\Delta)^{\#} \to \infty\Cat $$ since $\N$ preserves small limits and is accessible as the forgetful functor $\mP(\Delta)^{\#} \to \mP(\Delta) $ preserves limits and colimits.
\end{remark}

\begin{remark}\label{restru}
By adjointness the left adjoint $| - |: \mP(\Delta)^{\#} \to \infty\Cat $ extends the left adjoint $ \mP(\Delta) \to \infty\Cat $
along the free functor. So $|[n]| \simeq \bDelta^n$.
Moreover there is a canonical equivalence $$| (\Delta^n)^t | \simeq \tau_{n-1}\bDelta^n$$
because for every $\mC \in \infty\Cat$ there is a canonical equivalence
$$ \Map_{\infty\Cat}(|(\Delta^n)^t |,\mC) \simeq  \Map_{\mP(\Delta)^{\#}}((\Delta^n)^t,\N(\mC)) \simeq \Map_{\infty\Cat}(\tau_{n-1}\bDelta^\n,\mC). $$

\end{remark}

\begin{notation}

Let $\bDelta^+ \subset \infty\Cat $ be the full subcategory spanned by the oriented simplices $\bDelta^n $ for $n \geq 0$ and 
$\tau_{n-1}(\bDelta^n) $ for $n \geq 0$.

\end{notation}

\begin{notation}\label{funi}
    
By \cref{restru} the functor $$ \Delta^+ \subset \mP(\Delta)^{\#} \xrightarrow{|-|} \infty\Cat $$ induces an essentially surjective functor $i: \Delta^+ \to \bDelta^+. $
The functor $i: \Delta^+ \to \bDelta^+ $
induces an adjunction $$ i_!: \mP(\Delta^+) \rightleftarrows \mP(\bDelta^+):i^*. $$ 

\end{notation}

\begin{remark}\label{stnerve}
The right adjoint $ i^*: \mP(\bDelta^+) \to \mP(\Delta^+) $ is conservative since $i$ is essentially surjective. 
The composition $ \infty\Cat \to \mP(\bDelta^+) \to \mP(\Delta^+) $
of the $\bDelta^+$-nerve with the functor $i^*$ is the Street nerve $\N$ since for every $\mC \in \infty\Cat$ there is a canonical equivalence
$$ i^*(\N_{\bDelta^+}(\mC)) \simeq \Map_{\infty\Cat}(-,\mC) \circ i \simeq \Map_{\infty\Cat}(|-|_{| \Delta^\op},\mC) \simeq \Map_{\mP(\Delta^+)}((-)_{| \Delta^\op},\N(\mC)) \simeq \N(\mC).$$
\end{remark}

\begin{remark}

Since we work non-univalent, the Street nerve 
$ \N: \infty\Cat \to \mP(\Delta^+)$
sends strict $\infty$-categories to presheaves of sets on $\Delta^+$.
Thus the Street nerve functor restricts to a functor
$ \N: \infty\Cat^\strict \to \mP_\Set(\Delta^+)$.
\end{remark}

\begin{notation}

The restricted Street nerve functor
$ \N: \infty\Cat^\strict \to \mP_\Set(\Delta^+)$
admits a left adjoint $$|-|_{\strict}: \mP_\Set(\Delta^+) \to \infty\Cat^\strict, $$ which factors as $$ \mP_\Set(\Delta^+) \subset \mP(\Delta^+) \xrightarrow{|-|} \infty\Cat \to \infty\Cat^\strict,$$
where the last functor is the left adjoint of the embedding
$\infty\Cat^\strict \subset \infty\Cat$.

\end{notation}

\begin{remark}

By construction, for every $X \in \mP_\Set(\Delta^+)$ there is a canonical functor $|X| \to |X|_\strict$, which is an equivalence if and only if $|X|$ is a strict $\infty$-category. This applies to $X= \Delta^n, (\Delta^n)^t$ for any $n \geq 0.$

\end{remark}

\begin{notation}Let $n \geq 0$ and $0 \leq k \leq n$.
\begin{enumerate}[\normalfont(1)]\setlength{\itemsep}{-2pt}
\item Let $\Delta^n_k$ be the simplicial set $\Delta^n$ equipped with the marking whose non-degenerate simplices are marked if
and only if they contain the vertices $${k-1,k,k+1} \cap [n].$$

\item Let $\Lambda^n_k$ be the simplicial set $\Lambda_k^n$ equipped with marking whose non-degenerate simplices are marked if and only if they contain the vertices $${k-1,k,k+1} \cap [n].$$

\item Let $ (\Delta^n_k)' $ by the simplicial set 
$\Delta^n$ equipped with the marking obtained from
$\Delta^n_k$ and additionally marking the $(k-1)$-st and $(k+1)$-st face of $\Delta^n.$

\item Let $ (\Delta^n_k)'' $ be the simplicial set 
$\Delta^n$ equipped with the marking obtained from
$(\Delta^n_k)'$ and additionally marking the $k$-th face of $\Delta^n.$

\end{enumerate}
    
\end{notation}

\begin{definition}\label{locmapl}
\begin{enumerate}[\normalfont(1)]\setlength{\itemsep}{-2pt}
\item The complicial horn extension is the canonical map
$$ \Lambda^n_k \to \Delta^n_k$$ for $n > 1$ and $0 < k < n$, which lifts the simplicial horn inclusion.

\item The complicial thinness extension is the canonical map
$$ (\Delta_k^n)' \to (\Delta_k^n)'' $$ for $n > 1$ and $0 \leq k \leq n$, which lifts the identity on underlying simplicial sets.

\end{enumerate}

\end{definition}

\begin{definition}

A complicial space is an object of $\mP(\Delta)^{\#}$ that is local with respect to the complicial horn extensions and complicial thinness extensions.
A complicial set is a complicial space whose underlying simplicial space is a simplicial set.
    
\end{definition}

\begin{remark}

Locality with respect to complicial horn extensions provides associative and unital composition.

Locality with respect to complicial thinness extension makes the marked simplices compatible with composition.
    
\end{remark}

\begin{definition}

A complicial space is univalent if it is local with respect to to the
map $\N(S^{n}(J)) \to * $ in $ \mP(\Delta^+)$ for every $n \geq 0$, where $J$ is the non-univalent $(1,1)$-category with two objects and one equivalence between them.
    
\end{definition}

\subsection{The Street--Roberts conjecture}

The following is due to \cite[Theorem 2.1.]{maehara2023orientals}:

\begin{theorem}\label{localeq}

The map $$\Delta^n \to \N(\bDelta^n) $$ of
stratified simplicial sets is local with respect to complicial spaces.

\end{theorem}

\begin{proof}

By \cite[Theorem 2.1.]{maehara2023orientals} the map $\Delta^n \to \N(\bDelta^n) $ of
stratified simplicial sets belongs to the closure of the union
of all complicial horn extensions and thinness extensions
under coproducts, pushouts along arbitrary maps, and transfinite compositions. 
Hence this map is a weak equivalence in the Verity model structure on stratified simplicial sets.
By \cite[Proposition 1.35., Theorem 2.12.]{ozornova2020model} the Verity model structure on stratified simplicial sets models complicial spaces.
\end{proof}

The following conjecture is of fundamental importance in higher category theory.
If it is shown to be true, it would facilitate comparisons between various models of $\infty$-category theory.
Namely, it would provide a direct equivalence between the categories of complicial spaces and Segal presheaves on the orientals, or any other dense family of oriented polytopes.

\begin{conjecture}\label{localeq2}

The map 
$$(\Delta^n)^t \to \N(\tau_{n-1}(\bDelta^n)) $$ of
stratified simplicial spaces is local with respect to complicial spaces.

\end{conjecture}

\begin{notation}

Let $\overline{\bDelta^+} \subset \infty\Cat$ be the smallest full subcategory containing $\bDelta^+$ and closed under retracts.
Note that restriction induces an equivalence $ \mP(\overline{\bDelta^+}) \simeq \mP(\bDelta^+).$     
\end{notation}

\begin{lemma}\label{retros}

Let $Z \in \mP(\bDelta^+) \simeq \mP(\overline{\bDelta^+})$ such that $i^*(Z)$ is a complicial space and $T \in \infty\Cat$ a 
retract of $\bDelta^n$ for some $n \geq 0.$ 
The canonical map 
$$ Z(\T) \simeq \Map_{\mP(\overline{\bDelta^+})}(\N_{\overline{\bDelta^+}}(T),Z) \simeq \Map_{\mP(\bDelta^+)}(\N_{\bDelta^+}(T),Z) \to \Map_{\mP(\Delta^+)}(\N(T),i^*(Z)) $$
is an equivalence.
  
\end{lemma}

\begin{proof}
The induced map $$ \Map_{\mP(\bDelta^+)}(\bDelta^n,Z) \to \Map_{\mP(\Delta^+)}(i^*(\bDelta^n),i^*(Z)) \to \Map_{\mP(\Delta^+)}(\Delta^n,i^*(Z)) $$ identifies with the canonical equivalence
$$ \Map_{\mP(\bDelta^+)}(\bDelta^n,Z) \simeq \Map_{\mP(\bDelta^+)}(i_!(\Delta^n), Z) \simeq \Map_{\mP(\Delta^+)}(\Delta^n,i^*(Z)). $$
Since $i_!(Z)$ is a complicial space, by \cref{localeq} the second map in the composition 
$$ \Map_{\mP(\Delta^+)}(i^*(\bDelta^n),i^*(Z)) \to \Map_{\mP(\Delta^+)}(\Delta^n,i^*(Z)) $$ is an equivalence. Thus the first map in the composition $$ \xi: \Map_{\mP(\bDelta^+)}(\bDelta^n,Z) \to \Map_{\mP(\Delta^+)}(i^*(\bDelta^n),i^*(Z))$$ is an equivalence.

Let $T \in \infty\Cat$ be a retract of $\bDelta^n$.
The map $$ \Map_{\mP(\bDelta^+)}(\N_{\bDelta^+}(T),Z) \to \Map_{\mP(\Delta^+)}(N(T),i^*(Z))$$ is a retract of the equivalence $\xi$ and so an equivalence, too.
Hence the canonical map
$$ Z(\T) \simeq \Map_{\mP(\overline{\bDelta^+})}(T,Z) = \Map_{\mP(\overline{\bDelta^+})}(\N_{\overline{\bDelta^+}}(T),Z) \simeq \Map_{\mP(\bDelta^+)}(\N_{\bDelta^+}(T),Z) \to \Map_{\mP(\Delta^+)}(\N(T),i^*(Z)) $$
is an equivalence.
\end{proof}

\begin{corollary}\label{retros2}

Let $Z \in \mP(\bDelta^+) \simeq \mP(\overline{\bDelta^+})$ such that $i^*(Z)$ is a complicial space and $T \in \mP(\bDelta^+)$ an object of the smallest full subcategory of $\mP(\bDelta^+) \simeq \mP(\overline{\bDelta^+})$ closed under small colimits and containing $\overline{\bDelta} \subset \overline{\bDelta^+}$.
The canonical map 
$$ \Map_{\mP(\bDelta^+)}(T,Z) \to \Map_{\mP(\Delta^+)}(i^*(T),i^*(Z)) $$
is an equivalence.

\end{corollary}

\begin{proof}
The full subcategory of $T \in \mP(\Delta^+)$ such that the map of the statement is an equivalence, is closed under small colimits since the functor
$i^*: \mP(\bDelta^+) \to \mP(\Delta^+)$ preserves small colimits,
and contains $\overline{\bDelta}$ by \cref{retros}.
The result follows.   
\end{proof}

By \cref{orientaldense} the full subcategory $\bDelta \subset \infty\Cat$ is dense.
Hence also the full subcategory $\bDelta^+ \subset \infty\Cat $ is dense. Thus the $\bDelta^+$-nerve $ \infty\Cat \to \mP(\bDelta^+)$ is fully faithful.
In other words, we have a localization $$ \mP(\bDelta^+) \rightleftarrows \infty\Cat,$$ whose right adjoint is the $\bDelta^+$-nerve.

\begin{proposition}\label{charnerve}

A presheaf $X \in \mP(\bDelta^+)$ is the $\bDelta^+$-nerve of an $\infty$-category if and only if $i^*(X) \in \mP(\Delta^+)$ 
is a complicial space.
    
\end{proposition}

\begin{proof}

If a presheaf on $\bDelta^+$ is the $\bDelta^+$-nerve of an $\infty$-category, then its image under $i^*$ is the Street nerve of an $\infty$-category and so a complicial space.
We prove the converse. Let $X$ be a presheaf on $\bDelta^+$ such that $i^*(X)$ is a complicial space.

By \cref{retract} there is an embedding $ \Theta \subset \overline{\bDelta}$ and so an embedding $j: \Theta \subset 
\overline{\bDelta} \subset \overline{\bDelta^+}$.
This embedding induces an adjunction 
$$j^*: \mP(\bDelta^+) \simeq \mP(\overline{\bDelta^+}) \rightleftarrows \mP(\Theta): j_* .$$

If $\iota^*(X)$ is a complicial space, the $\Theta$-space $i^*(X)$ satisfies the Segal condition, and therefore by \cref{theta} is the $\Theta$-nerve of an $\infty$-category $\mC$. Hence $j_*(j^*(X)) \simeq \N_{\bDelta^+}(\mC)$ is the $\bDelta^+$-nerve of $\mC$ by density of $\bDelta^+$ in $\infty\Cat$.
So it suffices to see that the unit $X \to j_*(j^*(X))$ is an equivalence.
Since $i^*$ is conservative, we have to see that the induced map
$i^*(X) \to i^*(j_*(j^*(X)))$ is an equivalence.
By assumption $i^*(X)$ is a complicial space. Moreover $i^*(j_*(j^*(X)))$
is the Street nerve of $\mC$, and so a complicial space, since $j_*(j^*(X))$ is the $\bDelta^+$-nerve of $\mC$. 
Hence $i^*(X) \to i^*(j_*(j^*(X)))$ is a map between complicial spaces.

By definition the unit $X \to j_*(j^*(X))$ in $ \mP(\bDelta^+) \simeq \mP(\overline{\bDelta^+})$ induces an equivalence after restriction to $\Theta.$ In particular, the unit $X \to j_*(j^*(X))$ is an equivalence when evaluated at any $n$-disk for every $n \geq 0.$

The $n$-disk $\bD^n$ is a retract of $\bDelta^n$ via the functor taking the unique non-degenerate $n$-morphism. 
By \cref{retros} the unit $X \to j_*(j^*(X))$ evaluated at any $n$-disk
identifies with the induced map
$$ \Map_{\mP(\Delta^+)}(\N(\bD^n),i^*(X)) \to \Map_{\mP(\Delta^+)}(\N(\bD^n),i^*(j_*(j^*(X)))).$$ 
Hence the latter map is an equivalence.
This implies that the map $i^*(X) \to i^*(j_*(j^*(X)))$ between complicial spaces is an equivalence since the family $\N(\bD^n)$ for $n \geq 0$ detects equivalences between complicial spaces by \cite[Theorem 2.4.4.13.]{loubaton2024categorical}.
\end{proof}

\begin{notation}

Right Kan extension along the embedding $j: \bDelta \subset \bDelta^+$
gives an embedding $j_*: \mP(\bDelta) \to \mP(\bDelta^+)$.
We obtain a functor $$\mP(\bDelta) \xrightarrow{j_*} \mP(\bDelta^+) \xrightarrow{i^*} \mP(\Delta^+). $$ 

\end{notation}

\begin{remark}

The functor $ i^* \circ j_*: \mP(\bDelta) \to \mP(\Delta^+) \simeq \mP(\Delta)^{\#} $
lifts the functor 
$\mP(\bDelta) \to \mP(\Delta) $ induced by the functor
$\bDelta^{(-)}: \Delta \to \bDelta$ of \cref{nerves}, along the forgetful functor $\mP(\Delta)^{\#} \to \mP(\Delta).$

So by construction, the functor $ i^* \circ j_*: \mP(\bDelta) \to \mP(\Delta^+) \simeq \mP(\Delta)^{\#} $ sends $X \in \mP(\bDelta)$ to $ X \circ \bDelta^{(-)}$
equipped with the marked simplices 
$$ \{\Map_{\mP(\bDelta^+)}(\N_{\bDelta^+}(\tau_{n-1}\bDelta^n),j_*(X)) \to \Map_{\mP(\bDelta^+)}(\N_{\bDelta^+}(\bDelta^n),j_*(X)) \}_{n \geq 0}. $$
This set of marked simplices agrees with the set 
$$ \{\Map_{\mP(\bDelta)}(\N_{\bDelta}(\tau_{n-1}\bDelta^n),X) \to \Map_{\mP(\bDelta)}(\N_{\bDelta}(\bDelta^n),X) \}_{n \geq 0} $$
since $j_* \circ \N_{\bDelta} \simeq \N_{\bDelta^+}$ and so 
$\N_{\bDelta} \simeq j^* \circ j_* \circ \N_{\bDelta} \simeq j^* \circ \N_{\bDelta^+}.$

In particular, the functor $ \mP(\bDelta) \subset \mP(\bDelta^+) \xrightarrow{i^*} \mP(\Delta^+) $ 
sends an $\infty$-category to its Street nerve. 

\end{remark}

\begin{corollary}\label{charnerve2}

A $\bDelta$-space $X \in \mP(\bDelta)$ satisfies the Segal condition 
if and only if its image under the canonical functor
$ i^* \circ j_* : \mP(\bDelta) \to \mP(\Delta^+) $  is a complicial space.
  
\end{corollary}

\begin{proof}

By \cref{polytopnerv} a $\bDelta$-space $X $ satisfies the Segal condition if and only if it is the $\bDelta$-nerve of an $\infty$-category. If this holds, the image under the canonical functor
$ i^* \circ j_* : \mP(\bDelta) \to \mP(\Delta^+) $  is the Street nerve of an $\infty$-category and so a complicial space.

Conversely, let $i^*(j_*(X))$ be a complicial space.
By \cref{charnerve} the presheaf $j_*(X) \in \mP(\bDelta^+)$ is the 
$\bDelta^+$-nerve of an $\infty$-category $\mC.$
Hence $X \simeq j^*(j_*(X)) \simeq j^*(\N_{\bDelta^+}(\mC)) \simeq j^*(j_*(\N_{\bDelta}(\mC))) \simeq \N_{\bDelta}(\mC) $ is the $\bDelta$-nerve of $\mC$ and so satisfies the Segal condition. 
\end{proof}

\begin{theorem}\label{Street2}

The Street nerve $ \N: \infty\Cat \to \mP(\Delta^+)$ is fully faithful
and the essential image are precisely the complicial spaces.
An $\infty$-category is univalent if and only if its Street nerve is univalent.
    
\end{theorem}

\begin{proof}
We first prove that the Street nerve is fully faithful.
Let $\bi: \Delta^+ \to \bDelta^+$ be the canonical functor of \cref{funi}. 
Let $$ F:=i_!: \mP(\Delta^+) \to \mP(\bDelta^+): G:=i^* $$ be the canonical adjunction.
The functor $G$ is conservative and preserves small colimits.

By \cref{stnerve} the functor $G$ sends the $\bDelta^+$-nerve of an $\infty$-category to the Street nerve of this $\infty$-category.
The Street nerve of an $\infty$-category is a local object of
$\mP(\Delta^+).$ Hence the functor $G$ preserves local objects so that the left adjoint $F$ preserves local equivalences.
Thus the adjunction $$F: \mP(\Delta^+) \rightleftarrows \mP(\bDelta^+): G$$
descends to an adjunction 
$$F': \mP(\Delta^+)^{\mathrm{loc}} \rightleftarrows \mP(\bDelta^+)^{\mathrm{loc}} \simeq \infty\Cat : G',$$
where $G'$ is the restriction of $G$ and $F'$ is the restriction of $F$ followed by the localization.

We prove that $G'$ is an equivalence. We prove first that $G'$ is 
fully faithful. We have to see that for every local object $X$ in $\mP(\bDelta^+)$ the counit $F'(G'(X)) \to X$ is an equivalence.
Since local equivalences between local objects are equivalences, it suffices to see that the counit $F'(G'(X)) \to X$ is a local equivalence. For that it is enough to see that the composition 
$$F(G(X)) \to F'(G(X)) = F'(G'(X)) \to X,$$ which is the counit, is a local equivalence,
where the morphism $ F(G(X)) \to F'(G(X))$ is the localization.

Next we prove that for every $X$ in $\mP(\bDelta^+)$ the counit $F(G(X)) \to X$ is a local equivalence.
Since $F$ and $G$ preserve small colimits and $\mP(\bDelta^+)$ is generated under small colimits by representables, and local equivalences are closed under colimits, we can assume that $X$ is representable.

We have to see that for every $n \geq 0$ the counits $$F(G(\bDelta^n)) \to \bDelta^n, \ F(G(\tau_{n-1}\bDelta^n)) \to \tau_{n-1}\bDelta^n $$ are local equivalences.
By \cref{localeq} the map $\Delta^n \to G(\bDelta^n) $ corresponding to the equivalence $F(\Delta^n) \simeq \bDelta^n, $ is a local equivalence.
Hence the image $F(\Delta^n) \to F(G(\bDelta^n)) $ is a local equivalence.
Thus the counit $F(G(\bDelta^n)) \to \bDelta^n $ is a local equivalence since the composition $F(\Delta^n) \to F(G(\bDelta^n)) \to \bDelta^n $ is the canonical equivalence. 

The case of $X= \tau_{n-1}\bDelta^n $ is similar:
by \cref{localeq2} the map $(\Delta^n)^t \to G(\tau_{n-1}\bDelta^n) $ corresponding to the equivalence $F((\Delta^n)^t) \simeq \tau_{n-1}\bDelta^n, $ is a local equivalence.
Hence the image $F((\Delta^n)^t) \to F(G(\tau_{n-1}\bDelta^n)) $ is a local equivalence.
Thus the counit $F(G(\tau_{n-1}\bDelta^n)) \to \tau_{n-1}\bDelta^n $ is a local equivalence since the composition $F((\Delta^n)^t) \to F(G(\tau_{n-1}\bDelta^n)) \to \tau_{n-1}\bDelta^n $ is the canonical equivalence. 
This shows that the Street nerve is fully faithful.

We prove next that $F'$ is fully faithful. We have to see that for every local object $Y$ in $\mP(\Delta^+)$ the unit $Y \to G'(F'(Y)) $ is an equivalence.
Since local equivalences between local objects are equivalences, it suffices to see that for every $Y \in \mP(\Delta^+)$ the unit $Y \to G'(F'(Y))$ is a local equivalence. For that it is enough to see that the composition 
$$Y \to G(F(Y)) \to G(F'(Y)) = G'(F'(Y)) $$ of the unit $Y \to G(F(Y))$ and the image under $G$ of the localization $ F(Y) \to F'(Y) $, is a local equivalence. Thus it suffices to see that the unit $Y \to G(F(Y))$ is a local equivalence and that the functor $G$ preserves local equivalences.

We prove first that for every $Y \in \mP(\Delta^+)$ the unit $Y \to G(F(Y))$ is a local equivalence.
Since $F$ and $G$ preserve small colimits and $\mP(\Delta^+)$ is generated under small colimits by representables, and local equivalences are closed under colimits, we can assume that $Y$ is representable.
We have to see that for every $n \geq 0$ the units
$$\Delta^n \to G(F(\Delta^n)) \simeq G(\bDelta^n), \ (\Delta^n)^t \to G(F((\Delta^n)^t)) \simeq G(\tau_{n-1}\bDelta^n) $$ are local equivalences. This is \cref{localeq} and \cref{localeq2}.

It remains to see that the functor $G$ preserves local equivalences.
By \cref{charnerve} a presheaf on $\bDelta^+$ is the $\bDelta^+$-nerve of an $\infty$-category if and only if its image under $G$ is a complicial space.
This implies that a presheaf on $\bDelta^+$ is the $\bDelta^+$-nerve of an $\infty$-category if and only if it is local with respect to
the image under $F$ of the canonical class of generating local equivalences of $\mP(\Delta^+).$ In other words, the image under $F$ of the canonical class of generating local equivalences of $\mP(\Delta^+)$ is a class of generating local equivalences of $\mP(\bDelta^+).$
Hence the functor $G$ preserves local equivalences if it sends 
this class of generating local equivalences of $\mP(\bDelta^+)$ to local equivalences.
This is equivalent to say that the functor $G \circ F: \mP(\Delta^+) \to \mP(\Delta^+) $ sends the canonical class of generating local equivalences of $\mP(\Delta^+)$ to local equivalences.
So it suffices to see that the functor $G \circ F: \mP(\Delta^+) \to \mP(\Delta^+) $ preserves local equivalences.
In the first part of the proof we have shown that for every
$Y \in \mP(\Delta^+) $ the unit $Y \to G(F(Y))$ is a local equivalence.
Hence a morphism $Y \to Z $ in $\mP(\Delta^+)$ is a local equivalence if and only if the induced morphism $G(F(Y)) \to G(F(Z)) $ in $\mP(\Delta^+)$ is a local equivalence since there is a commutative square
\[
\xymatrix{
Y \ar[d]\ar[r] & Z \ar[d]\\
G(F(Y)) \ar[r] & G(F(Z)).
}
\]

It remains to see that an $\infty$-category is univalent if and only if its Street nerve is univalent.
Let $X \in \infty\Cat.$ Then $X$ is univalent if and only if
$X$ is local with respect to to the functor $ S^{n}(J) \to * $ for every $n \geq 0,$ where $J$ is the non-univalent category with two objects and an equivalence.
We have shown in the first part of the proof that $\N: \infty\Cat \to \mP(\Delta^+)$ is fully faithful.
Hence $X$ is local with respect to to the functor $ S^{n}(J) \to * $ for every $n \geq 0$ if and only if $\N(X)$ is local with respect to to the map $ \N(S^{n}(J)) \to * $ for every $n \geq 0$.
\end{proof}

\subsection{The Gray tensor product and lax transformations}
In this subsection we use the adjoints of the Gray tensor to define the notion of lax and oplax natural transformation.

By \cref{Graytensor} the category $\infty\Cat^\Steiner$ carries a canonical monoidal structure denoted by $\boxtimes$ whose tensor unit is the final $\infty$-category, called the Gray tensor product of Steiner $\infty$-categories.
Let $\mC \subset \infty\Cat^\Steiner$
be a dense full subcategory, to which the Gray-monoidal structure restricts.
By density the embedding $\mC \subset \infty\Cat$
induces a localization $\mP(\mC) \rightleftarrows \infty\Cat.$

The following theorem is due to Campion \cite[Theorem 3.4, Theorem 4.1]{campion2023graytensorproductinftyncategories}.
We recall the proof for the reader's convenience.

\begin{proposition}\label{folkloc}

A presheaf on $\infty\Cat_{\mathrm{fg}}^\Steiner$ is in the essential image of the $\infty\Cat_{\mathrm{fg}}^\Steiner$-nerve $\N_\Steiner: \infty\Cat \to \mP(\infty\Cat_{\mathrm{fg}}^\Steiner)$
if and only if it is local with respect to the map $\emptyset \to \N_\Steiner(\emptyset)$ and the maps
$$ \N_\Steiner(\partial\bD^n)  \coprod_{\N_\Steiner(\bD^n)} \N_\Steiner(X) \to \N_\Steiner(\partial\bD^n \coprod_{\bD^n} X) $$
for every $n \geq 0$, finitely generated Steiner $\infty$-category $X$ and inclusion
$\partial\bD^n \to X$ such that the pushout $\partial\bD^n \coprod_{\bD^n} X$ in $\infty\Cat$ is a Steiner $\infty$-category.
    
\end{proposition}

\begin{proof}

Every presheaf on $\infty\Cat_{\mathrm{fg}}^\Steiner$ is local with respect to the maps of the statement since $\infty\Cat_{\mathrm{fg}}^\Steiner \subset \infty\Cat$ is dense and so the maps of the statement are inverted by the left adjoint extension of the embedding $\infty\Cat_{\mathrm{fg}}^\Steiner \subset \infty\Cat$, which is left adjoint to $\N_\Steiner.$

Conversely, let $X$ be a presheaf on $\infty\Cat_{\mathrm{fg}}^\Steiner$.
The canonical embedding $i: \Theta \subset \infty\Cat$ induces an adjunction $i^*: \mP(\infty\Cat_{\mathrm{fg}}^\Steiner) \rightleftarrows \mP(\Theta) i_*. $
The assumptions imply that the $\Theta$-space $i^*(X)$ satifies the Segal condition and so is the $\Theta$-nerve of an $\infty$-category $\mC$.
By density of $\Theta,$ the $\infty\Cat_{\mathrm{fg}}^\Steiner$-space $i_*(i^*(X))$
is the $\infty\Cat^\Steiner$-nerve of $\mC$ and so by the first part of the proof is local with respect to the maps of the statement.
To complete the proof, it suffices to see that the unit $X \to i_*(i^*(X))$ is an equivalence. By the triangle identities, the unit $X \to i_*(i^*(X))$ induces an equivalence under $i^*$ and so induces an equivalence at any disk. So it suffices to show that a map $X \to Y $ of $\infty\Cat^\Steiner$-spaces which are local with respect to the maps of the statement, is an equivalence if it induces equivalences at any disk.
Such a map also induces equivalences at any boundary of a disk
since for every $n \geq 0$ there is a canonical equivalence
$ \partial\bD^{n+1} \simeq \bD^n \coprod_{\partial\bD^{n}} \bD^n$
and $\partial\bD^0= \emptyset.$
So the map $X \to Y $ of $\infty\Cat_{\mathrm{fg}}^\Steiner$-spaces is an equivalence 
because every finitely generated Steiner $\infty$-category is a loopfree polygraph and so is built by finitely many iterated cell attachments along inclusions from the empty category.  
\end{proof}

\begin{theorem}\label{locmon}Let $\mC \subset \infty\Cat^\Steiner_\mathrm{fg}$
be a dense full subcategory, to which the Gray-monoidal structure restricts.
\begin{enumerate}[\normalfont(1)]\setlength{\itemsep}{-2pt}
\item The convolution monoidal structure descends along 
the localization
$$\mP(\mC) \rightleftarrows \infty\Cat. $$

\item The convolution monoidal structure descends along the localization
$$ \mP_{\Set}(\mC)\simeq\mP(\mC)\otimes{\Set}\rightleftarrows \infty\Cat^{\mathrm{strict}}. $$
\end{enumerate}

\end{theorem} 

\begin{proof}We prove (1). The proof of (2) is similar but can also be proven elementary \cite[Théorème A.15]{Dimitri_Ara_2020}.
We first assume that $\mC=  \infty\Cat^\Steiner_\mathrm{fg}$.
We first prove that the localization is compatible with the convolution monoidal structure.
For this we have to see that for every $Z \in \mP(\infty\Cat^{\Steiner}_\mathrm{fg})$
the functors $$(-) \boxtimes Z, Z \boxtimes (-) : \mP(\infty\Cat^{\Steiner}_\mathrm{fg}) \to \mP(\infty\Cat^{\Steiner}_\mathrm{fg})$$
preserve local equivalences.
We prove the first case, the second case is similar.

Let $j: \infty\Cat^{\Steiner}_\mathrm{fg} \to \mP(\infty\Cat^{\Steiner}_\mathrm{fg})$ be the Yoneda-embedding, which is monoidal for the convolution.
Since the category $\mP(\infty\Cat^{\Steiner}_\mathrm{fg})$ is generated by the representables under small colimits, and the convolution monoidal structure is compatible with small colimits, we can assume that $Z$ is representable.
By \cref{folkloc} we have to prove that the map
$\emptyset \simeq \emptyset \boxtimes Z \to j(\emptyset) \boxtimes Z \simeq j(\emptyset \boxtimes Z) \simeq j(\emptyset)$ 
is a local equivalence, which trivially holds, and that for every finitely generated Steiner $\infty$-category $Y$ and inclusion $\bD^n \to Y$ such that the pushout $ \bD^n \coprod_{\partial \bD^n} Y $ in $\infty\Cat$ is a Steiner $\infty$-category, the following map is a local equivalence:
\begin{align*}
\left(j(\bD^n)\coprod_{j(\partial\bD^n)}j(A)\right)\boxtimes j(Z) \simeq
j(\bD^n)\boxtimes j(Z) \coprod_{j(\partial\bD^n)\boxtimes j(Z)}j(A)\boxtimes j(Z)&\simeq
j(\bD^n\boxtimes Z)\coprod_{j(\partial\bD^n\boxtimes Z)}j(A\boxtimes Z)\\
\xrightarrow{\alpha} j\left(\bD^n\boxtimes Z\coprod_{\partial\bD^n\boxtimes Z}A\boxtimes Z\right)\xrightarrow{\beta} j((\bD^n\coprod_{\partial\bD^n} A)\boxtimes  Z)&\simeq j(\bD^n\coprod_{\partial\bD^n} A)\boxtimes j(Z).
\end{align*}
The morphism $\alpha$ is evidently a local equivalence.
So it suffices to see that $\beta$ is an equivalence.
The functor $(-)\boxtimes Z: \infty\Cat^\strict \to \infty\Cat^\strict $
preserves colimits as a left adjoint. So the following commutative square 
\[
\xymatrix{
\partial\bD^n \boxtimes  Z \ar[d]\ar[r] & A \boxtimes  Z \ar[d]\\
\bD^n \boxtimes  Z \ar[r] & (\bD^n\coprod_{\partial\bD^n} A)\boxtimes  Z
}
\]
in $\infty\Cat^\strict$ is a pushout square in $\infty\Cat^\strict$.
By \cref{pasting1} this commutative square is also a pushout square in $\infty\Cat$ because the functor $(-)\boxtimes Z: \infty\Cat^\strict \to \infty\Cat^\strict $ preserves the saturated class generated by the inclusions $\partial \bD^n \subset \bD^n$ for $n \geq 0$ by \cite[Theorem 5.6]{ara.folkmodel}.

Now let $\mC$ be arbitrary. 
By the first part of the proof the localization functor
$\mP(\infty\Cat^{\Steiner}_\mathrm{fg}) \to \infty\Cat$ is monoidal.
Thus the embedding $\infty\Cat^{\Steiner}_\mathrm{fg} \subset \infty\Cat$ is monoidal since the Yoneda embedding is monoidal.
By density the monoidal embedding $\mC \subset \infty\Cat^{\Steiner}_\mathrm{fg} \subset \infty\Cat$ uniquely extends to a left adjoint monoidal functor
$\mP(\mC) \rightleftarrows \infty\Cat$, which refines the localization functor.
\end{proof}

We recover the following result of Campion \cite{campion2022cubesdenseinftyinftycategories}.

\begin{corollary}\label{locmon2}
Let $\cube \subset \infty\Cat$ denote the full subcategory of oriented cubes, equipped with the Gray monoidal structure.
\begin{enumerate}[\normalfont(1)]\setlength{\itemsep}{-2pt}
\item The convolution monoidal structure descends along 
the localization $$\mP(\cube) \rightleftarrows \infty\Cat.$$

\item The convolution monoidal structure descends along 
the localization $$ \mP_\Set(\cube) \rightleftarrows \infty\Cat^{\mathrm{strict}}.$$

\end{enumerate}

\end{corollary}

\begin{remark}\label{stricto}
	
The monoidal localization $\pi_0 : \mS \leftrightarrows \Set$ gives rise to a monoidal localization
$$ \Fun(\cube^\op, \mS) \leftrightarrows \Fun(\cube^\op,\Set) $$ that descends to a monoidal localization
$\infty\Cat \leftrightarrows \infty\Cat^\strict$ for the Gray monoidal structures.
In particular, the embedding $\infty\Cat^\strict \subset \infty\Cat$ is lax monoidal for the Gray monoidal structures.
	
\end{remark}

\begin{notation}
	
Since the Gray tensor product defines a presentably monoidal structure on $\infty\Cat$, it is closed: for every $\infty$-category $\mC$ the functor $$ \mC \boxtimes (-): \infty\Cat \to \infty\Cat$$ admits a right adjoint $\Fun^\lax(\mC,-)$, and the functor $$ (-) \boxtimes \mC : \infty\Cat \to \infty\Cat$$ admits a right adjoint $\Fun^\oplax(\mC,-)$.
\end{notation}

\begin{definition}Let $\F,\G:\mC \to \mD $ be functors of $\infty$-categories.
\begin{enumerate}[\normalfont(1)]\setlength{\itemsep}{-2pt}
\item A lax natural transformation $\F \to \G$ is a morphism in $\Fun^\lax(\mC,\mD)$.

\item An oplax natural transformation $\F \to \G$ is a morphism in $\Fun^\oplax(\mC,\mD)$.
\end{enumerate}
\end{definition}

\begin{remark}
	
Let $\mC,\mD$ be $\infty$-categories. The canonical functors $\mC \to *, \mD \to *$
give rise to functors $\mC \boxtimes \mD \to \mC \boxtimes * \simeq \mC, \mC \boxtimes \mD \to * \boxtimes \mD \simeq \mD$ and so to a functor $\mC \boxtimes \mD \to \mC \times \mD.$
By adjointess the latter functor induces functors
$$ \Fun(\mC,\mD)\to \Fun^\lax(\mC,\mD)\qquad\text{and}\qquad\Fun(\mC,\mD)\to \Fun^\oplax(\mC,\mD).$$
\end{remark}

\begin{remark}\label{lao}\label{grayspace}
If $\mC$ is an $\n$-category and $\mD$ an $\m$-category for $\n,\m \geq 0$,
then $\mC \boxtimes \mD$ is an $\n+\m$-category.
This holds since $\n\Cat$ is closed under small colimits in $\infty\Cat$ and $\n\Cat$ is generated under small colimits by the oriented $\ell$-cubes for $1 \leq \ell \leq \n$ by \cref{cubicaldense}.
For $n=m=0$ one finds that the Gray-monoidal structure restricts to the 
full subcategory of spaces $\mS \subset \infty\Cat$.
The restricted Gray-monoidal structure on $\mS$ is the cartesian structure since $\mS$ is generated in $\infty\Cat$ under small colimits by the final category, the tensor unit for the Gray tensor product and the cartesian product.
Moreover the left adjoint $\tau_0: \infty\Cat \to \mS$ of the monoidal embedding $\mS \subset \infty\Cat$ is monoidal since
$\infty\Cat$ is generated by the cubes under small colimits and
the image of any cube under $\tau_0$ is contractible.

\end{remark}

\begin{proposition}\label{dua}
	
There are canonical monoidal involutions
$$(-)^\op, (-)^\co: (\infty\Cat, \boxtimes) \simeq (\infty\Cat, \boxtimes)^\rev. $$	
$$(-)^\op, (-)^\co: (\infty\Cat^{\mathrm{strict}}, \boxtimes) \simeq (\infty\Cat^{\mathrm{strict}}, \boxtimes)^\rev. $$	

There are commutative squares of lax monoidal embeddings:
$$\begin{xy}
\xymatrix{
\infty\Cat^{\mathrm{strict}} \ar[d]^{(-)^\op} \ar[r] & \infty\Cat  \ar[d]^{(-)^\op}
\\ 
\infty\Cat^{\mathrm{strict}} \ar[r] & \infty\Cat,
}
\end{xy}
\begin{xy}
\xymatrix{
\infty\Cat^{\mathrm{strict}} \ar[d]^{(-)^\co} \ar[r] & \infty\Cat \ar[d]^{(-)^\co}
\\ 
\infty\Cat^{\mathrm{strict}} \ar[r] & \infty\Cat.
}
\end{xy}
$$

\end{proposition}

\begin{proof}
The first part follows from \cref{duac} by density.
The lax monoidal embedding $\infty\Cat^{\mathrm{strict}} \to \infty\Cat$ factors as lax monoidal functors $$\infty\Cat^{\mathrm{strict}} \subset \Fun((\infty{\Cat^\Steiner})^\op,\Set) \subset \Fun((\infty{\Cat^\Steiner})^\op,\mS) \to \infty\Cat$$ and so preserves the monoidal equivalences $(-)^\op, (-)^\co.$
\end{proof}

\begin{corollary}\label{grayhoms}
Let $\mC,\mD \in \infty\Cat$.
There are canonical equivalences $$ \Fun^\oplax(\mC,\mD)^\op \simeq \Fun^\lax(\mC^\op,\mD^\op), $$$$\Fun^\oplax(\mC,\mD)^\co \simeq \Fun^\lax(\mC^\co,\mD^\co).$$	
\end{corollary}

\begin{proof}
Let $\mB \in  \infty\Cat$.
The desired equivalence is represented by the following equivalence
$$ \Map_{\infty\Cat}(\mB, \Fun^\oplax(\mC,\mD)) \simeq \Map_{\infty\Cat}(\mB \boxtimes \mC, \mD) \simeq \Map_{\infty\Cat}((\mB \boxtimes \mC)^\op, \mD^\op)  \simeq $$$$\Map_{\infty\Cat}(\mC^\op \boxtimes \mB^\op, \mD^\op) \simeq \Map_{\infty\Cat}(\mB^\op, \Fun^\lax(\mC^\op,\mD^\op))$$
induced by the equivalence of \cref{dua}.
The second equivalence is proven the same way.
\end{proof} 

\begin{remark}

By construction the functor $(-)^\op: \infty\Cat \to \infty\Cat $ factors as $$\infty\Cat \simeq {\infty\Cat}\mathrm{-}\Cat \xrightarrow{(-)^\co_!} {\infty\Cat}\mathrm{-}\Cat \xrightarrow{(-)^\circ} {\infty\Cat}\mathrm{-}\Cat \simeq \infty\Cat$$ and the functor $(-)^\co: \infty\Cat \to \infty\Cat$ factors as $$\infty\Cat \simeq {\infty\Cat}\mathrm{-}\Cat \xrightarrow{(-)^\op_!} {\infty\Cat}\mathrm{-}\Cat \simeq \infty\Cat.$$

The involutions $(-)^\op, (-)^\co: \infty\Cat \to \infty\Cat$ of \cref{ruik} restrict to those of \cref{dua}
on $\infty\Cat^\Steiner.$
So by density the involutions $(-)^\op, (-)^\co: \infty\Cat \to \infty\Cat$ of \cref{ruik} are equivalent to those of \cref{dua}.

\end{remark}

\subsection{The suspension formula}
We prove a formula for the suspension in terms of the Gray tensor product.

\begin{theorem}\label{cor:weakvsstrictsuspension}\label{thas2}
For every $\infty$-category $X$ there is a pushout square in $\infty\Cat:$
\[
\xymatrix{
X \coprod X \ar[r]\ar[d] & X \boxtimes \bD^1 \ar[d]\\
\bD^0 \coprod \bD^0 \ar[r] & S(X).
}
\]
\end{theorem}

\begin{proof}
For every pair of Steiner $\infty$-categories $X, Y$ there is a natural commutative square in $\infty\Cat^\strict:$
\begin{equation}\label{70}
\xymatrix{
X \coprod X \ar[r]\ar[d] & X \boxtimes \bD^1 \ar[d]\\
\bD^0 \coprod \bD^0 \ar[r] & S(X).
}
\end{equation}
Let $\mathfrak{S}:\infty\Cat \to\infty\Cat$
be the functor sending $X$ to 
$$ \mathfrak{S}(X) := (\bD^0 \coprod \bD^0) \coprod_{X \coprod X} X \boxtimes \bD^1. $$

The functors \begin{equation}\label{func10}
S, \mathfrak{S}: \infty\Cat \to \infty\Cat_{\partial\bD^1 /} \end{equation} preserve small colimits.
The latter since it preserves weakly contractible colimits and the initial object.
Hence by density of Steiner $\infty$-categories the natural commutative square \ref{70} extends to a natural commutative square in $\infty\Cat$ for $X$ any $\infty$-category.
This natural commutative square gives rise to a functor
$$\rho:  \mathfrak{S}(X) \to S(X), $$ which is natural in $X \in \infty\Cat. $
We like to see that $\rho$ is an equivalence.
Since the functors \ref{func10} preserve small colimits, by generation of the disks we can reduce to the case that $X $ is a disk.

So we have to see that for every $n \geq 0$ the functor
$\rho: \mathfrak{S}(\bD^n) \to S(\bD^n) = \bD^{n+1}$ is an equivalence,
which we prove 
by induction on $ n \geq 0$.
For $n = 0$ the functor $\rho$ is the identity.
We prove the statement for $n$ assuming that $\rho: \mathfrak{S}(\bD^m) \to S(\bD^m) = \bD^{m+1}$ is an equivalence for every $ m < n$.
Since $\partial\bD^n$ is the colimit of disks of dimension smaller $n$ and the functors \ref{func10} preserve small colimits, $\rho: \mathfrak{S}(\partial\bD^n) \to S(\partial\bD^n) = \partial\bD^{n+1}$ is an equivalence.
Hence the commutative diagram
\begin{equation}\label{o1.5}
\xymatrix{
\partial\bD^n \coprod \partial\bD^n \ar[r]\ar[d] & \bD^1 \boxtimes \partial\bD^n \ar[d]\\
\bD^0 \coprod \bD^0 \ar[r] & \partial\bD^{n+1}}
\end{equation}
is a pushout square.
This commutative square agrees with the outer square of the following commutative diagram:
\begin{equation}\label{o2.5}
\xymatrix{
\partial\bD^n \coprod \partial\bD^n \ar[r]\ar[d] & \bD^1 \boxtimes \partial\bD^n \ar[d]\\
\bD^n \coprod \bD^n \ar[r] \ar[d] & \iota_n(\bD^1 \boxtimes \bD^n) \ar[d] 
\\
\bD^0 \coprod \bD^0 \ar[r] & \partial\bD^{n+1}.}
\end{equation}
The top commutative square is a pushout square by \cref{cell2}.
Thus by the pasting law also the bottom commutative square is a pushout square.

Consider the following commutative diagram:
\begin{equation}\label{o3.5}
\xymatrix{
\partial\bD^{n+1} \ar[r]\ar[d] & \bD^{n+1} \ar[d] 
\\\iota_n(\bD^1 \boxtimes \bD^n) \ar[d] \ar[r] & \bD^1 \boxtimes \bD^n \ar[d]
\\
\partial\bD^{n+1} \ar[r] & \bD^{n+1}.}
\end{equation}
The outer commutative square is trivially a pushout square since both vertical functors are the identities.
The top commutative square is a pushout square by \cref{cello}.
So by the pasting law also the bottom commutative square is a pushout square. Consequently, the outer square in the following commutative square is a pushout square:
\begin{equation}\label{o4}
\xymatrix{
\bD^n \coprod \bD^n \ar[r] \ar[d] & \iota_n(\bD^1 \boxtimes \bD^n) \ar[r]\ar[d] & \bD^1 \boxtimes \bD^n \ar[d] 
\\
\bD^0 \coprod \bD^0 \ar[r] & \partial\bD^{n+1} \ar[r] & \bD^{n+1}.}
\end{equation}
\end{proof}

The following proposition follows from work of Campion \cite[Theorem 3.14.]{campion2023graytensorproductinftyncategories}, which we recall for the reader's convenience.

\begin{proposition}

The localization $\infty\Cat \rightleftarrows \infty\Cat^\univ$ is monoidal.
    
\end{proposition}

\begin{proof}

We have to see that for every $X \in \infty\Cat$ the functors $(-)\boxtimes X, X \boxtimes(-) : \infty\Cat \to \infty\Cat $ preserve local equivalences.

By \cref{cubicaldense} the category $\infty\Cat$ is generated by the oriented cubes under small colimits. Since the collection of local equivalences is closed under small colimits, we can assume that $X = \cube^m$ for some $m \geq 0.$
Since the Gray tensor product is associative, we can assume that $X = \bD^1.$
We prove that the functor $(-)\boxtimes\bD^1: \infty\Cat \to \infty\Cat $ preserves local equivalences. The other case is similar.
It suffices to see that the functor $(-)\boxtimes\bD^1: \infty\Cat \to \infty\Cat $ preserves the generating local equivalences.
For that we have to see that for every $n \geq 0$ 
the induced functor $ S^n(J) \boxtimes\bD^1 \to S^n(*) \boxtimes\bD^1 = \bD^n \boxtimes\bD^1 $ is a local equivalence, where $J$ is the non-univalent $(1,1)$-category with two objects and one isomorphism between them.
We prove this statement by induction on $n \geq 0.$
We start with the induction start $n=1.$

The functor $ J \boxtimes \bD^1 \to * \boxtimes \bD^1 \simeq \bD^1 $
factors as $  J \boxtimes \bD^1 \to S(J) \to S(*) = \bD^1,  $
where the first functor is the canonical one and the second functor
is the image under $S$ of the functor $J \to *,$ which is a local equivalence. So it suffices to see that the first functor in the composition is a local equivalence.

By \cref{thas2} there is a pushout square 
\[
\xymatrix{
J \boxtimes \partial\bD^1  \ar[d]\ar[r] & \partial\bD^1 \ar[d]\\
J \boxtimes \bD^1 \ar[r] & S(J)
}
\]
So the bottom functor is a local equivalence since the functor
$J \to *$ and so also the functor $J \coprod J \to * \coprod *$ are local equivalences.

This proves the induction start. We continue with the induction step.
Let $n > 1$ and let $\theta : S^{n-1}(J) \to S^{n-1}(*) $ be the canonical functor, which is a generating local equivalence. We assume that the functor $$ \theta \boxtimes \bD^1: S^{n-1}(J) \boxtimes\bD^1 \to S^{n-1}(*) \boxtimes\bD^1 $$ is a local equivalence.
We prove that the functor $$ S(\theta) \boxtimes \bD^1: S^n(J) \boxtimes\bD^1 \to S^n(*) \boxtimes\bD^1 $$ is a local equivalence.
Since the collection of local equivalences is closed under colimits, also $\theta \boxtimes \partial\bD^1 = \theta \coprod \theta$ is a local equivalence. 
The functor $S: \infty\Cat \to \infty\Cat$, which preserves weakly contractible colimits, preserves the generating local equivalences and so all local equivalences. Hence also $S(\theta), S(\theta \boxtimes \partial\bD^1), S(\theta \boxtimes \bD^1)$ are local equivalences. 

By \cite[Theorem 4.2.11.]{gepner2025oriented} there is a canonical equivalence
$$ S(\theta) \boxtimes \bD^1 \simeq (\bD^1 \vee S(\theta)) \coprod_{\partial\bD^1} ( S(\theta) \vee \bD^1) \coprod_{S(\theta \boxtimes \partial\bD^1)} S(\theta \boxtimes \bD^1). $$

This proves the induction step since the collection of local equivalences is closed under colimits.
\end{proof}

\subsection{Join and slice}

\begin{notation} For every small category $\mC$ that admits an initial object let $ \mP_{\mathrm{red}}(\mC) \subset \mP(\mC) $  
be the full subcategory of presheaves on $\mC$
that are reduced, i.e. send the initial object to the final space.
\end{notation}

\begin{remark}\label{gener}

The full subcategory $ \mP_{\mathrm{red}}(\mC) \subset \mP(\mC)$ is the localization with respect to the single morphism $\emptyset \to \iota(\emptyset),$
where $\iota: \mC \to \mP(\mC)$ is the Yoneda embedding.
In particular, the Yoneda embedding $\iota$
lands in $ \mP_{\mathrm{red}}(\mC) $ and preserves the initial object
and $\mP_{\mathrm{red}}(\mC)$ is generated under small colimits by the representables.
     
\end{remark}

\begin{remark}\label{weakco}
The full subcategory $ \mP_{\mathrm{red}}(\mC) $ is closed in $ \mP(\mC) $ under small weakly contractible colimits.
    
\end{remark}

\begin{lemma}\label{joinconv}
Let $\mC $ be a small monoidal category that admits an initial object.
The convolution monoidal structure on $\mP(\mC)$ restricts to a monoidal structure on $ \mP_{\mathrm{red}}(\mC).$

\end{lemma}

\begin{proof}
Since the Yoneda embedding of $\mC$ is monoidal, the tensor unit of the convolution monoidal structure on $\mP(\mC)$ is representable
and so belongs to $\mP_{\mathrm{red}}(\mC)$.
For every $X,Y\in\mP(\mC)$ there is a canonical equivalence
\begin{align*}
(X \ot Y)(\emptyset)\simeq &\underset{(A,B) \in\mC \times \mC}{\colim} X(A)\times Y(B) \simeq (\colim_{A\in \mC} X(A)) \times (\colim_{B \in\mC} Y(B)) \simeq X(\emptyset) \times Y(\emptyset).
\end{align*}
\end{proof}

\begin{remark}
Let $\mC $ be a small monoidal category that admits an initial object.
The Yoneda embedding $\iota: \mC \to \mP(\mC)$ is monoidal so that also the induced embedding
$\iota: \mC \to \mP_{\mathrm{red}}(\mC)$ is monoidal.
In particular, by \cref{gener} the tensor unit of $\mP_{\mathrm{red}}(\mC)$ is initial if the tensor unit of $\mC$ is initial. 
\end{remark}\label{yonemon}

\begin{lemma}\label{joinconv2}Let $\mC,\mD $ be small monoidal categories whose tensor units are initial.
Let $i: \mC \subset \mD $ be a monoidal embedding.
The left adjoint embedding $i_!: \mP_{\mathrm{red}}(\mC) \to \mP_{\mathrm{red}}(\mD)$ refines to a monoidal embedding.
\end{lemma}

\begin{proof}
The monoidal embedding $i: \mC \subset \mD$ uniquely extends to a left adjoint monoidal embedding $i_!: \mP(\mC) \to \mP(\mD)$.
This functor sends the unique morphism $\emptyset \to \iota_\mC(\emptyset)$ to the unique morphism $\emptyset \to \iota_\mD(\emptyset)$ and so descends to a functor
$i_!: \mP_{\mathrm{red}}(\mC) \to \mP_{\mathrm{red}}(\mD)$.
This functor factors as $\mP_{\mathrm{red}}(\mC) \subset \mP(\mC) \xrightarrow{i_!} \mP(\mD) \to \mP_{\mathrm{red}}(\mD)$, where the last functor is the localization functor. By \cref{joinconv} the embeddings $\mP_{\mathrm{red}}(\mC) \subset \mP(\mC), \mP_{\mathrm{red}}(\mD) \subset \mP(\mD)$ are monoidal. Hence the localization functor
$\mP(\mD) \to \mP_{\mathrm{red}}(\mD)$ is oplax monoidal as it is the left adjoint of the latter. 
Hence the composition $\phi: \mP_{\mathrm{red}}(\mC) \subset \mP(\mC) \xrightarrow{i_!} \mP(\mD) \to \mP_{\mathrm{red}}(\mD)$ is oplax monoidal.
We complete the proof by showing that this composition is monoidal.
For this we have to see that for every $A,B \in \mP(\mC)$
the structure morphism $\phi(A\ot B) \to \phi(A) \ot \phi(B)$ 
is an equivalence.
Let $A \in \mP(\mC)$. The full subcategory of $\mP(\mC)$ of all $B$ such that the structure morphism is an equivalence contains the initial object, which is the tensor unit, and is closed under small weakly contractible colimits using \cref{weakco}, and so is closed under small colimits.
Since $\mP_{\mathrm{red}}(\mC)$ is generated by the representables (\cref{gener}),
we can assume that $B$ is representable.
Let $B \in \mC$. The full subcategory of $\mP(\mC)$ of all $A$ such that the structure morphism is an equivalence contains the initial object, which is the tensor unit, and is closed under small weakly contractible colimits and so is closed under small colimits.
So we can assume that $A, B$ are representable.
In this case the structure morphism identifies with the one of $i: \mC \to \mD$ using \cref{yonemon} and so is an equivalence.
\end{proof}

Let $\mC $ be a full subcategory of $\infty\Cat^\Steiner$,
to which the join monoidal structure restricts, and which is dense
in $\infty\Cat$.
Since $\mC \subset \infty\Cat$ is dense,
the embedding $\mC \subset \infty\Cat$
induces a localization
$$ \mP(\mC) \rightleftarrows \infty\Cat $$
that restricts to a localization $ \mP_{\mathrm{red}}(\mC) \rightleftarrows \infty\Cat. $

\begin{theorem}\label{locmon3}Let $\mC $ be a full subcategory of $\infty\Cat_{\mathrm{fg}}^\Steiner$, to which the join monoidal structure restricts, and which is dense in $\infty\Cat$.
\begin{enumerate}[\normalfont(1)]\setlength{\itemsep}{-2pt}
\item The monoidal structure of \cref{joinconv} descends along the 
localization $$\mP_{\mathrm{red}}(\mC) \rightleftarrows \infty\Cat.$$
\item The monoidal structure of \cref{joinconv}
descends along the localization $$ \mP_{\mathrm{red}}(\mC)\otimes\Set \rightleftarrows \infty\Cat^{\mathrm{strict}}. $$
\end{enumerate}

\end{theorem} 

\begin{proof}We prove (1). The proof of (2) is similar.
We first assume that $\mC = \infty\Cat_{\mathrm{fg}}^\Steiner.$
We have to see that for every $Z \in \mP_{\mathrm{red}}(\infty\Cat_{\mathrm{fg}}^\Steiner)$
the functors $$(-) \star Z, Z \star (-) : \mP_{\mathrm{red}}(\infty\Cat_{\mathrm{fg}}^\Steiner) \to \mP_{\mathrm{red}}(\infty\Cat_{\mathrm{fg}}^\Steiner)$$ preserve local equivalences.
We prove the first case, the second case is similar.
 
The full subcategory of $Z \in \mP_{\mathrm{red}}(\infty\Cat_{\mathrm{fg}}^\Steiner)$ such that $(-) \star Z:  \mP_{\mathrm{red}}(\infty\Cat_{\mathrm{fg}}^\Steiner) \to \mP_{\mathrm{red}}(\infty\Cat_{\mathrm{fg}}^\Steiner)$ preserves local equivalences, contains the empty category, which is the tensor unit of $\mP_{\mathrm{red}}(\infty\Cat_{\mathrm{fg}}^\Steiner)$, and is closed under small weakly contractible colimits and so is closed under small colimits.
By \cref{gener} the category $\mP_{\mathrm{red}}(\infty\Cat_{\mathrm{fg}}^\Steiner) $ is generated by the representables under small colimits.
Hence we can assume that $Z$ is representable. By \cref{folkloc} we have to prove that the map
$Z \simeq \emptyset \star Z \to j(\emptyset) \star Z \simeq j(\emptyset \star Z) \simeq Z $ 
is a local equivalence, which is the identity, and that for every finitely generated Steiner $\infty$-category $Y$ and inclusion $\bD^n \to Y$ such that the pushout $ \bD^n \coprod_{\partial \bD^n} Y $
in $\infty\Cat$ is a Steiner $\infty$-category, the following map is a local equivalence, where $j$ denotes the Yoneda embedding:
$$
\left(j(\bD^n)\coprod_{j(\partial\bD^n)}j(A)\right)\star j(Z) \simeq
j(\bD^n)\star j(Z) \coprod_{j(\partial\bD^n)\star j(Z)}j(A)\star j(Z)\simeq
j(\bD^n\star Z)\coprod_{j(\partial\bD^n\star Z)}j(A\star Z)$$
$$
\xrightarrow{\alpha} j\left(\bD^n\star Z\coprod_{\partial\bD^n\star Z}A\star Z\right)\xrightarrow{\beta} j((\bD^n\coprod_{\partial\bD^n} A)\star  Z)\simeq j(\bD^n\coprod_{\partial\bD^n} A)\star j(Z).
$$
The morphism $\alpha$ is evidently a local equivalence. 
So it suffices to see that $\beta$ is an equivalence.

The functor $(-)\star Z: \infty\Cat^\strict \to \infty\Cat^\strict $
preserves pushouts. So the following commutative square 
\[
\xymatrix{
\partial\bD^n \star  Z \ar[d]\ar[r] & A \star  Z \ar[d]\\
\bD^n \star  Z \ar[r] & (\bD^n\coprod_{\partial\bD^n} A)\star  Z
}
\]
in $\infty\Cat^\strict$ is a pushout square in $\infty\Cat^\strict$.
By \cref{pasting1} this commutative square is also a pushout square in $\infty\Cat$ because the functor $(-)\star Z: \infty\Cat^\strict \to \infty\Cat^\strict $ preserves the saturated class generated by the inclusions $\partial \bD^n \subset \bD^n$ for $n \geq 0$ by \cite[Proposition 7.5, Theorem 7.14]{ara.folkmodel}.

Now let $\mC$ be arbitrary and $i: \mC \subset \infty\Cat_{\mathrm{fg}}^\Steiner $ the monoidal embedding.
Since $\mC$ is dense in $\infty\Cat$, the restricted nerve functor
$ \infty\Cat \to \mP_{\mathrm{red}}(\infty\Cat_{\mathrm{fg}}^\Steiner)$
factors as the restricted nerve functor
$ \infty\Cat \to \mP(\mC) $ followed by the embedding $\mP(\mC) \to \mP(\infty\Cat_{\mathrm{fg}}^\Steiner)$ taking right Kan-extensions.
Hence by adjointness the localization functor
$\mP(\infty\Cat_{\mathrm{fg}}^\Steiner) \to \infty\Cat$ factors as
$i^*: \mP(\infty\Cat_{\mathrm{fg}}^\Steiner) \to \mP(\mC)$ followed by the
localization functor $\mP(\mC) \to \infty\Cat$.
The functor $i^*: \mP(\infty\Cat_{\mathrm{fg}}^\Steiner) \to \mP(\mC)$
restricts to a functor $i^*: \mP_{\mathrm{red}}(\infty\Cat_{\mathrm{fg}}^\Steiner) \to \mP_{\mathrm{red}}(\mC)$. Hence the localization functor
$\mP_{\mathrm{red}}(\infty\Cat_{\mathrm{fg}}^\Steiner) \to \infty\Cat$, the restriction of the previous localization functor, factors as
$i^*: \mP_{\mathrm{red}}(\infty\Cat_{\mathrm{fg}}^\Steiner) \to \mP_{\mathrm{red}}(\mC)$ followed by the
localization functor $\mP_{\mathrm{red}}(\mC) \to \infty\Cat$.

Thus the functor $i_!: \mP_{\mathrm{red}}(\mC) \to \mP_{\mathrm{red}}(\infty\Cat_{\mathrm{fg}}^\Steiner)$ followed by the localization functor $\mP_{\mathrm{red}}(\infty\Cat_{\mathrm{fg}}^\Steiner) \to \infty\Cat$
factors as 
$i_!: \mP_{\mathrm{red}}(\mC) \to \mP_{\mathrm{red}}(\infty\Cat_{\mathrm{fg}}^\Steiner)$ followed by $i^*: \mP_{\mathrm{red}}(\infty\Cat_{\mathrm{fg}}^\Steiner) \to \mP_{\mathrm{red}}(\mC)$ followed by the
localization functor $\mP_{\mathrm{red}}(\mC) \to \infty\Cat$, which is the localization functor $\mP_{\mathrm{red}}(\mC) \to \infty\Cat$.
By the first part of the proof the localization functor $\mP_{\mathrm{red}}(\infty\Cat_{\mathrm{fg}}^\Steiner) \to \infty\Cat$ is monoidal.
By \cref{joinconv2} the embedding $i_!: \mP_{\mathrm{red}}(\mC) \to \mP_{\mathrm{red}}(\infty\Cat_{\mathrm{fg}}^\Steiner)$ is monoidal.
\end{proof}

The join monoidal structure on $\infty\Cat^\Steiner$ restricts to
$ \bDelta_{\geq -1}$ and $ \bDelta_{\geq -1}$ is a dense subcategory of $\infty\Cat.$
We obtain the following corollary:

\begin{corollary}\label{locmon4}
\begin{enumerate}[\normalfont(1)]\setlength{\itemsep}{-2pt}
\item The monoidal structure of \cref{joinconv} descends along the localization $\mP_{\mathrm{red}}(\bDelta^+) \rightleftarrows \infty\Cat. $

\item The monoidal structure of \cref{joinconv} descends along the localization $ \mP_{\mathrm{red}}(\bDelta^+)\otimes\Set \rightleftarrows \infty\Cat^{\mathrm{strict}}. $

\end{enumerate}

\end{corollary} 

We also define the antijoin:

\begin{definition}
Let $X, Y \in \infty\Cat.$ The antijoin of $X, Y$ is
$$X \bar{\star} Y := (X^{\co} \star Y^{\co})^{\co} .$$

\end{definition}

\begin{definition} Let $n \geq 0$.
The $n$-th antioriental is $ (\bDelta{^n})^\co $.

\end{definition}

\begin{remark}

Let $n \geq 0$.
The $n$-th antioriental is the $n$-fold antijoin of $\bD^0.$

\end{remark}

Next we define lax and oplax slice $\infty$-categories.
By \cref{locmon3} the categories $\infty\Cat, \infty\Cat^\strict$
carry monoidal structures compatible with small weakly contractible colimits.
In particular, for every $X \in \infty\Cat$ the induced functors
$$ (-) \star X,\quad X \star (-): \infty\Cat \to \infty\Cat$$
preserve small weakly contractible colimits and send the empty category, the tensor unit for the join, to $X.$ Thus the latter functors give rise to functors $$ (-) \star X, \quad X \star (-): \infty\Cat \to \infty\Cat_{X/ }$$ that preserve the initial object and small weakly contractible colimits and so preserve small colimits.
Presentability of $\infty\Cat$ implies that these functors admit a right adjoint:

\begin{definition}Let $X$ be a small $\infty$-category.
\begin{enumerate}[\normalfont(1)]\setlength{\itemsep}{-2pt}
\item The oplax slice, or oplax over $\infty$-category, functor is the right adjoint $$ \infty\Cat_{X/ } \to \infty\Cat, \qquad (F:X \to Y) \mapsto Y_{//^\oplax F}$$ of the functor $ (-) \star X: \infty\Cat \to \infty\Cat_{X/ }.$

\item The lax coslice, or lax under $\infty$-category, functor is the right adjoint $$ \infty\Cat_{X/ } \to \infty\Cat,\qquad (F:X \to Y) \mapsto Y_{F//^\lax }$$ of the functor $ X \star (-): \infty\Cat \to \infty\Cat_{\X / }.$

\item The lax slice, or lax over $\infty$-category, functor is the right adjoint $$ \infty\Cat_{\X / } \to \infty\Cat,\qquad (F:X \to Y) \mapsto Y_{//^\lax F} := (Y^\co_{//^\oplax F^\co})^\co $$
of the functor $ (-) \bar{\star} X: \infty\Cat \to \infty\Cat_{X/ }.$

\item The oplax coslice, or oplax under $\infty$-category, functor is the right adjoint $$ \infty\Cat_{X/ } \to \infty\Cat,\qquad (F:X \to Y) \mapsto Y_{F//^\oplax}:=(Y^\co_{F^\co//^\lax })^\co $$
of the functor $ X \bar{\star} (-): \infty\Cat \to \infty\Cat_{\X / }.$

\end{enumerate}

\end{definition}

\begin{remark}
If unspecified, the slice $\mC_{//F}$ will refer to the oplax slice, and the coslice $\mC_{F//}$ will refer to the oplax coslice.
\end{remark}

\begin{lemma}Let $F: X \to Y $ be a functor.
There is a canonical equivalence of $\infty$-categories
$$ (Y_{//^\oplax F})^\op \simeq (Y^\op)_{F^\op//^\lax }.$$
\end{lemma}

\begin{proof}

There is a canonical equivalence 
$$ (X \star (-))^\op \simeq (-)^\op \star {X^\op}.$$

Hence there is the following canonical equivalence natural in $Z \in \infty\Cat:$
$$ \Map_{\infty\Cat}(Z, (Y_{//^\oplax F})^\op) \simeq  \Map_{\infty\Cat}(Z^\op, Y_{//^\oplax F}) \simeq \Map_{\infty\Cat_{X/}}(Z^\op \ast X, F)$$
$$ \simeq \Map_{\infty\Cat_{X^\op/}}(X^\op \ast Z, F^\op) \simeq  \Map_{\infty\Cat}(Z, (Y^\op)_{F^\op //^\lax }).$$
\end{proof}

\subsection{The join formula}

\begin{theorem}\label{cor:weakvsstrictjoin}
For every pair of $\infty$-categories $X$ and $Y$ there is a pushout square in $\infty\Cat:$
\[
\xymatrix{
X \boxtimes Y \coprod X \boxtimes Y \ar[r]\ar[d] & X\boxtimes \bD^1 \boxtimes Y \ar[d]\\
X \coprod Y \ar[r] & X\star Y.
}
\]
\end{theorem}

\begin{proof}

For every Steiner $\infty$-categories $X, Y$ there is a natural commutative square in $\infty\Cat^\strict:$

\begin{equation}\label{u}
\xymatrix{
X \boxtimes Y \coprod X \boxtimes Y \ar[r]\ar[d] & X\boxtimes \bD^1 \boxtimes Y \ar[d]\\
X \coprod Y \ar[r] & X\star Y,
}
\end{equation}
where $X \star Y$ denotes the join of Steiner $\infty$-categories that we extended to the join of $\infty$-categories. 

Let $\diamond:\infty\Cat\times\infty\Cat\to\infty\Cat$
be the functor sending $(X,Y)$ to 
$$ X \diamond Y := X \coprod Y \coprod_{X \boxtimes Y \coprod X \boxtimes Y} X\boxtimes \bD^1 \boxtimes Y. $$

For any $\infty$-category $Y$ the functors \begin{equation}\label{func1}
(-) \ast Y, (-) \diamond Y, \infty\Cat \to \infty\Cat_{Y /} \end{equation} preserve small colimits because they preserve weakly contractible colimits and the initial object.
Hence by density of Steiner $\infty$-categories the natural commutative square \ref{u} extends to a natural commutative square in $\infty\Cat$ for $X$ an $\infty$-category and $Y$ a Steiner $\infty$-category.
For any $\infty$-category $X$ the functors \begin{equation}\label{func2}
X \ast (-), X \diamond (-), \infty\Cat \to \infty\Cat_{X /} \end{equation} preserve small colimits.
Hence by density of Steiner $\infty$-categories the latter extension further extends to a natural commutative square in $\infty\Cat$ for $X, Y \in \infty\Cat.$ This natural commutative square gives rise to a functor
$$\rho:  X \diamond Y \to X \ast Y, $$ which is natural in $X, Y \in \infty\Cat. $
We like to see that $\rho$ is an equivalence.
Since the functors \ref{func1}, \ref{func2} preserve small colimits, by generation of the orientals we can reduce to the case that $X, Y $ are orientals.

So we have to see that for every $n,m \geq 0$ the canonical functor
$$\rho: \bDelta^n \diamond \bDelta^m \to \bDelta^n \ast \bDelta^m = \bDelta^{n+m+1}$$ is an equivalence. We first reduce to the case that $n=0$.
The functor $$\bDelta^0 \diamond (\bDelta^n \diamond \bDelta^{m}) \simeq (\bDelta^0 \diamond \bDelta^n) \diamond \bDelta^{m} \xrightarrow{\rho \diamond \bDelta^{m} } (\bDelta^0 \ast \bDelta^n) \diamond \bDelta^{m} \xrightarrow{\rho} (\bDelta^0 \ast \bDelta^n) \ast \bDelta^{m} $$ factors as
$$\bDelta^0 \diamond (\bDelta^n \diamond \bDelta^{m}) \xrightarrow{\bDelta^0 \diamond \rho}
\bDelta^0 \diamond (\bDelta^{n} \ast \bDelta^m) \xrightarrow{\rho} \bDelta^0 \ast (\bDelta^{n} \ast \bDelta^m) \simeq (\bDelta^0 \ast \bDelta^{n}) \ast \bDelta^m .$$
Hence it follows by induction on 
$ n \geq 0$ that $\rho $ is an equivalence for any $n, m \geq 0$
if $\rho $ is an equivalence for any $m \geq 0$ and $n=0$.
So we can assume that $n=0$.

We prove by induction on $ n \geq 0$ that the functor $$\rho: \bDelta^0 \diamond \bDelta^n \to \bDelta^0 \ast \bDelta^n = \bDelta^{n+1}$$ is an equivalence.
For $n = 0$ the functor $\rho$ is the identity.
We prove the statement for $n$ assuming that the functor $$\rho: \bDelta^0 \diamond \bDelta^m  \to \bDelta^0 \ast \bDelta^m = \bDelta^{m+1}$$ is an equivalence for every $ m < n$.
Recall that $\partial\bDelta^n=\iota_{n-1}\bDelta^n$ is the maximal $n-1$-category in $\bDelta^n$.
Since $\partial\bDelta^n$ is the colimit of orientals of dimension smaller $n$ and the functors \ref{func2} for $X=\bDelta^0$ preserve small colimits, the functor $$\rho: \bDelta^0  \diamond \partial\bDelta^n \to \bDelta^0 \ast \partial\bDelta^n$$ is an equivalence.

Consider the commutative diagram
\begin{equation}\label{o1}
\xymatrix{
\partial\bDelta^n \coprod \partial\bDelta^n \ar[r]\ar[d] & \bDelta^1 \boxtimes \partial\bDelta^n \ar[d]\\
\bDelta^0 \coprod \partial\bDelta^n \ar[r] \ar[d] & \bDelta^0 \ast \partial\bDelta^n \ar[d] 
\\
\bDelta^0 \coprod \bDelta^n \ar[r] & \iota_n(\bDelta^0 \ast \bDelta^n).}
\end{equation}
The top commutative square is a pushout square since the functor $$\rho: \bDelta^0  \diamond \partial\bDelta^n \to \bDelta^0 \ast \partial\bDelta^n$$ is an equivalence.
The bottom commutative square is a pushout square by \cref{cell}.
Thus also the outer commutative square is a pushout square.
The outer commutative square agrees with the outer square of the following commutative diagram:
\begin{equation}\label{o2}
\xymatrix{
\partial\bDelta^n \coprod \partial\bDelta^n \ar[r]\ar[d] & \bDelta^1 \boxtimes \partial\bDelta^n \ar[d]\\
\bDelta^n \coprod \bDelta^n \ar[r] \ar[d] & \iota_n(\bDelta^1 \boxtimes \bDelta^n) \ar[d] 
\\
\bDelta^0 \coprod \bDelta^n \ar[r] & \iota_n(\bDelta^0 \ast \bDelta^n).}
\end{equation}
The top commutative square is a pushout square by \cref{cell2}.
Thus by the pasting law also the bottom commutative square is a pushout square.

Now consider the following commutative square:
\begin{equation}\label{o3}
\xymatrix{
\partial\bD^{n+1} \ar[r]\ar[d] & \bD^{n+1} \ar[d] 
\\\iota_n(\bDelta^1 \boxtimes \bDelta^n) \ar[d] \ar[r] & \bDelta^1 \boxtimes \bDelta^n \ar[d]
\\
\iota_n(\bDelta^0 \ast \bDelta^n)\ar[r] & \bDelta^0 \ast \bDelta^n.}
\end{equation}
The outer commutative square is a pushout square by \cref{orientdec}.
The top commutative square is a pushout square by \cref{cello}.
So by the pasting law also the bottom commutative square is a pushout square. Consequently, the outer square in the following commutative square is a pushout square:
\begin{equation}\label{o}
\xymatrix{
\bDelta^n \coprod \bDelta^n \ar[r] \ar[d] & \iota_n(\bDelta^1 \boxtimes \bDelta^n) \ar[r]\ar[d] & \bDelta^1 \boxtimes \bDelta^n \ar[d] 
\\
\ast \coprod \bDelta^n \ar[r] & \iota_n(\bDelta^0 \ast \bDelta^n) \ar[r] & \bDelta^0 \ast \bDelta^n.}
\end{equation}
\end{proof}

\begin{corollary}
There is a canonical equivalence
$$Y \bar{\star} X \simeq  X \boxtimes \bD^1 \boxtimes Y \coprod_{X \boxtimes \{0\} \boxtimes Y \coprod X \boxtimes \{1\} \boxtimes Y} Y \coprod X, $$
where the functors $ X \boxtimes \{0\} \boxtimes Y \to Y, X \boxtimes \{1\} \boxtimes Y \to X $ are the projections.
\end{corollary}

\begin{proof}
There is a canonical equivalence
$$Y \bar{\star} X \simeq (Y^\co \boxtimes \bD^1 \boxtimes X^\co \coprod_{Y^\co \boxtimes \{0\} \boxtimes X^\co \coprod Y^\co \boxtimes \{1\} \boxtimes X^\co} Y^\co \coprod X^\co)^\co $$$$ \simeq X \boxtimes \bD^1 \boxtimes Y \coprod_{X \boxtimes \{0\} \boxtimes Y \coprod X \boxtimes \{1\} \boxtimes Y} Y \coprod X, $$
where the functors $ X \boxtimes \{0\} \boxtimes Y \to Y, X \boxtimes \{1\} \boxtimes Y \to X $ are the projections.
\end{proof}

\begin{corollary}

The join monoidal structure descends along the 
localization $$\infty\Cat \rightleftarrows \infty\Cat^\univ.$$

\end{corollary}

\begin{proof}

This follows immediately from the join formula of \cref{cor:weakvsstrictjoin}, that the Gray tensor product on $\infty\Cat $ descends along the 
localization $$\infty\Cat \rightleftarrows \infty\Cat^\univ,$$
which holds by \cref{locmon}, and that the collection of local equivalences is closed under small colimits.
\end{proof}

\begin{remark}
There is a canonical equivalence
$$(Y^{\op} \star X^{\op})^{\op} \simeq X \boxtimes \bD^1 \boxtimes Y \coprod_{X \boxtimes \{1\} \boxtimes Y \coprod X \boxtimes \{0\} \boxtimes Y} X \coprod Y \simeq X \star Y, $$
where the functors $ X \boxtimes \{1\} \boxtimes Y \to Y, X \boxtimes \{0\} \boxtimes Y \to X $ are the projections.
Consequently, there is a canonical equivalence
$$ X \bar{\star} Y=(X^{\co} \star Y^{\co})^{\co} \simeq (Y^{\co\op} \star X^{\co\op})^{\co\op}. $$

\end{remark}

\begin{remark}

By the join formula of \cref{cor:weakvsstrictjoin} there are canonical equivalences
$$ \Fun_*^\lax(\bD^0 \star \mB, (\mC, X)) \simeq \Fun^\lax(\mB, \mC_{X//^\lax}), $$
$$ \Fun_*^\oplax(\mB \star \bD^0, (\mC, X)) \simeq \Fun^\oplax(\mB, \mC_{//^\oplax X }). $$

There are canonical equivalences 
$$ \Fun_*^\oplax(\bD^0 \bar{\star} \mB, (\mC, X)) \simeq \Fun^\oplax(\mB, \mC_{X//^\oplax}), $$
$$ \Fun_*^\lax(\mB \bar{\star} \bD^0, (\mC, X)) \simeq \Fun^\lax(\mB, \mC_{//^\lax X }). $$

\end{remark}

\subsection{The bicone and the orthoplex}
\begin{definition}\label{orientedbicone}
Let $X$ be an $\infty$-category.
The {\em antioriented bicone} $X^\abd$ is the pushout
\[
\xymatrix{
& X\ar[rd]\ar[ld] & \\
\bD^0\,\bar{\star}\,X\ar[rd] & & X\star\bD^0\ar[ld]\\
& X^\abd &
}
\]
of the left anticone and right cone on $X$, glued along the two inclusions from $X$.
\end{definition}
\begin{proposition}
The antioriented bicone functor $(-)^\abd:\infty\Cat\to\infty\Cat_{\partial\bD^1/}$, viewed as a functor from $\infty$-categories to bipointed $\infty$-categories, preserves colimits.
\end{proposition}
\begin{corollary}
The antioriented bicone functor $(-)^\abd:\infty\Cat\to\infty\Cat_{\partial\bD^1/}$ admits a right adjoint which sends a bipointed $\infty$-category $(X;s,t)$ to the $\infty$-category
\[
{X_{s//^\co}}\underset{X}{\times} {X_{//t}}
\]
of objects of $X$ equipped with a morphism from $s$ and a morphism to $t$.
\end{corollary}
\begin{definition}
For any bipointed $\infty$-category $(X;s,t)$ the {\em antioriented bislice} is the $\infty$-category
\[
{X_{s//^\oplax}}\underset{X}{\times} {X_{//^\lax t}}.
\]
\end{definition}
\begin{remark}
Applying the $(-)^\co$ involution to the antioriented bicone functor results in the {\em oriented bicone} functor $(-)^\bd:\infty\Cat\to\infty\Cat_{\partial\bD^1/}$, which is also a colimit preserving functor from $\infty$-categories to bipointed $\infty$-categories.
It determines the {\em oriented bislice} functor, which sends a bipointed $\infty$-category $(X;s,t)$ to the $\infty$-category
\[
{X_{s//^\lax}}\underset{X}{\times} {X_{//^\oplax t}}.
\]
\end{remark}

\begin{proposition}\label{whiskering}
For an $\infty$-category $Y$, there are bipointed whiskering maps $S(Y)\to\bD^1\vee S(Y)$ and $S(Y)\to S(Y)\vee\bD^1$.
Moreover, these maps are functorial in maps $Y'\to Y$.  
\end{proposition}
\begin{proof}
The whiskering maps are special cases of the bipointed maps
\[
S(X\times Y\times Z)\to S(X)\vee S(Y)\vee S(Z)
\]
defined in adjoint form as the composite
\begin{align*}
X\times Y\times Z&\to\Mor_{S(X)}(0,1)\times\Mor_{S(Y)}(1,2)\times\Mor_{S(Z)}(2,3)\\
&\to\Mor_{S(X)\vee S(Y)\vee S(Z)}(0,1)\times\Mor_{S(X)\vee S(Y)\vee S(Z)}(1,2)\times\Mor_{S(X)\vee S(Y)\vee S(Z)}(2,3)\\
&\to\Mor_{S(X)\vee S(Y)\vee S(Z)}(0,3).
\end{align*}
Here we have translated the names of the objects so that the above maps are induced by wedge sum inclusions and composition in $S(X)\vee S(Y)\vee S(Z)$.
The construction of these maps is functorial in maps $S(X')\to S(X)$, $S(Y')\to S(Y)$, $S(Z')\to S(Z)$ induced by maps $X'\to X$, $Y'\to Y$, $Z'\to Z$.
\end{proof}
\begin{proposition}\label{suspensiontobicone}
There is a canonical natural transformation of functors
\[
S(-)\to (-)^\abd:\infty\Cat\to\infty\Cat_{\partial\bD^1/}.
\]
\end{proposition}
\begin{proof}
Let $X$ be an $\infty$-category.
By construction of the bicone, there are inclusions $\bD^0\bar{\ast}\,X\to X^\abd$ and $X\ast\bD^0\to X^{\abd}$ compatible with the inclusion of $X$.
The units of the (anti)join and (co)slice adjunctions yield maps
\[
X\to (\bD^0\bar{\ast}\, X)_{0{//^\oplax}}\qquad\textrm{and}\qquad X\to (X\ast\bD^0)_{//^\lax 1},
\]
where $0$ and $1$ denote the first and last objects of $X^{\abd}$.
By adjunction, these maps are equivalent to the diagram
\[
\xymatrix{
\{0\}\ar[rr] & \ar@{=>}[d] & \bD^0\bar{\ast}\, X\ar[d]\\
X\ar[u]\ar[d]\ar[rr] & \ar@{=>}[d] & X^{\abd}\\
\{1\}\ar[rr] & & X\ast\bD^0\ar[u]
}
\]
in which the $2$-cells are the $1$-cells in $\Fun^{\lax}(X,X^{\abd})$ corresponding to the quotient maps
\[
\bD^1\bar{\boxtimes}\, X\to\bD^0\bar{\ast}\, X\qquad\text{and}\qquad X\boxtimes\bD^1\to X\ast\bD^0.
\]
Since the target of the top $2$-cell is equal to the source of the bottom $2$-cells, we may compose these $2$-cells to obtain a diagram
\[
\xymatrix{
& \{0\}\ar[rd]\ar@{=>}[dd] &\\
X\ar[ru]\ar[rd] & & X^{\abd}\\
& \{1\}\ar[ru] &
}
\]
which is adjoint to a bipointed map $S(X)\to X^{\abd}$.
\end{proof}

\begin{remark}
Evidently, the construction of the functor $S(X)\to X^{\abd}$ is natural in $1$-morphisms $X\to Y$.
We do not assert that the resulting natural transformation is natural in $n$-morphisms for $n>1$, although if one keeps track of the orientations it seems plausible that some such statement could be made.
\end{remark}

\begin{remark}
The functor $S(\bD^n)\to(\bD^n)^{\abd}$ of \cref{suspensiontobicone} is compatible with both source and target inclusions $\bD^m\to\bD^n$ for any $m<n$.
That is, the square
\[
\xymatrix{
S(\bD^m)\ar[r]\ar[d] & (\bD^m)^{\abd}\ar[d]\\
S(\bD^n)\ar[r] & (\bD^n)^{\abd}}
\]
commutes, where the vertical maps are included by one of the inclusions.
\end{remark}

\begin{notation}
Set ${{\abd}}^0=\bD^0$, so that $\partial{{\abd}}^0=\emptyset$, and inductively define
\[
\partial{\abd}^n=(\partial{\abd}^{n-1})^{\abd}.
\]
\end{notation}

\begin{notation}\label{attachingmap}
Suppose given a map $\partial\bD^{n}\to\partial{\abd}^{n}$.
Using \cref{suspensiontobicone}, we obtain a composite morphism
\[
\partial\bD^{n+1}=S(\partial\bD^{n})\to S(\partial{\abd}^{n})\to (\partial{\abd}^{n})^{\abd}=\partial{\abd}^{n+1}.
\]
Starting with the evident isomorphism
\[
\alpha_1:\partial\bD^1\to\partial{\abd}^1,
\]
applying this procedure results in {\em attaching maps} $\alpha_n:\partial\bD^n\to\partial{\abd}^n$, for all positive integers $n$. 
\end{notation}

\begin{definition}
The {\em antioriented orthoplex} ${\abd}^n$ is the pushout
\[
\xymatrix{
\partial\bD^{n}\ar[r]^{\alpha_n}\ar[d] & \partial{\abd}^{n}\ar[d]\\
\bD^{n}\ar[r] & {\abd}^{n}
}
\]
in which the unique $n$-cell is attached by the attaching map $\alpha_n$ constructed in \cref{attachingmap} above.
\end{definition}

\begin{remark}
Applying the $(-)^\co$ involution, this also defines the {\em oriented orthoplex} $\bd^n$ for each $n\geq 0$.
\end{remark}
\begin{remark}
We suspect that the full subcategory $\bd\subset\infty\Cat$ spanned by the oriented orthoplexes $\bd^n$ is dense, though this seems somewhat more challenging to show than for the orientals and the cubes.
\end{remark}

\appendix

\section{\mbox{Chain models of strict $\infty$-categories}}

\subsection{The Gray tensor product of Steiner $\infty$-categories}

\begin{definition}\label{invo} For every augmented directed chain complex $(\A,\B, \partial, \epsilon)$ let
\begin{enumerate}[\normalfont(1)]\setlength{\itemsep}{-2pt}
\item $(\A,\B, \partial, \epsilon)^\op$ be the  augmented directed complex $(\A,\B, \partial', \epsilon),$	where $\partial'_\n= (-1)^{\n} \partial_\n$ for $\n \geq 0.$

\item $(\A,\B, \partial, \epsilon)^\co$ be the  augmented directed complex $(\A,\B, \partial'', \epsilon),$	where $\partial''_\n= (-1)^{\n+1}\partial_\n$ for $\n \geq 0.$
\end{enumerate}	

\end{definition}

\begin{remark}By construction of $\nu$ for every augmented directed complex $\A$ there are canonical isomorphisms
$$ \nu(\A^\op) \cong \nu(\A)^\op, \ \nu(\A^\co) \cong \nu(\A)^\co.$$

By construction of $\lambda$ for every strict $\infty$-category $\mC$ there are canonical isomorphisms
$$ \lambda(\mC^\op) \cong \lambda(\mC)^\op, \ \lambda(\mC^\co) \cong \lambda(\mC)^\co.$$

\end{remark}

\begin{lemma}\label{lemmcub} Let $\A,\B$ be augmented directed complexes.
\begin{enumerate}[\normalfont(1)]\setlength{\itemsep}{-2pt}
\item There is a canonical isomorphism of augmented directed complexes $$\A^\op \ot \B^\op \cong (\B\ot\A)^\op$$
that sends $\X \ot \Y$ to $\Y \ot \X$ for $\X \in \A_\bi, \Y \in \B_\bj$ and $\bi,\bj \geq 0.$
The latter isomorphism makes the functor $(-)^\op: \mathrm{ADC}^\rev \to \mathrm{ADC}$ to a monoidal involution.

\item There is a canonical isomorphism of augmented directed complexes $$ \A^\co \ot \B^\co \cong (\B\ot\A)^\co $$
that sends $\X \ot \Y$ to $\Y \ot \X$ for $\X \in \A_\bi, \Y \in \B_\bj$ and $\bi,\bj \geq 0.$
The latter isomorphism makes the functor $(-)^\co: \mathrm{ADC}^\rev \to \mathrm{ADC}$ to a monoidal involution.

\end{enumerate}		
\end{lemma}

\begin{proof}

Let $\bi,\bj,\bk \geq 0$ and $\bi+\bj=\bk$ and $\X \in \A_\bi, \Y \in \B_\bj.$

(1): $$ \partial^{\A^\op \ot \B^\op}_{\bk}(\X\ot \Y)  = \partial_\bi^{\A^\op}(\X) \ot \Y +
(-1)^{\bi} \X \ot \partial_\bj^{\B^\op}(\Y) = $$$$(-1)^{\bi} \partial_\bi^{\A}(\X) \ot \Y +
(-1)^{\bi} \X \ot (-1)^{\bj} \partial_\bj^{\B}(\Y)= (-1)^{\bi} \partial_\bi^{\A}(\X) \ot \Y + (-1)^{\bk} \X \ot \partial_\bj^{\B}(\Y),$$
$$\partial^{(\B \ot \A)^\op}_{\bk}(\Y\ot \X)= (-1)^{\bk} \partial^{\B \ot \A}_{\bk}(\Y\ot \X)= (-1)^{\bk}(\partial^{\B}_{\bj}(\Y) \ot \X + (-1)^{\bj}\Y \ot \partial^\A_\bi(\X))= $$
$$(-1)^{\bk}\partial^{\B}_{\bj}(\Y) \ot \X + (-1)^{\bi+2\bj}\Y \ot \partial^\A_\bi(\X)).$$	
This isomorphism makes the functor $(-)^\op: \mathrm{ADC}^\rev \to \mathrm{ADC}$ into a monoidal involution since for every $\A,\B,\C \in \mathrm{ADC}$ the following diagrams commute:
$$\begin{xy}
\xymatrix{
(\A^\op \ot \B^\op) \ot \C^\op \ar[d] \ar[r] & (\A \ot \B)^\op \ot \C^\op \ar[r] & ((\A \ot \B) \ot \C)^\op \ar[d]
\\ 
\A^\op \ot (\B^\op \ot \C^\op) \ar[r] &  \A^\op \ot (\B \ot \C)^\op  \ar[r] & (\A \ot (\B \ot \C))^\op,
}
\end{xy}
$$
$$\begin{xy}
\xymatrix{
\A^\op \ot \bZ \ar[d]^\cong \ar[r]^\cong & \A^\op \ar[d]^\cong
\\ 
\A^\op \ot \bZ^\op \ar[r]^\cong &  (\A \ot \bZ)^\op,
}
\end{xy}
\begin{xy}
\xymatrix{
\bZ \ot \A^\op \ar[d]^\cong \ar[r]^\cong & \A^\op \ar[d]^\cong
\\ 
\bZ^\op \ot \A^\op \ar[r]^\cong &  (\bZ \ot \A)^\op.
}
\end{xy}
$$

(2): $$ \partial^{\A^\co \ot \B^\co}_{\bk}(\X\ot \Y)  = \partial_\bi^{\A^\co}(\X) \ot \Y +
(-1)^{\bi} \X \ot \partial_\bj^{\B^\co}(\Y) = $$$$(-1)^{\bi+1} \partial_\bi^{\A}(\X) \ot \Y +
(-1)^{\bi} \X \ot (-1)^{\bj+1} \partial_\bj^{\B}(\Y)= (-1)^{\bi+1} \partial_\bi^{\A}(\X) \ot \Y + (-1)^{\bk+1} \X \ot \partial_\bj^{\B}(\Y),$$
$$\partial^{(\B \ot \A)^\co}_{\bk}(\Y\ot \X)= (-1)^{\bk+1} \partial^{\B \ot \A}_{\bk}(\Y\ot \X)= (-1)^{\bk+1}(\partial^{\B}_{\bj}(\Y) \ot \X + (-1)^{\bj}\Y \ot \partial^\A_\bi(\X))= $$
$$(-1)^{\bk+1}\partial^{\B}_{\bj}(\Y) \ot \X + (-1)^{\bi+1+2\bj}\Y \ot \partial^\A_\bi(\X)).$$	
This isomorphism makes the functor $(-)^\co: \mathrm{ADC}^\rev \to \mathrm{ADC}$ into a monoidal involution since for every $\A,\B,\C \in \mathrm{ADC}$ the corresponding diagrams from above, with $(-)^\op$ replaced by $(-)^\co$, commute.
\end{proof}

\begin{corollary}\label{duac0}

There are canonical monoidal involutions
$$(-)^\op, (-)^\co : (\mathrm{ADC}, \otimes) \simeq (\mathrm{ADC}, \otimes)^\rev. $$

\end{corollary}

\begin{corollary}\label{selfdual} Let $\n \geq0$.
There are canonical isomorphisms $$(\cube^\n)^\op \cong \cube^\n \cong (\cube^\n)^\co.$$	

\end{corollary}

\begin{corollary}\label{duac}
The monoidal involutions of \cref{duac0} restrict to monoidal involutions
$$(-)^\op, (-)^\co : (\cube, \boxtimes) \simeq (\cube, \boxtimes)^\rev. $$
\end{corollary}

\subsection{The join of Steiner $\infty$-categories}

\begin{remark}\label{suspmod3} Let $\A$ be an augmented directed complex.
There is an identity $$ \A \diamond \bZ = \A \oplus \A[1] \oplus \bZ $$
as graded abelian groups respecting the graded submonoids.
For every $\ell > 1$ the $\ell$-th differential is $$ \partial^\A_\ell+ ((-1)^{\ell-1}\id, \partial^\A_{\ell-1}): (\A \diamond \bZ)_\ell = \A_\ell \oplus \A_{\ell-1} \to \A_{\ell-1} \oplus \A_{\ell-2}.$$
The first differential $$ (\A \diamond \bZ)_1 = \A_0 \oplus \A_1 \to (\A \diamond \bZ)_0=\A_0 \oplus \bZ $$ is the map $(\id, -\epsilon^\A)+ \partial_1^\A.$
The augmentation is $$\epsilon^\A + \id_\bZ : (\A \diamond \bZ)_0 =\A_0 \oplus \bZ \to \bZ. $$
	
\end{remark}

\begin{lemma}\label{lemmcub2} Let $\A,\B$ be augmented directed complexes.
There is a canonical isomorphism of augmented directed complexes $$\A^\op \diamond \B^\op \cong (\B\diamond\A)^\op.$$
The latter isomorphism makes the functor $(-)^\op: \mathrm{ADC}^\rev \to \mathrm{ADC}$ to a monoidal involution.
	
\end{lemma}

\begin{proof}

There is a canonical isomorphism
$$\A^\op \diamond \B^\op = \A^\op \ot \lambda(\bD^1) \ot \B^\op \coprod_{\A^\op \ot \lambda(\partial\bD^1)\ot\B^\op} \A^\op \oplus \B^\op \cong $$ 
$$ \A^\op \ot \lambda(\bD^1)^\op \ot \B^\op \coprod_{\A^\op \ot \lambda(\partial\bD^1)^\op \ot \B^\op} \B^\op \oplus \A^\op \cong (B\diamond A)^\op.$$

The latter makes the functor $(-)^\op: \mathrm{ADC}^\rev \to \mathrm{ADC}$ to a monoidal involution since for every $\A,\B,\C \in \mathrm{ADC}$ the following diagrams trivially commute:
$$\begin{xy}
\xymatrix{
(\A^\op \diamond \B^\op) \diamond \C^\op \ar[d] \ar[r] & (\A \diamond \B)^\op \diamond \C^\op \ar[r] & ((\A \diamond \B) \diamond \C)^\op \ar[d]
\\ 
\A^\op \diamond (\B^\op \diamond \C^\op) \ar[r] &  \A^\op \diamond (\B \diamond \C)^\op  \ar[r] & (\A \diamond (\B \diamond \C))^\op,
}
\end{xy}
$$
$$\begin{xy}
\xymatrix{
\A^\op \diamond \bZ \ar[d]^\cong \ar[r]^\cong & \A^\op \ar[d]^\cong
\\ 
\A^\op \diamond \bZ^\op \ar[r]^\cong &  (\A \diamond \bZ)^\op,
}
\end{xy}
\begin{xy}
\xymatrix{
\bZ \diamond \A^\op \ar[d]^\cong \ar[r]^\cong & \A^\op \ar[d]^\cong
\\ 
\bZ^\op \diamond \A^\op \ar[r]^\cong &  (\bZ \diamond \A)^\op.
}
\end{xy}
$$
\end{proof}

We obtain the following corollaries:

\begin{corollary}
Let $n \geq 0.$
There is a canonical equivalence of $\infty$-categories
$$(\bDelta^n)^\op \simeq \bDelta^n.$$

\end{corollary}

\begin{corollary}\label{duacjoin}

There is a canonical monoidal involution
$$(-)^\op: (\bDelta, \star) \simeq (\bDelta, \star)^\rev.$$

\end{corollary}

\begin{lemma}\label{cell}\label{cell2}
Let $n \geq 0$ and $\mO$ a $n$-dimensional Steiner $\infty$-category
that has a unique non-invertible $n$-morphism.
The following canonical commutative squares are pushout squares in $\infty\Cat$:
\[
\xymatrix{
\iota_{n-1}(\mO) \ar[d]\ar[r] & \iota_{n-1}(\mO) \star\bD^0\ar[d]\\
\mO \ar[r] & \iota_{n}(\mO \star \bD^0),
}\qquad
\xymatrix{
\iota_{n-1}(\mO) \coprod \iota_{n-1}(\mO) \ar[r]\ar[d] & \bD^1 \boxtimes \iota_{n-1}(\mO) \ar[d]\\
\mO \coprod \mO \ar[r] & \iota_n(\bD^1 \boxtimes \mO).}
\]

\end{lemma}

\begin{proof}
All squares consist of Steiner $\infty$-categories and inclusions preserving atomic generators.
Moreover by \cref{orientdec} the left vertical functors of all squares are in the saturated class generated by the inclusion $ \partial\bD^n \subset \bD^n$.
Hence by \cref{pasting1} and \cref{pasting2} it suffices to see that
all squares are pushout squares in $\infty\Cat^\Steiner$.

The left commutative square corresponds to the following pushout square of Steiner complexes:
\[
\xymatrix{
\lambda(\iota_{n-1}(\mO)) \ar[d]\ar[r] & \bZ \oplus \lambda(\iota_{n-1}(\mO))[1] \oplus \lambda(\iota_{n-1}(\mO)) \ar[d] \\ \lambda(\mO) \ar[r] & \bZ \oplus \lambda(\iota_{n-1}(\mO))[1] \oplus \lambda(\mO).
}
\]

The right commutative square corresponds to the following pushout square of Steiner complexes:
\begin{equation}\label{o5}
\xymatrix{
\lambda(\iota_{n-1}(\mO)) \oplus \lambda(\iota_{n-1}(\mO)) \ar[r]\ar[d] & \lambda(\iota_{n-1}(\mO)) \oplus \lambda(\iota_{n-1}(\mO)) \oplus \lambda(\iota_{n-1}(\mO)) \ar[d]\\
\lambda(\mO) \oplus \lambda(\mO) \ar[r] & \lambda(\mO) \oplus \lambda(\iota_{n-1}(\mO)) \oplus \lambda(\mO).}
\end{equation}
\end{proof}

\subsection{The suspension of Steiner $\infty$-categories}

\begin{definition}Let $A$ be an augmented directed complex.
\begin{enumerate}[\normalfont(1)]\setlength{\itemsep}{-2pt}
\item The suspension complex of $A$ is $$S(A):=
A \ot\lambda(\bD^1) \coprod_{A\ot\lambda(\partial\bD^1)}\lambda(\partial\bD^1).$$
		
\item The co-suspension complex of $A$ is $$\bar{S}(A):=S(A^\co)^\co.$$
		
\end{enumerate}
\end{definition}

\cref{lemmcub} and \cref{duac} imply the following corollary:

\begin{remark}\label{suspmod} Let $\A$ be an augmented directed complex.
There is an identity $$ S(\A)= \A[1]\oplus \bZ \oplus \bZ $$
as graded abelian groups respecting the graded submonoids.
For every $\ell > 1$ the differentials are $$ S(\A)_\ell = \A_{\ell-1}\to S(\A)_{\ell-1}= \A_{\ell-2}, \X \mapsto \partial^\A(\X). $$
The first differential is the map $$ (\epsilon, -\epsilon): S(\A)_1 = \A_0 \to S(\A)_0 = \mathbb{Z}\oplus \mathbb{Z}.$$
The augmentation is the sum $$+: S(\A)_0=\mathbb{Z} \oplus \mathbb{Z}.$$

\end{remark}

\begin{remark}\label{isomo} Let $A$ be an augmented directed complex.
There is an identity $$ S(A)= A[1]\oplus \bZ \oplus \bZ= \bar{S}(A)$$
as graded abelian groups respecting the graded submonoids.
For every $\ell > 1$ the differentials are $$ S(A)_\ell = A_{\ell-1}\to S(A)_{\ell-1}= A_{\ell-2}, X \mapsto \partial^A(X), \ \bar{S}(A)_\ell = A_{\ell-1}\to \bar{S}(A)_{\ell-1}= A_{\ell-2}, X \mapsto -\partial^A(X).$$
The differentials $$ S(A)_1 = A_0 \to S(A)_0 = \mathbb{Z}\oplus \mathbb{Z}, \ \bar{S}(A)_1 = A_0 \to \bar{S}(A)_0 = \mathbb{Z}\oplus \mathbb{Z}$$ are the map  $(\epsilon, -\epsilon).$
The augmentation is the sum $+: \mathbb{Z} \oplus \mathbb{Z}\to  \mathbb{Z}.$
	
In particular, we find that the identity of $A[1]\oplus \bZ \oplus \bZ $ determines an isomorphism \begin{equation}\label{sussp}
S(A)\cong \bar{S}(A^{\co\op})\end{equation}
of augmented directed complexes.
	
\end{remark}

\cref{lemmcub} and \cref{duac} imply the following corollary:

\begin{corollary}
Let $A$ be an augmented directed complex.
There is a canonical isomorphism $$\bar{S}(A)\cong\lambda(\bD^1) \ot A \coprod_{\lambda(\partial\bD^1)\ot A}\lambda(\partial\bD^1).$$	
	
\end{corollary}

\begin{remark}\label{isomo1}
By definition of $\lambda$ for every strict $\infty$-category $\mC$
there is a canonical isomorphism 
$S(\lambda(\mC)) \cong \lambda(S(\mC))$ of augmented directed complexes.
Thus there is a canonical isomorphism 
$$\bar{S}(\lambda(\mC)) \cong S(\lambda(\mC)^\co)^\co \cong 
S(\lambda(\mC^\co))^\co \cong \lambda(S(\mC^\co))^\co \cong 
\lambda(S(\mC^\co)^\co) \cong \lambda(\bar{S}(\mC))$$ of augmented directed complexes.
\end{remark}

\begin{remark}By \cref{isomo1} the $\n$-disk $\bD^\n = S^\n(*)$ is a Steiner $\infty$-category and there is a canonical isomorphism of augmented directed complexes $$ \lambda(\bD^\n)= \lambda(S^\n(*))\cong S^\n(\bZ).$$
Thus $ \lambda(\bD^n) \cong S^\n(\bZ)$ is a $\n$-dimensional Steiner complex,
which arises by applying $S$ iterately:
$$ 	(... \to 0 \to ... \to 0 \to \bZ \xrightarrow{(\id,-\id)} \bZ \oplus \bZ \xrightarrow{(\id,-\id)} \bZ \oplus \bZ \xrightarrow{(\id,-\id)} ... \xrightarrow{(\id,-\id)} \bZ \oplus \bZ),$$
which is free with one generator in degree $\n$ and two generators in any lower degree, and the augmentation is the sum $+: \bZ \oplus \bZ \to \bZ.$ 

\end{remark}

\begin{example}\label{example_theta}
Every object of $\Theta$ is a Steiner $\infty$-category: the functors $ \bD^{j_\ell} \to \bD^{i_\ell},\bD^{j_\ell} \to \bD^{i_{\ell-1}}$
are the images under $\nu$ of augmented directed chain maps
$ \lambda(\bD^{j_\ell}) \to  \lambda(\bD^{i_\ell}), \lambda(\bD^{j_\ell}) \to  \lambda(\bD^{i_{\ell-1}})$
and the pushout 
$$ \lambda(\bD^{i_0}) \coprod_{ \lambda(\bD^{j_1})}  \lambda(\bD^{i_1}) \coprod_{\lambda(\bD^{j_2})} ... \coprod_{ \lambda(\bD^{j_n})}  \lambda(\bD^{i_n}) $$
in $\mathrm{ADC}$ is a Steiner complex and is preserved by $\nu: \mathrm{ADC} \to \infty\Cat$ by \cref{pasting1}, \cref{pasting2}.
\end{example}

\bibliographystyle{plain}
\bibliography{mainbib}
\end{document}